\documentclass[12pt]{article}
\usepackage{geometry}                
\usepackage{graphicx}
\usepackage{amssymb}
\usepackage{amsmath}
\usepackage{mathtools}
\usepackage{mathrsfs}
\usepackage{xcolor}
\usepackage{amsthm}
\usepackage{hyperref}
\usepackage{enumitem}
\newcommand*{\NN}{\mathbb{N}}
\newcommand*{\ZZ}{\mathbb{Z}}

\newcommand*{\RR}{\mathbb{R}}
\newcommand*{\CC}{\mathbb{C}}

\newcommand*{\PP}{\mathbb{P}}
\DeclareMathOperator*{\EE}{\mathbb{E}}

\newcommand*{\calO}{\mathcal{O}}
\newcommand*{\calP}{\mathcal{P}}
\newcommand*{\calU}{\mathcal{U}}
\newcommand*{\calV}{\mathcal{V}}

\renewcommand\d{\delta}

\renewcommand\k{\kappa}

\newcommand\s{\sigma}

\newcommand{\G}{\Gamma}

\newcommand{\Si}{\Sigma}
\newcommand{\eps}{\epsilon}

\newcommand*{\A}{\mathtt{A}}
\newcommand*{\ta}{\mathtt{a}}
\newcommand*{\B}{\mathtt{B}}

\newcommand*{\tK}{\mathtt{K}}
\newcommand*{\tL}{\mathtt{L}}

\newcommand{\bta}{\mathbf{a}}
\newcommand{\btb}{\mathbf{b}}

\newcommand*{\1}{\mathbf{1}}

\newcommand*{\mb}[1]{\mathbf{#1}}

\newcommand*{\st}{\,:\,}

\newcommand*{\ball}[3][\relax]{\mathrm{B}^{#1}(#2, #3)} 

\newcommand*{\cluster}{\mathrm{cl}}

\def\cc{{\curvearrowright}}

\newtheorem{lemma}{Lemma}[section]
\newtheorem{cor}[lemma]{Corollary}
\newtheorem{prop}[lemma]{Proposition}
\newtheorem{theorem}[lemma]{Theorem}

\newtheorem{mainthm}{Theorem}

\theoremstyle{definition}

\newtheorem*{remark}{Remark}
\newtheorem{defn}[lemma]{Definition}

\newcommand*{\Chec}{\mathtt{Check}}
\newcommand*{\Ver}{\mathtt{Vert}}

\newcommand*{\codegrowth}{G_{\mathrm{cw}}}

\newcommand*{\kik}{\mathrm{H_K}}
\newcommand*{\hann}{\mathrm{h_{ann}}}
\newcommand*{\unif}{\mathrm{unif}}

\DeclareMathOperator{\Prob}{Prob}

\DeclareMathOperator{\Hom}{Hom}
\DeclareMathOperator{\Aut}{Aut}
\DeclareMathOperator{\Sym}{Sym}

\DeclareMathOperator{\TV}{TV}

\DeclareMathOperator{\shent}{H}
\DeclareMathOperator{\TC}{TC}
\DeclareMathOperator{\info}{I}
\DeclareMathOperator{\h}{h}

\newcommand*{\Rok}{\mathrm{Rok}}

\newcommand{\Loc}[3][\relax]{\mathrm{Loc}_{#1}(#2, #3)}

\newcommand{\dee}{\textrm{d}}

\DeclarePairedDelimiter{\abs}{\lvert}{\rvert}

\DeclarePairedDelimiter{\norm}{\|}{\|}
\newcommand{\nnorm}[1]{{\left\vert\kern-0.25ex\left\vert\kern-0.25ex\left\vert #1 
    \right\vert\kern-0.25ex\right\vert\kern-0.25ex\right\vert}}
\DeclarePairedDelimiterX{\inprod}[2]{\langle}{\rangle}{#1,\ #2}

\newcommand{\bigslant}[2]{{\left.\raisebox{.2em}{$#1$}\middle/\raisebox{-.2em}{$#2$}\right.}}

\begin{document}

\title{Algebraic dynamical systems from LDPC codes satisfy a strong negation of the weak Pinsker property}

\author{Tim Austin\footnote{supported in part by NSF grant DMS-1855694}, Lewis Bowen\footnote{supported in part by NSF grant DMS-2154680}\ \ and Christopher Shriver\footnote{supported in part by NSF grant DMS-1937215}}

\maketitle

\begin{abstract}
We construct an explicit algebraic example of a subshift of finite type over a group $\Gamma$ with an invariant Markov measure which has completely positive sofic entropy (with respect to `most' sofic approximations) and yet does not have a direct Bernoulli factor, because its model spaces shatter into exponentially many clusters of sub-exponential size. The example and its analysis are related to random low-density parity-check (LDPC) codes. 
\end{abstract}


\tableofcontents

\section{Introduction}\label{S:intro}

This paper constructs explicit dynamical systems with unusual properties related to recent work on the weak Pinsker property and shattering. The construction is explained next; the background, motivation and precise statements are developed afterwards.

Fix natural numbers $d,k$ and let $\Gamma=\Gamma_{d,k}$ be the $d$-fold free product of order-$k$ cyclic groups:
\[\Gamma := \langle s_1,\dots,s_d:\ s_1^k = \dots = s_d^k = e\rangle = \underbrace{\ZZ_k\ast \cdots \ast \ZZ_k}_{d},\]
where $\ZZ_k$ means $\ZZ/k\ZZ$. The set of all functions $x:\Gamma \to \ZZ_2$ is denoted $\ZZ_2^\Gamma$. This is a compact Abelian group under pointwise addition with the pointwise convergence topology. Let $X \le \ZZ_2^\Gamma$ be the closed subgroup defined by
\[X = \left\{x \in \ZZ_2^\Gamma:\ \sum_{j=0}^{k-1}x_{gs_i^j} = 0 \quad \forall g\in \Gamma,\ i = 1,\dots,d\right\},\]
and let $\mu = m_X$ be the Haar probability measure on $X$.

For $g\in \Gamma$, let $T^g: \ZZ_2^\Gamma \to \ZZ_2^\Gamma$ be the continuous group automorphism given 
by permuting indices on the left:
\begin{equation}\label{eq:left-action}
T^g((x_h)_{h \in \Gamma}) = (x_{g^{-1}h})_{h\in \Gamma}.
\end{equation}
The subgroup $X$ is invariant under this action, and hence so is its Haar measure.

This state space is easily visualized in terms of the Cayley graph of $\Gamma$ with its generators $s_1,\dots,s_k$.  Through each group element $g$, each $s_i$ generates a $k$-cycle.  So each vertex of the Cayley graph lies in $d$ of these $k$-cycles, and there are no other relations in the group, so these $k$-cycles are attached together into a hyper-tree.  With this picture in mind, a member of $X$ is simply an assignment of zeros and ones to the vertices of the Cayley graph such that the sum around every $k$-cycle is even.  For this reason, and by analogy with similar constructions in coding theory, we call $X$ a \textbf{parity check subshift}.  Indeed, certain random finite parity check codes play a crucial auxiliary role later in the paper: see Section~\ref{sec:LDPC}.

Informally stated, our main results are these:
\begin{itemize}
\item If $k > d \ge 3$ then the sofic entropy of the dynamical system $(X,m_X,T)$ is $(1-d/k)\log(2)$.
\item Every nontrivial factor of $(X,m_X,T)$ has positive sofic entropy and therefore positive Rokhlin entropy (this property is called `completely positive entropy' or `CPE').  In fact, we prove the stronger assertion that the outer Pinsker factor of $(X,m_X,T)$ is trivial.
\item The system $(X,m_X,T)$ is not isomorphic to a direct product of a nontrivial Bernoulli shift with another system. Combined with the previous conclusion, this is a strong negation of the weak Pinsker property.
\item The system $(X,m_X,T)$ is not weakly contained in a Bernoulli shift.  It is one of the first examples that has completely positive entropy and also this property.
\end{itemize}

Next we introduce background needed to state our main results precisely.

\subsection{Background: classical entropy theory}


Kolmogorov introduced entropy theory into dynamics for the purpose of distinguishing Bernoulli shifts up to measure conjugacy. Given a standard probability space $(\tK,\k)$, the {\bf Bernoulli shift over a countable group $\G$ with base space $(\tK,\k)$} consists of the probability space $(\tK^\G, \k^\G)$ together with the action of $\G$ by permuting indices as in~\ref{eq:left-action}.  A sample of $(\tK^\G,\k^\G)$ is a random $\tK$-valued configuration $(x_g)_{g \in \G}$ whose coordinates are i.i.d. with law $\k$. 

Suppose we are given a standard probability space $(X,\mu)$ (where we have left the sigma-algebra out of the notation for simplicity). Let $\Aut(X,\mu)$ denote the group of all measure-preserving automorphisms of $(X,\mu)$. A {\bf pmp} (probability-measure-preserving) action of $\G$ is a homomorphism $T:\G \to \Aut(X,\mu)$. The triple $(X,\mu,T)$ is a {\bf $\G$-system}. We also refer to it as a {\bf system}  or {\bf action} if $\G$ is understood. 

If we are given two $\G$-systems $(X_i,\mu_i,T_i)$ then a measurable map $\Phi:X_1 \to X_2$ is a {\bf factor map} if it is a.e. $\G$-equivariant (this means $\Phi(T_1^g x)=T_2^g \Phi(x)$ for all $g\in \G$ and $\mu_1$-a.e. $x\in X_1$) and the pushforward measure satisfies $\Phi_*\mu_1=\mu_2$. More precisely, we allow that $\Phi$ be defined only on a subset of full measure. If $\Phi$ is invertible (after ignoring a null set) then it is a {\bf measure-conjugacy} or {\bf isomorphism}. 

If $\Gamma = \ZZ$ then an action of the integers is given by a single transformation $T \in \Aut(X,\mu)$. Thus it makes sense to consider whether two transformations are measurably conjugate.

A problem attributed to von Neumann asks whether there could be two Bernoulli shifts over the group of integers which are not measurably conjugate. To answer this, Kolmogorov defined the entropy rate of a dynamical system in the special case in which $\G=\ZZ$ \cite{kolmogorov-1958, kolmogorov-1959}. He proved entropy is invariant under measure-conjugacy and computed entropy rates for Bernoulli shifts, thereby answering the problem in the affirmative. In fact, the entropy rate of a Bernoulli shift action is the same as the Shannon entropy of the base space. When the base space is $(\tK,\k)$ and $\tK$ is countable, its {\bf Shannon entropy} is 
$$H(\k) = - \sum_{k \in \tK} \k(\{k\}) \log (\k(\{k\})).$$
If $\k$ is not supported on a countable set, then its Shannon entropy is defined to be $+\infty$. 

Kolmogorov's theory extends fairly directly to the case when $\G$ is amenable. The first published work on entropy theory for general amenable groups is due to Kieffer \cite{kieffer-1975a}. 

Since Kolmogorov's pioneering work, entropy and Bernoulli shifts have played a central role in classifying dynamical systems. For example, Sinai proved that if an ergodic action of $\ZZ$ has positive entropy then it factors onto a Bernoulli shift of the same entropy \cite{sinai-weak}. Because entropy cannot increase under a factor map, this shows that Bernoulli factors witness entropy. Inspired by Sinai's theorem, Ornstein proved that Bernoulli shifts over the integers are isomorphic if and only if they have the same entropy \cite{ornstein-1970a, ornstein-1970b}.   These results were extended to the case of amenable acting groups in~\cite{OW87}.

Shannon entropy is easily seen to be additive under direct products, and this property is inherited by Kolmogorov's entropy rate. Naively, one might guess that any ergodic system is isomorphic to a direct product of a Bernoulli shift with a zero entropy system. This turns out to be false; counterexamples to weaker claims appear in \cite{MR0399416, MR0330416, MR0316682}. If it were true for a system which was not itself isomorphic to a Bernoulli shift, then additivity implies that the system would have a nontrivial (direct) factor with zero entropy. A system is said to have {\bf completely positive entropy} (CPE) if every nontrivial factor has positive entropy. Here a ``trivial'' factor is a measure-preserving system where the measure is a delta mass at a single point. This system is a factor of every other system, and it is easy to see that its entropy is zero. The paper \cite{MR0382598} shows that for any positive number $h>0$ there exist uncountably many pairwise non-isomorphic transformations which are CPE and have entropy $h$. 

A factor map $\pi:(X_1,\mu_1,T_1) \to (X_2,\mu_2,T_2)$ is said to be {\bf direct} or {\bf split} if there is another factor map $\xi:(X_1,\mu_1,T_1) \to (X_3,\mu_3,T_3)$ so that the pair $(\pi,\xi)$ together forms an isomorphism
\[(X_1,\mu_1,T_1) \to (X_2 \times X_3,\mu_2 \times \mu_3,T_2\times T_3).\]
Note that the measure on the right-hand side is required to be the product, so in particular the factor maps $\pi$ and $\xi$ must generate independent sigma-subalgebras of subsets of $X_1$. While Sinai's factor theorem shows the existence of Bernoulli factors, it does not say anything about the existence of {\em direct} Bernoulli factors. 

In the 1970s, Thouvenot defined a system to have the weak Pinsker property (WPP) if for every $\epsilon>0$ it is isomorphic to a direct product of a Bernoulli shift with a system of entropy less than $\epsilon$ \cite{MR0453982}. In other words, a system has the WPP if its entropy is witnessed by {\em direct} Bernoulli factors. Thouvenot asked whether every ergodic transformation has the WPP. The first author recently proved that this is indeed the case \cite{MR3905465}. Moreover, the statement holds whenever the acting group $\G$ is amenable. 

\subsection{Background: sofic entropy theory}

The second author constructed a system without the WPP in the special case when the group $\G$ is a free group of sufficiently high rank \cite{bowen2022}. To explain, we need to pause for a moment to discuss entropy theory when the acting group is not amenable.

An example due to Ornstein and Weiss in \cite{OW87} suggested it might not be possible to extend entropy theory to non-amenable groups. However, this changed with the introduction of sofic entropy theory \cite{bowen-jams-2010}. The new theory applies to all sofic groups, which is a class of groups containing amenable and linear groups, for example. It is unknown whether all countable groups are sofic. Sofic entropy theory is reviewed in \S \ref{S:prelim-entropy}.

A sofic approximation to a group $\G$ is a sequence $\Si$  of partial actions on finite sets which approximates the action of the group on itself by left-translations. To be precise, $\Sigma=(\sigma_n)_{n\in \NN}$ where $\sigma_n: \Gamma \to \Sym(V_n)$, $V_n$ are finite sets, $\Sym(V_n)$ is the symmetric group on $V_n$ and the sequence is required to satisfy for all $g, h, f \in \Gamma$ such that $f$ is not the identity,
\begin{eqnarray*}
1 &=& \lim_{n\to\infty} |V_n|^{-1} |\{ v\in V_n:~ \s_n(gh)v = \s_n(g)\s_n(h)v\}|\\
0&=& \lim_{n\to\infty} |V_n|^{-1} |\{ v\in V_n:~ \s_n(f)v = v\}|.
\end{eqnarray*}

A group is called {\bf sofic} if it admits a sofic approximation. 
The sofic entropy of a system $(X,\mu,T)$ (defined in Section~\ref{S:prelim-entropy}) depends a priori on a choice of sofic approximation, although for many actions where it has been computed, it has been shown not to. 

Many classical results extend to the sofic setting. For example, the sofic entropy of a Bernoulli shift action is equal to the Shannon entropy of the base. Sofic entropy is a measure-conjugacy invariant and so two Bernoulli shifts with different sofic entropy are not isomorphic. In recent work, Seward completed the converse direction: for any countable group $\G$, if two Bernoulli shifts over $\G$ have the same {\em base space Shannon entropy} then they are measurably conjugate \cite{stepin-1975, bowen-ornstein-2012, seward2022}. This converse does not depend on sofic entropy, which might not even be defined.

In a series of works generalizing Krieger's Theorem \cite{seward-kreiger-1, MR3959054}, Seward introduced Rokhlin entropy. To define it, suppose we are given an action $T:\G \to \Aut(X,\mu)$ and a countable measurable partition $\calP$ of $X$. We say the partition is {\bf generating} if the smallest sigma-algebra containing it which is also $T(\G)$-invariant is the sigma-algebra of all measurable sets (up to sets of measure zero). Then Rokhlin entropy is defined to be the infimum of $H_\mu(\calP)$ over all generating partitions, where 
$$H_\mu(\calP) = - \sum_{P \in \calP} \mu(P)\log(\mu(P))$$
is the {\bf Shannon entropy} of $(\calP,\mu)$.

It is immediate that Rokhlin entropy is a measure-conjugacy invariant. Moreover, it upper bounds sofic entropy. In fact, it is unknown whether Rokhlin entropy equals sofic entropy whenever the latter is not minus infinity (which can happen). On the other hand, the only known method for computing a lower bound to Rokhlin entropy uses sofic entropy. For example, it is unknown how to compute the Rokhlin entropy of Bernoulli shift actions, except in the case when $\G$ is assumed to be sofic.

In a different paper \cite{MR4066472} Seward generalized Sinai's factor theorem: every ergodic system with positive Rokhlin entropy factors onto a Bernoulli shift with the same entropy. 

However, other structural results about classical Kolmogorov--Sinai entropy break down outside the world of amenable groups. For example, Ornstein and Weiss' example shows that sofic entropy can {\em increase} under a factor map. In fact recent work of the second author shows that if $\G$ is an arbitrary non-amenable group then every Bernoulli shift over $\G$ factors onto every Bernoulli shift over $\G$ \cite{MR3962876}. For example, Bernoulli shifts of small entropy factor onto Bernoulli shifts of infinite entropy. 

Although Bernoulli shifts themselves have been classified, there is no known substitute for broader `Ornstein theory', which provides necessary and sufficient conditions for general ergodic processes to be isomorphic to Bernoulli shifts.  Moreover, some specific counterexamples show that the story must change substantially for some non-amenable groups.  For example, when $G$ has property (T), Popa and Sasyk~\cite{popa-sasyk} have given simple examples of factors of Bernoulli shifts that are not isomorphic to Bernoulli shifts.

It is also known that the weak Pinsker property does not hold for all non-amenable groups and for the main non-amenable notions of entropy.  The first counterexample appeared in~\cite{bowen2022}. While that counterexample does not have the WPP, it might still admit some direct Bernoulli factors. In other words, the system might be measurably conjugate to the direct product of a Bernoulli shift with another system, but one cannot choose the other system to have entropy less than $\eps$ if $\eps>0$ is chosen low enough.  In this work we present a new counterexample which does not admit \emph{any} nontrivial Bernoulli factors. Here, a Bernoulli shift $(\tK^\Gamma,\kappa^\Gamma,T)$ is said to be {\bf trivial} if $\kappa$ is supported on a single point in $\tK$. So a trivial Bernoulli shift is measurably conjugate to the trivial system which consists of $\Gamma$ acting on a single point.


\subsection{Main results}

Recall from the introduction that $\G$ is the $d$-fold free power of $\ZZ_k$, $X \subset (\ZZ_2)^\G$ is a certain closed ``parity check'' subgroup and $T: \G \to \Aut(X,m_X)$ is the canonical shift action by automorphisms.

\begin{mainthm}\label{mainthm1}
Let $k > d \ge 3$.  Then there exists a sofic approximation $\Sigma =(\sigma_n)_n$ to $\Gamma$ such that
\[ \h_\Sigma(X,m_X,T) = (1-d/k)\log 2\]
and along which the outer Pinsker factor of $(X,m_X,T)$ is trivial.  In particular, $(X,m_X,T)$ has completely positive sofic entropy along $\Sigma$.
\end{mainthm}

Since sofic entropy is always bounded above by Rokhlin entropy, the last part of this theorem has the following immediate corollary:

\begin{cor}
The outer Rokhlin Pinsker factor of $(X,m_X,T)$ is trivial, so it has completely positive Rokhlin entropy.
\end{cor}

\begin{mainthm}\label{mainthm2}
    If $k > d \ge 3$ then the system $(X,m_X,T)$ has no non-trivial direct Bernoulli factors.
\end{mainthm}

\begin{remark}
Theorems \ref{mainthm1} and \ref{mainthm2} should hold in the more general setting in which $\Gamma = \Gamma_1 \ast \cdots \ast \Gamma_d$ where each $\Gamma_i$ is a group of order $k$. That is, we do not need to require that each $\Gamma_i$ is cyclic. The proofs are essentially the same.
\end{remark}

\begin{remark}
The {\bf weak Pinsker entropy} of a system $(X,\mu,T)$ is defined to be the supremum of the Shannon entropies $H(\tK,\k)$ over all direct Bernoulli factors of the form $\G \cc (\tK,\k)^\G$. This concept was introduced in \cite{MR4424970}. So Theorem \ref{mainthm2} implies that $(X,m_X,T)$ has zero weak Pinsker entropy. 
\end{remark}


Finally, our methods also give the following.

\begin{mainthm}\label{mainthm5}
    If $k > d \ge 3$ then the system $(X,m_X,T)$ is not weakly contained in a Bernoulli system.
\end{mainthm}

For a definition of weak containment of measure-preserving systems, see for example \cite{burton2020}.

Theorems~\ref{mainthm2} and \ref{mainthm5} are consequences of the system having totally shattered microstate spaces along some sofic approximation: see Corollary~\ref{cor:shatt-cons}.

Together, Theorems~\ref{mainthm1} and~\ref{mainthm5} show that $(X,m_X,T)$ has completely positive sofic entropy along some sofic approximation, hence also completely positive Rokhlin entropy, but is not weakly contained in a Bernoulli shift.  The authors are not aware of other systems for which both of these properties have been verified previously.

\subsection{Probabilistic versions of the main theorems}

Both Theorem~\ref{mainthm1} and Theorem~\ref{mainthm2} will be derived as corollaries of probabilistic theorems. To explain, we say that a permutation of a set $V$ is \textbf{$k$-uniform} if it consists entirely of $k$-cycles: that is, its cycle type is $[k^{n/k}]$.  Recall that $\G = \G_{d,k} =  \langle s_1,\dots,s_d:\ s_1^k = \dots = s_d^k = e\rangle$. A homomorphism $\s:\G \to \Sym(V)$ is \textbf{$k$-uniform} if for each generator $s_i$, the image $\s(s_i)$ is $k$-uniform. Let $\PP_n$ be the uniform distribution on the set $\Hom_{\mathrm{unif}}(\Gamma,\Sym(V_n))$ of $k$-uniform homomorphisms. We will always assume $n$ is chosen so that $k$ divides $n$ (since otherwise there are no $k$-uniform permutations). 

We infer the existence of sofic approximations with the desired properties from the next proposition, which follows immediately from a more precise estimate in \cite[Lemma 3.1]{airey2022}.

 \begin{prop}\label{prop:probably-sofic}
There are subsets $\Omega_n^{\mathrm{sofic}}$ of $\Hom_{\mathrm{unif}}(\Gamma,\Sym(V_n))$ with ${\PP_n(\Omega_n^{\mathrm{sofic}})\to 1}$ and such that $(\sigma_n)_n$ is a sofic approximation to $\Gamma$ whenever $\sigma_n \in \Omega_n^{\mathrm{sofic}}$ for all $n$. \qed
\end{prop}

With Proposition \ref{prop:probably-sofic} in hand, Theorem~\ref{mainthm1} is a corollary of the following probabilistic version.

\begin{theorem}\label{mainthm3}
    There are subsets $\Omega_n' \subseteq \Omega_n^{\mathrm{sofic}}$ with $\PP_n(\Omega_n')\to 1$ 
    and such that, if $\sigma_n \in \Omega_n'$ for every $n$, then
    \begin{itemize}
\item[a.] the sofic entropy of $(X,m_X,T)$ along any subsequence of $(\sigma_n)_n$ equals
\[(1-d/k)\log 2;\]
 \item[b.] every nontrivial factor of $(X,m_X,T)$ inherits positive sofic entropy along $(\sigma_n)_n$.
    \end{itemize}
\end{theorem}

Part (b) of this theorem refers to `inherited' sofic entropy.  This quantity has previously been referred to as `extension' sofic entropy, `outer' sofic entropy, or `sofic entropy in the presence'.  The outer Pinsker factor along $(\sigma_n)_n$ is the largest factor of $(X,m_X,T)$ for which this quantity vanishes, so part (b) above asserts that the outer Pinsker factor of $(X,m_X,T)$ along $(\sigma_n)_n$ is trivial.  We recall these definitions and discuss our choice of terminology in Subsection~\ref{S:prelim-inherited-entropy} below.

\begin{remark}
    It is tempting to summarize part (b) of Theorem~\ref{mainthm3} as ``$(X,m_X,T)$ has completely positive entropy with high probability according to $\PP_n$''.  But we must be careful, because this summary hides an important detail about the order of quantifiers.  For part (b) of Theorem~\ref{mainthm3}, the high-probability subsets $\Omega_n'$ are found a priori, and then \emph{any} factor of $(X,m_X,T)$ has positive sofic entropy along \emph{any} choice of $\sigma_n$ from those subsets.  One could instead take ``completely positive entropy with high probability'' to mean that for every factor of $(X,m_X,T)$ there is a sequence of high-probability sets $\Omega'_n$ (depending on the factor) within which one sees the desired positive sofic entropy.  The latter conclusion is formally weaker, and we do not know whether there exist examples that satisfy one but not the other.  Since the collection of all factors may be uncountable, one cannot use a simple diagonal argument to prove that these two formulations are equivalent.  (If one identifies factors with their conditional expectation operators, then their induced strong operator topology is separable, but the semi-continuity properties of sofic entropy still do not obviously combine with approximations in this topology to enable a more complicated diagonal argument.)
\end{remark}

Theorem~\ref{mainthm2} is also a corollary of a probabilistic assertion about randomly-chosen $\sigma_n$ and the geometry of the space of microstates. We will prove that the space of microstates shatters in a strong sense, which roughly speaking means that it splits into a union of exponentially many well-separated clusters, each of which has sub-exponential size. The precise definition is given in \S \ref{S:shattering}, and the precise statement of the result for the parity-check subshift (Theorem \ref{thm:umm}, part (3)) is formulated in \S \ref{S:outline} below. It was essentially already known from \cite{bowen2022} that this phenomenon is incompatible with having a direct Bernoulli factor. We give a formal proof of this incompatibility in \S \ref{S:shattering}.

Our focus on the geometry and on shattering is inspired by similar ideas from statistical physics on random sparse graphs \cite{MR2277138}. As far as we know, the term `shattering' first appears in \cite{MR3385742}. In other works,  this phenomenon is called dynamical replica symmetry breaking \cite{MR847620, MR2317690}.



\subsection{Outline of the paper}

This paper is divided into two parts. In Part I, we consider general symbolic dynamical systems with a focus on the case in which $\G$ is a free product of finite cyclic groups.
\begin{itemize}
\item Section \ref{S:prelim} is a review of sofic group theory and sofic entropy. 
\item Section \ref{sec:B-K} introduces Kikuchi entropy and annealed entropy for actions of $\G_{d,k}$. These entropies generalize the $F$-functional and the $f$-invariant from \cite{bowen2010a} and are strongly related to the first moment method in statistical physics. This is the main tool for proving  Theorem \ref{mainthm3} part (a). 
\item Section \ref{S:cpeandmm} relates completely positive inherited entropy to a version of uniform mixing for model spaces that we call `property M'.  These general notions provide the background and tools needed to prove Theorem \ref{mainthm3} part (b).
\item Section \ref{S:shattering} defines totally shattered microstate spaces and shows how this property prevents both direct Bernoulli factors and weak containment in a Bernoulli shift.
\end{itemize}

Part II focuses on the parity check sub-shifts which appear in the main theorems:
\begin{itemize}
    \item Section \ref{sec:LDPC} discusses random (finite) LDPC codes.
    \item Section \ref{sec:propM} proves that for a typical sequence of these codes, the sequence of uniform measures on the codebooks has property M.
    \item Section \ref{S:lde} proves that typically these uniform measures converge locally and empirically to the Haar measure. Since the sequence also has property M, this gives Theorem~\ref{mainthm3}(b) as an application of Theorem~\ref{T:wmmcpe}.
    \item Section \ref{sec:mainthm3pf} proves that the sofic entropy of the Haar measure is $(1-d/k) \log 2$ over most sofic approximations (Theorem~\ref{mainthm3}(a)).
    \item Section \ref{S:mainthm2pf} proves that the Haar measure typically has totally shattered microstate spaces (Theorem~\ref{thm:umm}(3)). This implies Theorem~\ref{mainthm2} (the Haar measure has no direct Bernoulli factors) via Corollary~\ref{cor:shatt-cons}(2).
\end{itemize}

\subsection{Notational conventions}
\label{sec:notation}
We use the following standard notation for approximate comparison of functions.  Let $f$ and $g$ be real-valued functions on the same domain $S$, and let $A$ be any set of additional parameters in specifying these functions.  Then:
\begin{itemize}
    \item In case $g$ is non-negative, we write $f = O_A(g)$ if there is a positive constant $c$, depending possibly on $A$ but on nothing else, such that $|f| \le c g$.
    \item In case $g$ is non-negative and $S = \NN$, we write $f = o_A(g)$ if $g(n)$ is strictly positive for all sufficiently large $n$ and $f(n)/g(n)\to 0$, where this convergence may be bounded by a function that tends to $0$ and depends only on $A$.
    \item For any $S$, we write $f \lesssim_A g$ or $g \gtrsim_A f$ if (i) both functions are non-negative, and (ii) there is a positive constant $c$, depending possibly on $A$ but on nothing else, such that $f \le c g$.  Note that this is similar to writing `$f = O_A(g)$', except here it is part of the assertion that $f$ and $g$ are both non-negative, in case this is not obvious.
\end{itemize}

Finally, if $(\Omega_n,\PP_n)$ is a sequence of probability spaces indexed by $n \in \NN$, we write `$o_p(1)$' as a placeholder for any sequence of random variables $X_n$ on these spaces
such that
\[\PP_n(|X_n| > \varepsilon) \to 0 \qquad \forall \varepsilon > 0.\]
So this is essentially an analog of $o(1)$ for convergence in probability.

\subsection{Acknowledgements}
The authors thank Brandon Seward for helpful discussions, and an anonymous referee for pointing out several corrections. The authors (ordered alphabetically by last name) were supported by NSF grants  DMS-1855694, DMS-2154680 and DMS-1937215 respectively.




For the purpose of open access, the authors have applied a Creative Commons Attribution (CC-BY) licence to any Author Accepted Manuscript version arising from this submission.

\part{General systems}\label{P1}

\section{Preliminaries}\label{S:prelim}

\subsection{Sofic groups}\label{S:prelim-sofic}

A {\bf sofic approximation} to a countable group $\Gamma$ is a sequence $\Sigma=\{\sigma_n\}_{n\in \NN}$ of maps
$$\sigma_n: \Gamma \to \Sym(V_n)$$
where $V_n$ are finite sets, $\Sym(V_n)$ is the symmetric group on $V_n$, and the sequence is required to be asymptotically homomorphic and free in the following sense: For every $g,h \in \Gamma$ we require that the homomorphism equation $\s_n(gh)v=\s_n(g)\s_n(h)v$ holds for asymptotically all $v\in V_n$:
$$1 = \lim_{n\to\infty} |V_n|^{-1} |\{ v\in V_n:~ \s_n(gh)v = \s_n(g)\s_n(h)v\}|.$$
For every non-identity element $g \in \Gamma \setminus \{1_\Gamma\}$, we require that the percentage of points fixed by $\s_n(g)$ tends to zero:
\begin{eqnarray*}
0&=& \lim_{n\to\infty} |V_n|^{-1} |\{ v\in V_n:~ \s_n(g)v = v\}|.
\end{eqnarray*}
A group $\Gamma$ is {\bf sofic} if it admits a sofic approximation. 

If $\Gamma$ admits a finite generating set $S$ then it is common to visualize a map $\sigma_n$ as above in terms of the labelled and directed graph $G(\sigma_n)=(V_n,E_n)$ it induces. The edges of this graph are pairs of the form $(v, \s_n(s)v)$ for $s\in S$ and $v\in V_n$, and the label of this pair is $s$. Then $\Sigma$ is a sofic approximation precisely when this sequence of graphs $G(\sigma_n)$ Benjamini-Schramm converges to the (labelled and directed) Cayley graph of $\Gamma$ with respect to $S$ \cite{MR2089244}. 

It is an exercise to check that all amenable groups and all residually finite groups are sofic. Because finitely generated linear groups are sofic and direct unions of sofic groups are also sofic, it follows that all linear groups are sofic. It is an open problem whether all countable groups are sofic. 

The class of sofic groups was introduced implicitly by Gromov \cite{MR1694588} and explicitly by Weiss \cite{weiss-2000}. For further background on sofic groups, see \cite{pestov-sofic-survey, capraro-lupini}.

\subsection{Sofic entropy}\label{S:prelim-entropy}


This section defines sofic entropy for subshifts using the formulation from \cite{bowen-ICM}. See also \cite{bowen-survey} or \cite{MR3616077} for more comprehensive references.

Let $\Gamma$ denote a countable group and let $\A$ be a finite set (called the {\bf alphabet}). Let $\A^\Gamma$ be the set of all functions $x:\Gamma \to \A$. We write either $x_g$ or $x(g)$ for the value of $x$ on $g\in \Gamma$, whichever is most convenient. We endow $\A^\Gamma$ with the pointwise convergence topology, under which it is compact and metrizable. 

Let $\Prob(\A^\Gamma)$ denote the space of all Borel probability measures on $\A^\Gamma$, which we endow with the weak* topology. In this topology, a sequence of Borel probability measures $(\mu_n)_{n \in \NN}$ converges to a measure $\mu_\infty$ if and only if for every continuous function $f:\A^\Gamma \to \RR$,  $\lim_{n\to\infty} \int f~d \mu_n= \int f~d \mu_\infty$. An equivalent characterization uses cylinder sets which are defined as follows. Given a finite subset $F \subset \Gamma$ and $x:F \to \A$ let $C(x,F)$ be the set of all functions $y:\Gamma \to \A$ such that $y(f)=x(f)$ for all $f\in F$. Then $(\mu_n)_{n \in \NN}$ converges to $\mu_\infty$ if and only if for every such $F$ and $x$, $\lim_{n\to\infty} \mu_n(C(x,F)) = \mu_\infty(C(x,F))$. This is because the cylinder sets $C(x,F)$ form a sub-basis for the topology on $\A^\Gamma$.

Let $T = (T^g)_{g\in \Gamma}$ be the shift action on $\A^\Gamma$ defined by $T^gx(f)= x(g^{-1}f)$ for $x \in \A^\Gamma$. This induces an action on $\Prob(\A^\Gamma)$ by pushforwards. The set of all shift-invariant Borel probability measures on $\A^\Gamma$ will be denoted $\Prob^\Gamma(\A^\Gamma)$.  If $\mu$ is a shift-invariant Borel probability measure on $\A^\G$ then the system $(\A^\G,\mu,T)$ is called a {\bf shift $\G$-system}. We define here the sofic entropy of such systems.



Given $\sigma:\Gamma \to \textrm{Sym}(V)$, $v \in V$ and $x:V \to \A$, the {\bf pullback name of $x$ at $v$} is the labeling $\Pi^\sigma_v(x) \in \A^\Gamma$ defined by
    \[ \Pi^\sigma_v(x)(g) = x_{\sigma(g^{-1})v} \quad \forall g \in \Gamma. \]
For the sake of building some intuition, note that when $\sigma$ is a homomorphism, the map $v \mapsto \Pi^\sigma_{v}(x)$ is $\Gamma$-equivariant (in the sense that $\Pi^\sigma_{\sigma(g)v}(x) = T^g \Pi^\sigma_{v}(x)$). In particular $\Pi^\sigma_{v}(x) \in \A^\Gamma$ is periodic. In general, we think of $\Pi^\sigma_v(x)$ as an approximate periodic point. 

The {\bf empirical measure of $x:V\to \A$} is 
$$P^\sigma_x = |V|^{-1} \sum_{v\in V} \delta_{\Pi^\sigma_v(x)} \in \Prob(\A^\Gamma)$$
where, for $y \in \A^\Gamma$, $\delta_y \in \Prob(\A^\Gamma)$ is the Dirac measure concentrated on $\{y\}$. For example, if $\sigma$ is a homomorphism then $P^\sigma_x$ is a $\Gamma$-invariant measure supported on the $\Gamma$-orbits of the pullback names $\Pi^\sigma_v(x)$. 


Given an open set $\mathcal{O} \subset \Prob(\A^\Gamma)$, a map $x:V \to \A$ is called an {\bf $(\mathcal{O},\sigma)$-microstate} if $P^\sigma_x \in \mathcal{O}$. Typically we take $\calO$ to be a small neighborhood of $\mu$, in which case we consider $(\calO,\sigma)$-microstates to be `good microstates for $\mu$'.
Let $\Omega(\sigma,\mathcal{O}) \subset \A^V$ denote the set of all $(\mathcal{O},\sigma)$-microstates. 

Let $\mu \in \Prob(\A^\Gamma)$ be $\Gamma$-invariant and let $\Sigma=(\sigma_i:\Gamma \to \textrm{Sym}(V_i))_{i \in \NN}$ be a sofic approximation to $\Gamma$.  We say that the system $(\A^\Gamma,\mu,T)$ {\bf has microstates along $\Sigma$} if for every neighbourhood $\calO$ of $\mu$
\[\Omega(\s_n,\calO) \ne \emptyset \quad \hbox{for all sufficiently large }n.\]
More generally, an arbitrary measure-preserving $\Gamma$-system {\bf has microstates along $\Sigma$} if every shift-system factor of it has microstates along $\Sigma$.

The {\bf $\Sigma$-entropy} of the action $(\A^\Gamma, \mu, T)$ is defined by
\begin{equation}\label{eq:sigma-ent}
h_\Sigma(\A^\Gamma, \mu, T):= \inf_{\mathcal{O} \ni \mu} \limsup_{i\to\infty} |V_i|^{-1} \log \abs{\Omega(\sigma_i,\calO)}
\end{equation}
where the infimum is over all open neighborhoods of $\mu$ in $\Prob(\A^\Gamma)$. We abbreviate this to $h_\Sigma(\mu)$ if the other data are clear from the context. This number depends on the action $(\A^\Gamma, \mu, T)$ only up to measure conjugacy~\cite{bowen-jams-2010}.  It therefore defines an invariant for any abstract measure-preserving system that can be represented up to measure conjugacy by a shift system with a finite alphabet, or, equivalently, that has a finite generating partition. If the system does not have microstates along any subsequence of $\Sigma$ then we declare that the $\Sigma$-entropy is $-\infty$.





\subsection{Factor maps and inherited entropy}\label{S:prelim-inherited-entropy}

The second part of Theorem~\ref{mainthm1} concerns the entropy not only of $(X,m_X,T)$ but also of all its factors. Since $X$ is a shift-invariant subset of $\ZZ_2^\Gamma$, the $\Sigma$-entropy of $(X,m_X,T)$ is an instance of formula~\eqref{eq:sigma-ent}.  But this system may have factors that are not measure conjugate to shift systems, so that formula~\eqref{eq:sigma-ent} does not apply.

Sofic entropy can be generalized to measure-preserving systems on standard measurable spaces in various ways: see, for instance,~\cite[Subsection 2.4]{bowen-survey} and~\cite[Subsection 3.1]{austin2016} for formulations and proofs of their equivalence.  However, rather than repeat these in detail here, we need only recall how they are controlled by another entropy notion that does permit us to reduce our work to the study of shift systems.

Towards defining this, consider a factor map between two shift systems on finite alphabets, say
\begin{equation}\label{eq:factor-map}
\Phi:(\A^\Gamma,\mu)\to (\B^\Gamma,\nu),
\end{equation}
where $\mu$ and $\nu$ are both invariant under the shift actions of $\Gamma$ on their respective spaces.  Rather than counting good microstates for $\mu$ or $\nu$ separately, we can ask how many of the good microstates for $\nu$ can be lifted to good microstates for $\mu$: that is, how many good microstates $\nu$ `inherits' through the map $\Phi$.  To make this precise, consider the graphical joining
\begin{equation}\label{eq:graphical}
\lambda := \int_{\A^\Gamma} \delta_{(x,\Phi(x))}\ d\mu(x).
\end{equation}
This is invariant under the shift action of $\Gamma$ on $(\A\times \B)^\Gamma$. Let $\mathrm{proj}_i$ be the coordinate projection from $(\A\times \B)^{V_i}$ to $\B^{V_i}$.  Finally, define the \textbf{inherited $\Sigma$-entropy of $\Phi$} to be
\begin{equation}\label{eq:inherited}
h_\Sigma(\mu,T\,;\,\Phi) := \inf_{\mathcal{O}\ni \lambda}\limsup_{i\to\infty}|V_i|^{-1}\log|\mathrm{proj}_i[\Omega(\sigma_i,\calO)]|.
\end{equation}
If $\Phi$ is the identity then this is easily checked to coincide with $h_\Sigma(\mu)$.

Like sofic entropy itself, inherited sofic entropy can be generalized to factor maps between arbitrary measure-preserving systems.  This general notion is not new, but our use of the term `inherited' is new.  Indeed, the idea behind this quantity is already implicit in Kerr's original approach to defining sofic entropy itself for general measure-preserving systems (see~\cite{kerr--soficdfn} or~\cite[Subsubsection 2.4.2]{bowen-survey}).  It was formulated and studied explicitly by Hayes in~\cite{hayes2016a,hayes2017}, who refers to it as `entropy in the presence' in recognition of a parallel usage in Voiculescu's theory of free entropy~\cite{Voi--III}.  It has also been studied by other authors under various names, including `outer sofic entropy' and `extension sofic entropy'.  It is reviewed in~\cite[Subsection 11.1]{bowen-survey}, which gives more complete references.  In spite of this history, we do propose the new name `inherited' entropy, because this seems to capture the idea behind the definition better.

Starting from~\eqref{eq:inherited}, one can define entropy for a factor map between general systems by carefully inserting a supremum over generating partitions of the lower system and an infimum over generating partitions of the upper system.  However, the proof that sofic entropy itself is invariant under measure conjugacy can be adapted directly to the quantity in~\eqref{eq:inherited}, showing that it is an invariant of $\Phi$, where we consider two factor maps to be equivalent if they appear downwards in a commuting square whose horizontal arrows are conjugacies.  As a corollary,~\eqref{eq:inherited} coincides with the abstract inherited entropy in the case of two shift systems.  A more general version of this argument can be found in~\cite[Proposition 2.9]{hayes2016a}.

One crucial advantage to working with inherited sofic entropy is its monotonicity under factor maps.  The following lemma is a special case of parts (iii) and (iv) of~\cite[Proposition 2.10]{hayes2016a}.

\begin{lemma}\label{lem:presence-monotone}
If
\[(X,\mu,T) \stackrel{\Pi}{\to} (Y,\nu,S) \stackrel{\Phi}{\to} (Z,\theta,R)\]
are factor maps, then
\[h_\Sigma(\mu,T\,;\,\Phi\circ\Pi) \le \min\big\{h_\Sigma(\nu,S\,;\,\Phi),\  h_\Sigma(\mu,T\,;\,\Pi)\big\}.\]

In particular, taking $\Pi$ (resp. $\Phi$) to be the identity, we obtain
\[h_\Sigma(\mu,T\,;\,\Phi) \le h_\Sigma(\mu,T) \quad \hbox{(resp.}\quad h_\Sigma(\mu,T\,;\,\Pi) \le h_\Sigma(\mu,S)\quad \hbox{).}\]
\qed
\end{lemma}

Inspired by this property, one defines the \textbf{outer $\Sigma$-Pinsker factor} of a measure-preserving system to be the largest factor whose inherited $\Sigma$-entropy is zero.  A routine argument shows that a unique maximal such factor exists.  Hayes' paper~\cite{hayes2016a} develops this story as well, although it had been considered in unpublished work previously.  As a result, the assertion that every nontrivial factor of a measure-preserving system has positive inherited $\Sigma$-entropy is equivalent to the assertion that the outer $\Sigma$-Pinsker factor of that system is trivial.  In addition, by the last inequality in Lemma~\ref{lem:presence-monotone}, if a system has this property then it also has completely positive $\Sigma$-entropy.  The observations together explain our formulation of the second part of Theorem~\ref{mainthm1}.

Starting from Lemma~\ref{lem:presence-monotone}, Hayes develops enough properties of the outer Pinsker factor to turn it into a valuable tool in the study of sofic entropy in general.  For instance, these properties are crucial to his proof that a certain large class of systems of algebraic origin all have completely positive sofic entropy~\cite{hayes2017}. Our reasons for using inherited entropy in the proof of Theorem~\ref{mainthm1} are similar, but the details of our proof are essentially disjoint from those of Hayes', and our LDPC system is not among the systems of algebraic origin that he considers in that reference.

Among its useful consequences, Lemma~\ref{lem:presence-monotone} gives the following:

\begin{cor}\label{cor:presence-monotone}
Let $(X,\mu,T)$ be a measure-preserving system.  If every nontrivial factor map of $(X,\mu,T)$ to a shift system has positive inherited $\Sigma$-entropy, then \emph{every} nontrivial factor map of $(X,\mu,T)$ has positive inherited $\Sigma$-entropy.
\end{cor}

\begin{proof}
Let $\Pi$ be a factor map to another nontrivial system $(Y,\nu,S)$.  Then $(Y,\nu)$ has a nontrivial partition into two measurable subsets.  By acting on this partition using $S$, we define a further factor map of the form
\[\Phi:(Y,\nu,S)\to (\{0,1\}^\Gamma,\theta,\rm{shift}),\]
where $\theta$ is not a Dirac measure.  Now our hypothesis gives that $h_\Sigma(\mu,T\,;\,\Phi\circ \Pi) > 0$, and this implies that $h_\Sigma(\mu,T\,;\,\Pi) > 0$ by Lemma~\ref{lem:presence-monotone}.
\end{proof}

Corollary~\ref{cor:presence-monotone} can simplify many of the technicalities involved in a proof of completely positive inherited entropy by letting us restrict our attention to factor maps between shift spaces.  Given such a map, say
\[\Phi:(\A^\Gamma,\mu)\to (\B^\Gamma,\nu),\]
it is uniquely determined by its coordinate at the identity of $\Gamma$, which is an arbitrary measurable map $\phi:\A^\Gamma\to \B$.  Starting from $\phi$, we write $\phi^\Gamma$ for the factor map that it induces, which is given by
\begin{equation}\label{eq:map-to-factor}
\phi^\Gamma(x)(\gamma) = \phi(\gamma^{-1} \cdot x) .
\end{equation}

When working with such a map $\Phi$, a further simplification is often necessary.  If $D$ is a finite subset of $\Gamma$, then the map $\phi$ above is called \textbf{$D$-local} if the image $\phi(x)$ depends only on the coordinates of $x$ indexed by $D$; equivalently, if $\phi$ factorizes into maps
\[\A^\Gamma \to \A^D\to \B,\]
where the first map is the coordinate projection.  We say that $\phi$ is \textbf{local} if it is $D$-local for some finite set $D$, and apply the same terminology to the whole of $\phi^\Gamma$. This is equivalent to $\phi^\Gamma$ being a continuous map for the product topologies on our shift spaces.

If $\mb{x} \in \A^{V}$ is some good microstate for $\mu$ over a map $\sigma \colon \Gamma \to \Sym(V)$, we might attempt to send it to a good microstate for $\nu$ using the map $\phi^\sigma \colon \A^V \to \B^V$ defined by
\begin{equation}
\label{eqn:microstatemap}
    \phi^\sigma(\mb{x})(v) = \phi( \Pi_v^\sigma \mb{x}) ,
\end{equation}
since the empirical distribution of $\phi^\sigma(\mb{x})$ would then be $\phi^\Gamma_* P_{\mb{x}}^\sigma$. Since $P_{\mb{x}}^\sigma$ is close to $\mu$, one would hope this would be close to $\nu$. This argument is correct in case $\phi^\Gamma_*$ acts continuously on measures, which in turn holds if $\phi$ is local.  In that case, $\phi^\sigma$ also has the following form of quantitative continuity.

\begin{lemma}
\label{lem:localmaplipschitz}
    Suppose $\A,\B$ are finite sets, $\phi \colon \A^\G \to \B$ is $D$-local, and $\sigma \colon \G \to \Sym(V)$. If $\mb{x}$ and $\mb{y}$ are elements of $\A^V$ that disagree in exactly one coordinate, then $\phi^\sigma(\mb{x})$ and $\phi^\sigma(\mb{y})$ disagree in at most $|D|$ coordinates.
\end{lemma}

Equivalently, this asserts that $\phi^\sigma$ is $|D|$-Lipschitz for the `normalized Hamming metrics' on $\A^V$ and $\B^V$.  This point of view is introduced and used for an application of Lemma~\ref{lem:localmaplipschitz} in Subsection~\ref{subs:shatt-cons}.

\begin{proof}
    Suppose that $\mb{x}(u)\ne \mb{y}(u)$, and let $v \in V$.  Since $\phi$ is $D$-local, we can have $\phi(\Pi_v^\sigma\mb{x}) \ne \phi(\Pi_v^\sigma\mb{y})$ only if there exists $\gamma \in D$ such that $\mb{x}(\sigma(\gamma^{-1}) v) \ne \mb{y}(\sigma(\gamma^{-1}) v)$, hence only if $u$ appears in the set $\sigma(D^{-1})\cdot v$.  Since $\sigma$ is an action by permutations, this holds if and only $v$ lies in $\sigma(D^{-1})^{-1}\cdot u$, and that set has cardinality at most $|D^{-1}| = |D|$.
\end{proof}

The arguments about $\phi^\sigma$ above can fail for general measurable factor maps, but this problem can be overcome by approximating these in measure by local factor maps.  We say that a sequence of maps
\[\psi_m:\A^\Gamma \to \B \quad (m=1,2,\dots)\]
is a \textbf{local approximating sequence} to $\phi$ if each $\psi_m$ is local and
\begin{equation}\label{eq:local-approx}
\mu\{\psi_m \ne \phi\} \to 0.
\end{equation}
(Note that this notion implicitly also depends on the measure $\mu$.)  These are the special case for shift spaces of the `almost Lipschitz approximating sequences' introduced and used in~\cite{austin2016} to study factor maps between more general measure-preserving systems.  At a few points in the sequel we refer to~\cite{austin2016} for properties that we need in our special case.  For instance, informally, if $\psi$ is a good enough local approximation to $\phi$, then $\psi^\sigma$ sends very good microstates for $\mu$ to fairly good microstates for $\phi^\Gamma_*\mu$ \cite[Lemma 4.10]{austin2016}.


\section{Bethe--Kikuchi entropy and annealed calculations}\label{sec:B-K}

Throughout this section, we set $\Gamma = \Gamma_{d,k} =\ZZ_k \ast \dots \ast \ZZ_k$ equal to the free product of $d$ copies of $\ZZ_k = \ZZ/k\ZZ$. 

The proof of Theorem \ref{mainthm1} part (a) (and its probabilistic version Theorem \ref{mainthm3}) relies on a first moment calculation. These kinds of computations have a long and interesting history in statistical physics and more recently appeared in the entropy theory of actions of the free group, where, for example, they were used to answer a long-standing open problem on the classification of Bernoulli shifts over free groups \cite{bowen2010a}. This history is reviewed in \S \ref{S:annealed-history}. 

In the next section, we introduce the uniformly random sofic approximation to $\Gamma$. This will be our main focus and we will only obtain results about deterministic sofic approximations indirectly as a consequence of the analysis of these random ones. 

In Section \ref{sec:kikformula}, we introduce Kikuchi entropy, which has historical roots in physics \cite{kikuchi1951} and is a version of the functional $F$ of \cite{bowen2010a} adapted from the free group to groups $\G$ which are free products of finite groups. This entropy is a first approximation to the annealed entropy which in recent ergodic-theory research was called the $f$-invariant \cite{bowen2010a}. These entropies are key ingredients for proving Theorem \ref{mainthm3} part (a).

\subsection{Random sofic approximations}\label{subs:rndm-sofic-approx}


Instead of directly constructing sofic approximation sequences $\s_n: \Gamma \to \Sym(V_n)$, we will construct probability measures on the space of homomorphisms from $\Gamma$ to $\Sym(kn)$. 
These measures were used for the same purpose in~\cite{airey2022}. 

Let $V$ be a finite set whose size is divisible by $k$.  Let us say that a permutation of $V$ is \textbf{$k$-uniform} if it consists entirely of $k$-cycles: that is, its cycle type is $[k^{n/k}]$.  Consider a $d$-tuple of $k$-uniform permutations $(\sigma(s_1),\dots,\sigma(s_d))$.  Then $\sigma(s_i^k)$ equals the identity permutation for each $i$, and so the tuple $(\sigma(s_i),\dots,\sigma(s_d))$ is the image of the generators $(s_1,\dots,s_d)$ under a homomorphism $\sigma:\Gamma \to \Sym(V)$.  We call such a homomorphism \textbf{$k$-uniform} if the images of these generators are all $k$-uniform permutations. Denote the set of $k$-uniform homomorphisms into $\Sym(V)$ by $\Hom_{\unif}(\Gamma, \Sym(V))$. Note that for arbitrary members of $\Hom(\Gamma, \Sym(V))$ the cycle sizes of the images of the generators of $\Gamma$ must be factors of $k$.

Given a homomorphism $\sigma \in \Hom_{\unif}(\Gamma, \Sym(V))$, consider the collection of all the orbits of the individual maps $\sigma_i = \sigma(s_i)$: that is, all the subsets of the form
\begin{equation}\label{eq:hyper-edge}
\{ \sigma(s_i^j)(v) \st 0 \leq j < k \}
\end{equation}
for $v \in V$ and $i \in [d]$. Taken together, these may be regarded as a hyper-graph on $V$.   Because we consider a $k$-uniform homomorphism, this hyper-graph is $k$-uniform in the usual sense in combinatorics: this is the origin of the terminology.  However, it could happen that two $\sigma_i$ and $\sigma_j$ give some vertex the same orbit.  In this eventuality it is better to think of the family of sets~\eqref{eq:hyper-edge} as a \emph{multi-hyper-graph}, in that we count each hyper-edge with this multiplicity.

For each $n$, we set $V_n := \{1,2,\dots,n\}$. 
 If $k$ divides $n$, then we let $\PP_n$ be the uniform distribution on $\Hom_{\mathrm{unif}}(\Gamma,\Sym(V_n))$.  The sofic approximations that appear in Theorem~\ref{mainthm1} are obtained at random from these distributions, and shown to have all the desired properties with high probability.

 In order to take this approach, Proposition~\ref{prop:probably-sofic} above is a basic prerequisite. It allows us to focus on typical properties of the microstate spaces, knowing that random homomorphisms will be good sofic approximation maps with high probability. It is implied by the more precise estimate in~\cite[Lemma 3.1]{airey2022}, where our $\PP_n$ is called `$\PP_n^u$'.  Note that attention is restricted to even $n$ in that reference, but that restriction is unnecessary for the case of $\PP_n^u$; it is included there only for the sake of the other probability distribution $\PP_n^p$ that is covered by the same lemma.

\subsection{Formula for Kikuchi entropy}
\label{sec:kikformula}


Let $\A$ be a finite set and $\Gamma = (\ZZ_k)^{\ast d}$ as above. Given a $\Gamma$-invariant measure $\mu \in \Prob(\A^\Gamma)$, we define the {\bf edge weight} $W_\mu( \cdot; \cdot) \colon \A^{\ZZ_k} \times [d] \to [0,1]$ by
	\[ W_\mu (\mb{a}; i) = \mu \{ \mb{x} \in \A^\Gamma \st \mb{x}(s_i^j) = \mb{a}(j) \ \forall 0 \leq j < k \} . \]
Each $W(\cdot; i)$ is a probability measure on $\A^{\ZZ_k}$ which records the statistics of $\mu$ on the hyper-edge $\{s_i, s_i^2, \ldots, s_i^{d-1}\}$.
Also define the {\bf vertex weight} 
	\[ W_\mu(\ta) = \mu \{ \mb{x} \in \A^\Gamma \st \mb{x}(e) = \ta \} \]
	for $\ta \in \A$. This probability measure on $\A$ records the single-site statistics of $\mu$. It is determined by the edge weight $W_\mu(\cdot; \cdot)$.

More generally, an abstract {\bf weight} $W$ is a $d$-tuple of probability vectors $W(\cdot; 1), \ldots , W(\cdot; d)$ on $\A^{\ZZ_k}$ such that there is a probability vector $W(\cdot)$ on $\A$ satisfying the consistency condition
    \[ W(\ta_0) = \sum_{\mb{a} \in \A^{\ZZ_k} \st \mb{a}(j) = \ta_0 } W(\mb{a}; i)\]
for every $i \in [d]$, $j \in \ZZ_k$ and $\ta_0 \in \A$.

For any weight $W$, we define the {\bf Kikuchi entropy} of $W$ by
\[ \kik(W) := (1-d) \shent( W(\cdot) ) + \frac{1}{k} \sum_{i \in [d]} \shent( W(\cdot; i) ), \]	
where $\shent(\cdot)$ denotes Shannon entropy, and for a $\G$-invariant measure $\mu$ on $\A^\Gamma$ we abbreviate $\kik(W_\mu)$ to $\kik(\mu)$.  The functional $\kik$ appears in our work because it gives the upper exponential growth rate of the expected number of microstates whose averaged hyper-edge marginals are approximately specified by $\mu$.  This is an analog of \cite[Theorem 1.4]{bowen2010}.  Given a microstate $\mb{x} \in \A^V$ and $\sigma \in \Hom(\Gamma, \Sym(V))$, let $W_{\mb{x}, \sigma}$ be the weight corresponding to $P_{\mb{x}}^\sigma \in  \Prob(\A^\Gamma)$, the empirical distribution of $\mb{x}$ over $\sigma$.  Also, given two weights $W,W'$, let
$$\norm{W - W'} := \max_{1\le i \le d} \max_{\mb{a} \in \A^{\ZZ_k}} |W(  \mb{a};i) - W'(\mb{a} ;i)|.$$

\begin{prop}
\label{prop:kikformula}
We have
	\[ \kik(\mu) = \lim_{\varepsilon \to 0} \limsup_{n \to \infty} \frac{1}{nk} \log \EE_{\sigma_n \sim \PP_{nk}} \abs*{\{ \mb{x} \in \A^{nk} \st \norm{W_{\mb{x}, \sigma_n} - W_\mu} < \varepsilon \}}. \]
\end{prop}

The proof of this proposition shows that we also obtain $\kik(\mu)$ if $\limsup$ is replaced with $\liminf$ on the right-hand side.

\begin{proof}
  Let $\ZZ_k$ act on $\A^{\ZZ_k}$ in the usual way: $\pi \bta(j) = \bta(j-\pi)$ for $\pi \in \ZZ_k$, $\bta \in \A^{\ZZ_k}$ and $j \in \ZZ_k$. We will say that two elements $\bta, \btb \in \A^{\ZZ_k}$ are equivalent if they are in the same $\ZZ_k$-orbit. Let $[\bta] \subset \A^{\ZZ_k}$ denote the equivalence class of $\bta$. For a weight $W$ write
		\[ W([\bta]; i) \coloneqq \sum_{\btb \in [\bta]} W(\btb; i) . \]
A weight $W$ is {\bf cyclically invariant} if $W(\bta; i)=W(\btb; i)$ for every equivalent pair $\bta,\btb$ in $\A^{\ZZ_k}$. For example, since $\mu$ is $\Gamma$-invariant, for each $i \in [d]$ the probability measure $W_\mu(\cdot; i) \in \Prob(\A^{\ZZ_k})$ is cyclically invariant.

	We first calculate $\EE \abs*{\{ \mb{x} \in \A^{kn} \st W_{\mb{x}, \sigma_n} = W\}}$ for an arbitrary cyclically invariant weight $W$ with denominator $kn$, \emph{i.e.} such that $kn \cdot W([\bta]; i) \in \ZZ$ for each $i, \bta$. Letting
		\[ N_n = \abs*{\Hom_{\unif}(\Gamma, \Sym(kn))} = \left( \frac{(kn)!}{n! k^n} \right)^d, \]
	we have
	\begin{align*}
		\EE \abs*{\{ \mb{x} \in \A^{kn} \st W_{\mb{x}, \sigma_n} = W\}}
			&= \frac{1}{N_n} \sum_{\sigma \in \Hom_{\unif}(\Gamma, \Sym(kn))} \abs*{\{ \mb{x} \in \A^{kn} \st W_{\mb{x}, \sigma} = W \}} \\
			&= \frac{1}{N_n} \sum_{\mb{x} \in \A^{kn}} \abs*{\{ \sigma \in \Hom_{\unif}(\Gamma, \Sym(kn)) \st W_{\mb{x}, \sigma} = W\}} .
	\end{align*}
	Now a labeling $\mb{x} \in \A^{kn}$ admits at least one $\sigma$ that gives the correct weight $W$ if and only if it has the correct ``vertex statistics,'' that is, $\frac{1}{kn}\abs*{\{ i \in [kn] \st \mb{x}(i) = \ta \}} = W(\ta)$ for all $\ta \in \A$. There are $\exp\{ kn ( \shent(W(\cdot)) + o(1))\}$ such $\mb{x}$, and each admits the same number of $\sigma$. From now on fix one such $\mb{x}$. 
	
For $i \in [d]$, let $G_i$ be the set of $k$-uniform permutations $\pi \in \Sym(kn)$ such that if $W(\cdot; \pi)$ is the probability vector on $\A^{\ZZ_k}$ given by
$$W(\bta; \pi) = (kn)^{-1} |\{v \in [kn]:~ \bta(t) = \mb{x}(\pi^t(v))~\forall 0\le t < k\}|$$
then $W(\cdot ; \pi) \equiv W(\cdot; i)$.
Any $k$-uniform homomorphism $\sigma$ with $W_{\mb{x}, \sigma} = W$ is determined by the permutations $\s(s_i)$ which must be in $G_i$. So the number of such homomorphisms is $\prod_{i=1}^d |G_i|$.

Let $G(\mb{x})$ be the set of permutations $g\in \Sym(kn)$ which fix $\mb{x}$ in the sense that $\mb{x}(v)=\mb{x}(gv)$ for all $v\in [kn]$. Observe that $G(\mb{x})$ acts on $G_i$ by conjugation. This means that if  $\pi_i$ is a fixed permutation in $G_i$ and $g\in G(\mb{x})$ then $g \pi_i g^{-1} \in G_i$. Moreover, this action is transitive.
So 
$$|G_i| = \frac{ |G(\mb{x})|}{|\textrm{Stab}(\pi_i)|}$$
where $\textrm{Stab}(\pi_i)$ is the set of $g \in G(\mb{x})$ with $g \pi_i g^{-1} = \pi_i$. 
Observe
$$|G(\mb{x})| = \prod_{\ta \in \A} \left( kn \cdot W(\ta) \right) !.$$
	There are two mechanisms by which a $g \in G(\mb{x})$ can stabilize $\pi_i$. Either $g$ can permute $k$-cycles with the same labels or it can rotate a given labeled $k$-cycle. Therefore,
	\[ |\textrm{Stab}(\pi_i)| = \prod_{[\bta] \in \A^{\ZZ_k}/\ZZ_k} \left( n \cdot W([\bta]; i) \right)! \left( \frac{k}{\abs{[\bta]}} \right)^{n \cdot W([\bta]; i)} . \]

	Putting everything together, we get
		\[ \EE \abs*{\{ \mb{x} \in \A^{kn} \st W_{\mb{x}, \sigma_n} = W\}} = \frac{ e^{kn (\shent(W(\cdot)) + o(1))} \left( n! k^n \prod_{\ta \in \A} \left( kn \cdot W(\ta) \right) ! \right)^d}{(kn)!^d \prod_{i \in [d]} \prod_{[\bta] \in \A^{\ZZ_k}/\ZZ_k} \left( n \cdot W([\bta]; i) \right)! \left( \frac{k}{\abs{[\bta]}} \right)^{n \cdot W([\bta]; i)} } . \]
	Applying Stirling's approximation $\log n! = n \log n - n + o(n)$, the logarithm of this is
		\[ kn \left( \shent(W(\cdot)) + d\sum_{\ta \in \A} W(\ta) \log W(\ta) - \frac{1}{k} \sum_{i \in [d]} \sum_{[\bta]} W([\bta]; i) \log \frac{W([\bta]; i)}{\abs{[\bta]}} \right) + o(n) . \]
	By cyclic invariance of each $W(\cdot; i)$, we have $W(\bta; i) = \frac{W([\bta]; i)}{\abs{[\bta]}}$, so this gives
        \[ \EE \abs*{\{ \mb{x} \in \A^{kn} \st W_{\mb{x}, \sigma_n} = W\}} = e^{nk F(W) + o(n)} \]
    where
        \[ F(W) \coloneqq (1-d) \shent( W(\cdot) ) + \frac{1}{k} \sum_{i \in [d]} \shent( W(\cdot; i) ) . \]

    Now, since the number of denominator-$n$ weights $W$ grows polynomially in $n$, it follows that for any $\varepsilon>0$
        \[\lim_{n \to \infty} \frac{1}{kn} \log \EE \abs*{\{ \mb{x} \in \A^{kn} \st \norm{W_{\mb{x}, \sigma_n} - W_\mu} < \varepsilon\}} = \sup \{ F(W) \st \norm{W - W_\mu} < \varepsilon\} \]
    and taking $\varepsilon$ to 0 gives the claimed formula, by continuity of $F$.
\end{proof}

The {\bf annealed entropy} of $\mu$ is defined by
    \[ \hann(\mu) = \inf_{\calO \ni \mu} \limsup_{n \to \infty} \frac{1}{nk} \log \EE_{\sigma_n \sim \PP_{nk}} \abs*{\Omega(\sigma_n, \calO)} \]
where the infimum is over all open neighborhoods of $\mu$.

To emphasize the relationship between $\kik(\mu)$ and $\hann(\mu)$, let
$$\calO_\eps(\mu) = \{\nu \in \Prob(\A^\Gamma):~ \norm{W_\nu - W_\mu} <\eps\}.$$
Then $\calO_\eps(\mu)$ is an open neighborhood of $\mu$ and Proposition~\ref{prop:kikformula} becomes
    \[ \kik(\mu) = \inf_{\eps>0} \limsup_{n \to \infty} \frac{1}{nk} \log \EE_{\sigma_n \sim \PP_{nk}} \abs*{\Omega(\sigma_n, \calO_\eps(\mu))}. \]
In particular $\hann (\mu) \leq \kik(\mu)$.

The next proposition shows that $\kik(\mu)$ is an upper bound for sofic entropy with respect to ``most'' sofic approximations.

\begin{prop}\label{P:upper-estimate}
  Let $\mu$ be a $\Gamma$-invariant Borel probability measure on $\A^\Gamma$. Then there are subsets $\Omega_n' \subseteq \Omega_n^{\mathrm{sofic}}$ with 
    $$\lim_{n\to\infty} \PP_n(\Omega_n')= 1$$ 
    and such that, if $\Sigma=\{\sigma_n\}_{n=1}^\infty$ satisfies $\sigma_n \in \Omega'_{i_n}$ for some increasing sequence $(i_n)_n$ with $k \mid i_n$ for all $n$ then 
  $$h_\Sigma(\A^\Gamma,\mu,T)\le \kik(\mu).$$
\end{prop}

\begin{proof}
Given a positive integer $n$ divisible by $k$, $\varepsilon>0$, and $\s:\Gamma \to \Sym(n)$, let $N_{n,\varepsilon}(\s)$ be $\abs*{\Omega(\sigma, \calO_\eps(\mu))}$. Think of $N_{n,\varepsilon}$ as a random variable with respect to the uniform measure on the space of $k$-uniform homomorphisms from $\Gamma$ to $\Sym(n)$. In addition, let
\begin{eqnarray*}
 \mathrm{H}_{\mathrm{K},\varepsilon}(\mu) &:=& \limsup_{n \to \infty} \frac{1}{n} \log \EE[N_{n,\varepsilon}],
\end{eqnarray*}
where here and below we always restrict $n$ to be a multiple of $k$.

For each $\varepsilon$ and $n$ as above, let $\Omega_{n,\eps}'$ be the set of those $\sigma \in \Omega_n^{\mathrm{sofic}}$ that satisfy $N_{n,\eps}(\sigma) \le e^{\sqrt{n}}\EE[N_{n,\varepsilon}]$, and let $\Omega_n' := \Omega_{n,1}'\cap \cdots \cap \Omega_{n,1/n}'$. 
Then Markov's inequality gives
$$\PP_n(\Omega_n') \ge 1 - \sum_{m=1}^n\PP_n\big(N_{n,1/m}(\sigma) > e^{\sqrt{n}}\EE[N_{n,1/m}]\big) \ge 1 - ne^{-\sqrt{n}} \to 1.$$

Finally, suppose that $\Sigma=\{\sigma_n\}_{n=1}^\infty$ satisfies $\sigma_n \in \Omega'_{i_n}$ for some increasing sequence $(i_n)_n$ with $k \mid i_n$ for all $n$. Then we also have $\sigma_n \in \Omega'_{i_n,1/m}$ whenever $i_n\ge m$.  Fixing $m$ and letting $n\to\infty$, it follows that
$$h_\Sigma(\A^\Gamma,\mu,T) \le \limsup_{n \to \infty} \frac{1}{i_n} \log  N_{i_n,1/m}(\sigma_n) \le  \limsup_{n \to \infty} \frac{1}{i_n} \log \EE[N_{i_n,1/m}] \le \mathrm{H}_{\mathrm{K},1/m}(\mu),$$
where the first inequality uses again the fact that $\calO_{1/m}(\mu)$ is an open neighborhood of $\mu$ for every positive integer $m$. Now letting $m\to\infty$, Proposition~\ref{prop:kikformula} gives $\kik(\mu) = \lim_{m\to\infty}\mathrm{H}_{\mathrm{K},1/m}(\mu)$, so this completes the proof.
\end{proof}




\subsection{A brief history}\label{S:annealed-history}

The most basic method for analyzing the behaviour of a random sofic approximation is the first moment method.  Our first indication of the typical number of microstates for $(X,m_X,T)$ over $\sigma_n$ chosen from $\PP_n$ is given by the expectation of that number. 

In the analogous setting of actions of free groups, such averages have been studied intensely in recent years.  In~\cite{bowen2010}, the exponential growth rate of the expected number of good microstates was shown to coincide with an invariant of systems previously introduced by the second author in~\cite{bowen2010a}, where it was used to solve the isomorphism problem for finite-state Bernoulli actions of free groups.  In those and several subsequent papers, this invariant was called the `f-invariant'.  Here we propose a new term instead: we refer to this quantity as `annealed entropy'. 

In work of the second author, the f-invariant was obtained as a limit of functionals referred to as $F$, which are annealed entropies of Markov approximations. As explained further below, this quantity first appeared in refinements of work of Kikuchi \cite{kikuchi1951}. For this reason, we call it Kikuchi entropy. In later sections, it is used to prove Theorem \ref{mainthm3}.

The reason for the name ``annealed entropy'' is a connection to statistics and statistical physics.  During the last forty years, very similar first-moment calculations for various configurations over large sparse random graphs have become a central feature of the analysis of `graphical models' in those disciplines.  Often, the use of such averages can be seen as a first attempt to find the value for a `typical' random graph. 
In such settings, the first moment is referred to as an annealed average: see, for instance, the usage in~\cite[Section IV.1]{MPV1987} or~\cite[Section 5.4]{mezard2009}.  Its use as a prediction of typical behaviour is called the `Bethe ansatz' (or sometimes the `replica symmetric' approximation in reference to a phenomenology in the study of spin glasses that we do not explain here: see, for instance,~\cite[Chapter I]{MPV1987} or~\cite[Chapter 8]{mezard2009}\footnote{Indeed, the second author has previously also suggested the term `replica-symmetric entropy'~\cite[Subsection 7.3]{bowen-ICM}, but we feel `annealed' reflects the general nature of this quantity better.}).

In fact, the origins of these quantities lie even further back in the statistical physics literature.  In the general setting for random graphical models studied in statistics, the leading order exponents in first moment calculations are given by quantities called `Bethe' or `Kikuchi' entropy.

The first of these terms refers to foundational work by Bethe~\cite{bethe1935}.  He estimated the free energy of a certain model of an alloy on a two-dimensional lattice by a recursive expansion that retained nearest-neighbour interactions but ignored the effect of loops in the lattice graph.  A more mathematical description is that the two-dimensional lattice is approximated by an infinite regular tree, and this is why such trees are now often called `Bethe lattices' in statistical physics.

In~\cite{kikuchi1951}, Kikuchi expanded on Bethe's ideas by proposing a more careful expansion that respects slightly more of the lattice structure.  In modern terms, this can be understood as an approximation to the lattice by a hyper-tree rather than a simple tree.  While Bethe argued mostly in terms of free energies, Kikuchi's paper includes various explicit formulae for entropy estimates, and these evolved over time into the quantities studied in statistical inference today.  See~\cite[Equations (A.7) and (C1.6)]{kikuchi1951} for early intimations of these modern formulae.  Bethe's and Kikuchi's approximation methods can also be found in physics surveys from closer to that time such as Section III in Burley's contribution~\cite[Chapter 9]{DomGre72}.

These formulae were brought explicitly into statistical theory around 2000 by Yedidia and various co-authors in a series of technical reports: see, in particular,~\cite{yedidia2000,YFW2001} and the further references given there.  While these references continued to emphasize free energy more than entropy, they do cover both: the explicit formula for Bethe entropy is~\cite[Formula (1.32)]{yedidia2000}, for example.

The term `annealed entropy' (instead of `f-invariant') emphasizes the connection between these two fields.  To be more precise, the annealed entropy of a measure-preserving action of a free group is defined as an infimum of the values of a more elementary quantity over Markov approximations to the action.  This more elementary quantity, denoted by `$F$' in~\cite{bowen2010a,bowen2010}, has precisely the same formula as Bethe entropy.  Thus, the same formula for an entropy-like quantity was discovered independently and then used for very similar first-moment calculations in both fields.

Bethe entropy has by now become a textbook topic in statistical inference with graphical models: see, for instance,~\cite[Section 4]{WJ2008} or~\cite[Chapter 14, especially Subsection 14.2.4]{mezard2009}.  Some of the theory of these models has also been analyzed rigorously in the probability literature. For example,~\cite{DMSS2014} proves that a first-moment quantity asymptotically agrees with the typical free energy for ferromagnetic Potts models over sparse graph sequences, justifying the `Bethe ansatz' for these models.

Whereas Bethe entropy can be understood as a functional of a probability distribution over a tree, the extension to Kikuchi entropy allows an underlying graph that is a hyper-tree, which is a hyper-graph $G=(V,E)$ such that there is some tree with vertex set $V$ whose subgraphs induced by hyper-edges of $G$ are all connected.


For the first-moment calculations we need below, the exponent is given by the analogous annealed entropy for an action of $\Gamma$, the $d$-fold free power of $\ZZ_k$, rather than a free group.  It turns out that this could again be defined as an infimum of a more elementary quantity over Markov approximations, where now the more elementary quantity is the Kikuchi entropy associated to the Cayley hyper-tree of $\Gamma$. We do not work this out completely here. Instead, we give a more direct formula for the annealed entropy and only show that it is bounded above by the Kikuchi entropy (see discussion following Prop.~\ref{prop:kikformula}).



\section{Completely positive entropy, local convergence, and model mixing}\label{S:cpeandmm}

In order to prove the completely positive entropy (CPE) statement in Theorem \ref{mainthm3}, we will use a variant of the main result of \cite{austinburton2019}. That paper proves that if a model measure sequence locally and empirically converges to the target measure and is uniformly model mixing then the system is CPE. The local and empirical convergence result needed to prove CPE will also help us establish the lower bound in Theorem \ref{mainthm3}, part (a). We review these concepts here.

As in previous sections, let $\A$ be a finite set, $\Gamma$ a countable group and $\mu$ be a $\Gamma$-invariant Borel probability measure on $\A^\Gamma$. We also let $\Sigma=(\sigma_n)_n$ be a sofic approximation, where $\sigma_n:\Gamma \to \Sym(V_n)$ for each $n$. 

Given a probability measure $\kappa$ on $\A^{V_n}$ and a vertex $v \in V_n$, the {\bf localization of $\kappa$ at $v$} is the probability measure
$$\Loc{\kappa}{v} = (\Pi_v^{\sigma_n})_* \kappa = \sum_{\mb{x} \in \A^{V_n} }  \kappa(\mb{x}) \d_{\Pi^{\sigma_n}_v(\mb{x})} \in \Prob(\A^\Gamma).$$
This is the law of the pull-back name of a $\kappa$-random sample, as viewed from a fixed $v\in V_n$. This measure depends on the homomorphism $\s_n$, but we will usually leave that dependence implicit. If we  want to specify $\sigma_n$, we use the notation $\Loc[\sigma_n]{\kappa}{v}$.

A {\bf model measure sequence} is a sequence $(\mu_n)_n$ of probability measures $\mu_n$ on $\A^{V_n}$. The sequence $(\mu_n)_n$ is said to converge to $\mu$ {\bf locally and empirically} if for every open neighborhood $\mathcal{O}$ of $\mu$ in $\Prob(\A^\Gamma)$,
\begin{eqnarray*}
1 &=& \lim_{n\to\infty} |V_n|^{-1} \, \abs{\{v\in V_n:~ \Loc{\mu_{n}}{v} \in \mathcal{O}\}} \\
1 &=& \lim_{n\to\infty} \mu_n(\{\mb{x} \in \A^{V_n}:~ P_{\mb{x}}^{\sigma_n} \in \mathcal{O}\}).
\end{eqnarray*}
Below, we will sometimes refer to the first equality holding for every $\calO$ as \emph{local convergence} and the second as \emph{empirical convergence} in order to be explicit about which property is relevant.

If the measures $\mu_n$ and/or the maps $\sigma_n$ are random with law $\PP_n$, then we say the sequence converges locally and empirically in probability to $\mu$ if the same limits hold in probability. Explicitly, for every open neighborhood $\calO$ of $\mu$ and every $\varepsilon>0$,
\begin{eqnarray*}
1 &=& \lim_{n\to\infty} \PP_n\left\{ |V_n|^{-1} \, \abs{\{v\in V_n:~ \Loc{\mu_{n}}{v} \in \mathcal{O}\}} > 1 - \varepsilon \right\} \\
1 &=& \lim_{n\to\infty} \PP_n\left\{ \mu_n(\{\mb{x} \in \A^{V_n}:~ P_{\mb{x}}^{\sigma_n} \in \mathcal{O}\}) > 1 - \varepsilon \right\}.
\end{eqnarray*}

It will be convenient to reformulate local convergence in probability in terms of total variation distance between marginals. To make this precise, we need notation for the marginals. 

Given a finite set $B \subset \Gamma$ and a probability measure $\nu$ on $\A^\Gamma$, let $\nu_B$ be the probability measure on $\A^B$ equal to the pushforward of $\nu$ under the projection map $\A^\Gamma \to \A^B$. This is the marginal of $\nu$ on $B$.

Let $d_{\TV}$ denote total variation distance. Because the sets of the form $\calO(B,\varepsilon,\mu)=\{\nu \in \Prob(\A^\G):~ d_{\TV}(\nu_B, \mu_B) < \varepsilon\}$ form a neighborhood basis for the topology at $\mu$, it follows that a sequence of random measures $\mu_n \in \Prob(\A^{V_n})$ converges locally in probability to a measure $\mu \in \Prob(\A^\Gamma)$ if and only if for every finite $B \subset \Gamma$ and $\varepsilon>0$,
		\begin{eqnarray}\label{E:localconvergence}
  \lim_{n\to\infty} \PP_n \left\{ \frac{1}{\abs{V_n}} \abs*{\left\{ v \in V_n \st d_{\TV} \big( \Loc{\mu_n}{v}_{B},\, \mu_{B} \big) > \varepsilon \right\}} > \varepsilon \right\} = 0.
  \end{eqnarray}

Versions of the next lemma have appeared several times before: for instance, inside the proof of \cite[Theorem 4.1]{bowen-expansive}, or explicitly as \cite[Lemma 5.4]{hayes-fuglede} or \cite[Corollary 5.7]{austin2016}. We include a proof for completeness.


\begin{lemma}
\label{lem:lepergodic}
    If a sequence of random measures $(\mu_n)_n$ converges locally in probability to an ergodic measure $\mu \in \Prob^\Gamma(\A^\Gamma)$ over some random sequence of homomorphisms, then it converges locally and empirically in probability to~$\mu$.
\end{lemma}
\begin{proof}
    For each $n$, let $\theta_n \in \Prob(\Prob^\Gamma(\A^\Gamma))$ denote the law of 
        \[ \frac{1}{\abs{V_n}} \sum_{v \in V_n} \Loc[\sigma_n]{\mu_n}{v} , \]
    where $(\sigma_n, \mu_n)$ are jointly distributed as given. As stated, $\theta_n$ is supported on $\Gamma$-imvariant measures because each $\s_n$ is a homomorphism and therefore the empirical measure $P^{\s_n}_{\mathbf{x}}$ is invariant, for any ${\bf x} \in \A^{V_n}$. 
    
    Passing to a subsequential limit if necessary, the sequence $(\theta_n)_n$ converges weakly to some $\theta \in \Prob(\Prob^\Gamma(\A^\Gamma))$. We first show the barycenter of $\theta$ must be $\mu$: given a continuous function $g \in C(\A^\Gamma)$
    \begin{align*}
        \iint g(\mb{z}) \, \nu(d\mb{z})\, \theta(d\nu)
            &= \lim_{n \to \infty} \iint g(\mb{z}) \, \nu(d\mb{z})\, \theta_n(d\nu) \\
            &= \lim_{n \to \infty} \EE\left[ \frac{1}{|V_n|} \sum_{v \in V_n} \int g(\Pi_v^{\sigma_n}\mb{x}) \, \mu_n(d\mb{x}) \right] .
    \end{align*}
    Now given $\varepsilon>0$, let $\calO \ni \mu$ be an open neighborhood such that if $\nu \in \calO$ then $\int g\, d\nu$ is within $\varepsilon$ of $\int g\, d\mu$. Then we control the expectation above by dividing up the terms based on whether $\mu_n$ looks like $\mu$ near $v$:
    \begin{align*}
        \EE\left[ \frac{1}{|V_n|} \sum_{v \in V_n} \int g(\Pi_v^{\sigma_n}\mb{x}) \, \mu_n(d\mb{x}) \right]
            &= \frac{1}{|V_n|} \EE\left[ \sum_{\substack{v \in V_n \\ \Loc{\mu_n}{v} \in \calO}} \int g(\Pi_v^{\sigma_n}\mb{x}) \, \mu_n(d\mb{x}) \right] \\
            &\qquad + \frac{1}{|V_n|} \EE\left[ \sum_{\substack{v \in V_n \\ \Loc{\mu_n}{v} \not\in \calO}} \int g(\Pi_v^{\sigma_n}\mb{x}) \, \mu_n(d\mb{x}) \right] .
    \end{align*}
    The magnitude of the second term is bounded by
        \[ \max \abs{g} \cdot \EE \left[ \frac{1}{|V_n|} \abs{\{v \in V_n : \Loc{\mu_n}{v} \not\in \calO \}}\right] \]
    which goes to 0 as $n$ goes to infinity by definition of local convergence in probability. By choice of $\calO$, the first term is within $\varepsilon$ of $\int g\, d\mu$ for large $n$. Since $\varepsilon>0$ was arbitrary, this proves that the (subsequential) limit must have barycenter $\mu$.
    
    Since $m$ is ergodic, the only possible subsequential limit with barycenter $\mu$ is $\delta_\mu$, so this is the true limit. This implies that for any $\varepsilon > 0$ and open $\calO \ni \mu$,
        \[ \PP \big\{ \mu_n(\Omega(\sigma_n, \calO)) > 1-\varepsilon \big\} \to 1 . \]
    This is because
       \[ \PP \big\{ 1-\mu_n(\Omega(\sigma_n, \calO)) > \varepsilon \big\} \leq \frac{\EE\big[1-\mu_n(\Omega(\sigma_n, \calO))]}{\varepsilon} = \frac{1 - \EE[\EE_{\mb{x} \sim \mu_n}[\1_{P_{\mb{x}}^{\sigma_n} \in \calO}]]}{\varepsilon} = \frac{1 - \theta_n(\calO)}{\varepsilon} \to 0\]
    using Markov's inequality, the tower law of expectation, and the portmanteau theorem.
\end{proof}

\subsection{Property M}

To define notions of model mixing, we will impose distance functions on the finite sets $V$ which form a given sofic approximation. For this purpose, we will assume $\Gamma$ is finitely generated and let $E \subset \Gamma$ be a finite symmetric generating subset. For $g\in \Gamma$, let $|g|$ be the word-length of $g$, which is the length of the shortest word in $E$ representing $g$. Given $\sigma:\Gamma \to \Sym(V)$, define distance in $V$ by
$$d_\sigma(v,w) = \min\{ |g|:~ g \in \Gamma, \sigma(g)v=w\}.$$
If there does not exist $g$ with $\sigma(g)v=w$ then we set $d_\sigma(v,w)=+\infty$. If $\sigma$ is not a homomorphism then $d_\sigma$ may fail to satisfy the triangle inequality.

A subset $S \subset V$ is {\bf $r$-separated} if $d_\sigma(v,w) > r$ for every pair of distinct $v,w \in S$. 

Suppose a model measure sequence $(\mu_n)_n$ converges locally and empirically to $\mu$. 
We say the sequence is {\bf uniformly model mixing (umm)} if for every finite $F \subset \Gamma$ and every $\eps>0$ there is some $r<\infty$ and a sequence of finite subsets $W_n \subset V_n$ such that 
$$|W_n| = (1-o(1))|V_n|$$
and if $S \subset W_n$ is $r$-separated then 
$$\shent( (\mu_n)_{\sigma_n^{F}(S)} ) \ge |S| (\shent(\mu_F) - \eps)$$
where 
\begin{itemize}
\item $\mu_F$ is the probability measure on $\A^F$ which is the pushforward of $\mu$ under the projection map $\A^\Gamma \to \A^F$;
\item $\sigma_n^F(S) = \{\sigma_n(f)s:~f \in F, s\in S\}$;
\item $ (\mu_n)_{\sigma_n^{F}(S)} $ is the probability measure on $\A^{\sigma_n^{F}(S)}$ which is the pushforward of $\mu_n$ under the projection map $\A^{V_n} \to \A^{\sigma_n^{F}(S)}$.
\end{itemize}

This is a microstates analog of uniform mixing, introduced by Rudolph and Weiss in~\cite{rudolphweiss} for actions of an amenable group; see also~\cite[Definition 10]{weiss-survey}, where the name `uniform mixing' appears for the first time. The main result of \cite{austinburton2019} is that if $(\mu_n)_n$ locally and empirically converges to $\mu$ and is uniformly model mixing with respect to a fixed sofic approximation $\Sigma$, then the system $(\A^\Gamma,\mu,T)$ has completely positive entropy with respect to $\Sigma$, in analogy with a corresponding result of~\cite{rudolphweiss}.

Unfortunately, we do not know whether the parity check sub-shifts of Theorem \ref{mainthm3} are uniformly model mixing. Instead we define a weaker version of model mixing which suffices. 

\begin{defn}[Property M]\label{dfn:M}

Suppose $(\mu_n)_n$ is a model measure sequence and $\mu \in \Prob^\Gamma(\A^\Gamma)$ is an invariant measure. We say the sequence has \textbf{property M} if for every $\eps>0$ and $0<r<\infty$ there is a sequence of subsets $S_n \subset V_n$ such that 
$$\liminf \frac{|S_n|}{|V_n|} >0$$
and
\begin{equation}\label{eq:high-ent}
\shent( (\mu_n)_{\sigma_n^{B_r}(S_n)} ) \ge |S_n| (\shent(\mu_{B_r}) - \eps)
\end{equation}
for all $n$, where $B_r = \ball{r}{e} \subset \Gamma$ denotes the ball of radius $r$ centered at the identity. In applications, $\mu$ will be a limit of the sequence $(\mu_n)_n$ but we do not impose any such requirement for the definition.
\end{defn}


In contrast with uniform model-mixing, we only require $S_n$ to have asymptotically positive density in $V_n$, and there is no uniform lower bound on this density across different choices of $\epsilon, r$. This density could be much smaller than $|B_r|^{-1}$, for example. We also do not require the sets $S_n$ to be separated, although the lower bound~\eqref{eq:high-ent} does usually imply a kind of approximate separation anyway.

We suspect that other variants of uniform model mixing could be used in a similar way to prove completely positive entropy, so Definition~\ref{dfn:M} is not an attempt at optimal generality.  This is why we have chosen a rather bland name for Property M, although it is convenient in our work below.

Here is the main result of this section:

\begin{theorem}\label{T:wmmcpe}
As above, let $\mu$ be a $\Gamma$-invariant probability measure on $\A^\Gamma$, $\Sigma=(\sigma_n:\Gamma \to \Sym(V_n))_n$ a sofic approximation, and $(\mu_n)_n$ a model measure sequence. Assume $(\mu_n)_n$ converges to $\mu$ locally and empirically along $\Sigma$ and has property M. Then every nontrivial factor of $(\A^\Gamma,\mu,T)$ inherits positive $\Sigma$-entropy.
\end{theorem}

\begin{proof}
By Corollary~\ref{cor:presence-monotone} it suffices to consider factor maps to other shift systems.  For these the proof is based on the proof of~\cite[Theorem 1.2]{austinburton2019}.

\vspace{7pt}

\emph{Step 1.}\quad We start by establishing some general entropy inequalities.  For two or more jointly distributed random variables $X_1, \ldots, X_k$, define the total correlation
	\[ \TC(X_1; \cdots; X_k) = \left(\sum_{i=1}^k \shent(X_i)\right) - \shent(X_1, \ldots, X_k) . \]
This is a generalization of mutual information to more than two random variables, introduced in \cite{watanabe1960}. It can also be recursively defined by setting $\TC(X_1; X_2) = \info(X_1; X_2)$ and for $k \geq 3$
	\[ \TC(X_1; \cdots; X_k) = \TC(X_1; \cdots; X_{k-1}) + \info(X_1, \ldots, X_{k-1}; X_k) . \]
Using this recursion and the data-processing inequality \cite[Theorem~2.8.1]{coverthomas2006}, it can be shown by induction on $k$ that if $f$ is any function and $Y_i = f(X_i)$ for each $i$ then
	\begin{equation}
    \label{ineq:dataprocessing}  
	   \TC(Y_1; \cdots; Y_k) \leq \TC(X_1; \cdots; X_k) .
    \end{equation}
This inequality has also appeared in \cite[Lemma~4.3]{austin2020}.
Note that the total correlation does not depend on the order in which the random variables are listed. Below, we will refer to the total correlation of a collection of random variables $\{X_i \st i \in I\}$ indexed by a finite set $I$ using the notation $\TC(\{ X_i \st i \in I\})$, since fixing an ordering would unnecessarily complicate notation.

The {\bf Rokhlin distance} between random variables $\alpha,\beta$ which are defined on the same probability space is defined by $\dee^{\Rok}_\mu(\alpha, \beta) = \shent_\mu(\alpha | \beta) + \shent_\mu(\beta | \alpha)$. This satisfies the triangle inequality, and it equals zero if and only if $\alpha$ and $\beta$ generate the same partition up to null sets. This distance can be used to control total correlation via the bound
\begin{align*}
    &\hspace{-1in} \abs*{\TC(X_1; \cdots; X_k) - \TC(Y_1; \cdots; Y_k)} \\
        &\leq \abs*{\shent(X_1, \ldots, X_k)- \shent(Y_1, \ldots, Y_k)}
            + \sum_{i=1}^k \abs*{\shent(X_i) - \shent(Y_i)} \\
        &\leq 2 \sum_{i=1}^k \big( \shent(X_i | Y_i) + \shent(Y_i | X_i) \big) = 2 \sum_{i=1}^k \dee^{\Rok}(X_i,Y_i).
\end{align*}

\vspace{7pt}

\emph{Step 2.}\quad Since $(\mu_n)_n$ locally and empirically converges to $\mu$, if $S_n \subset V_n$ satisfies $\liminf \frac{\abs{S_n}}{\abs{V_n}} > 0$ then
	\[ \frac{1}{\abs{S_n}} \sum_{v \in S_n} \shent \big( (\mu_n)_{\sigma_n^{B_r}(v)} \big) = \shent(\mu_{B_r}) + o(1). \]
So property M implies that for every $r,\varepsilon >0$ there is a sequence of subsets $S_n \subset V_n$ such that $\liminf \frac{\abs{S_n}}{\abs{V_n}} > 0$ and
	\[ \frac{1}{\abs{S_n}} \TC\big( \{ (\mb{x}_n)^{\sigma_n^{B_r}(v)} \st v \in S_n \} \big) \leq \varepsilon \]
for all large $n$, where $\mb{x}_n$ is a random sample of $\mu_n$ and the projections $\{ (\mb{x}_n)^{\sigma_n^{B_r}(v)} \st v \in S_n \}$ are jointly distributed in the natural way (from a common sample of $\mb{x}_n$).

\vspace{7pt}

\emph{Step 3.}\quad Now let $\phi \colon \A^\Gamma \to \B$ be a measurable map into a finite set $\B$ generating a factor map $\phi^\Gamma$ as in~\eqref{eq:map-to-factor}. If this factor is nontrivial then $\shent_\mu(\phi) > 0$. We want to show that the property M assumption on $\mu_n$ implies that
\[h_\Sigma(\mu\,;\,\phi^\Gamma) > 0.\]
Let $\lambda$ be the graphical joining of the factor map $\phi^\Gamma$ as in~\eqref{eq:graphical}.

Fix $r \in \NN$ and a $\ball{r}{e}$-local function $\psi \colon \A^\Gamma \to \B$ which approximates $\phi$ closely enough in measure that $\dee^{\Rok}_\mu(\psi, \phi) < \frac{1}{8}\shent_\mu(\phi)$. 

Now with $\varepsilon = \frac{1}{8}\shent_\mu(\phi)$ and this $r$, let $(S_n)_n$ be the sequence of subsets of $V_n$ given by property M. Since $\psi$ is $\ball{r}{e}$-local, the data-processing inequality (\ref{ineq:dataprocessing}) above implies that
	\[ \TC\big( \{ {\psi^{\sigma_n}}(\mb{x}_n)(v) \st v \in S_n \} \big) \leq \TC\big( \{ (\mb{x}_n)^{\sigma_n^{B_r}(v)} \st v \in S_n \} \big) \leq \varepsilon \abs{S_n} \]
where ${\bf x}_n$ is a random sample of $\mu_n$ and $\psi^{\sigma_n} \colon \A^{V_n} \to \B^{V_n}$ is defined as in~\eqref{eqn:microstatemap}.

\vspace{7pt}

\emph{Step 4.}\quad Let $(\psi_m)_m$ be a local approximating sequence sequence to $\phi$, meaning that~\eqref{eq:local-approx} holds and hence $\dee^{\Rok}_\mu(\phi, \psi_m)$ also converges to 0. Since the Rokhlin distance satisfies the triangle inequality, there is some $M \in \NN$ such that if $m \geq M$ then $\dee^{\Rok}_\mu(\psi, \psi_m) < \frac{1}{8} \shent_\mu(\phi)$.

Given an open neighborhood $\calO \ni \lambda$, by \cite[Prop.~4.10]{austin2016} there is some open neighborhood $\calU \ni \mu$ and some $m \geq M$ such that, for all large enough $n$, the map $(\mathrm{id}_{\A^{V_n}},\psi_m^{\sigma_n})$ sends $\calU$-microstates to $\calO$-microstates.
Let us also assume $m$ is large enough that $\shent_\mu(\psi_m) \geq \frac{1}{2} \shent_\mu(\phi)$.

Now fix some $R$ such that $\psi$ and $\psi_m$ are both $\ball{R}{e}$-local. By local convergence of $\mu_n$ to $\mu$, for any $\delta>0$, the fraction of $v \in S_n$ for which the local marginal $\Loc{\mu_n}{v}_{B_R}$ is within total variation distance $\delta$ of $\mu_{B_R}$ is $1-o(1)$. For the rest of the $v \in S_n$, the term in the sum below has the upper bound $2 \log \abs{\B}$:
    \begin{align*}
        &\hspace{-1in} \tfrac{1}{\abs{S_n}} \abs*{\TC\big( \{ {\psi^{\sigma_n}}(\mb{x}_n)(v) \st v \in S_n \} \big) - \TC\big( \{ {\psi_m^{\sigma_n}}(\mb{x}_n)(v) \st v \in S_n \} \big)} \\
            &\leq \frac{2}{\abs{S_n}} \sum_{v \in S_n} \left( \shent_{\Loc{\mu_n}{v}} (\psi | \psi_m) + \shent_{\Loc{\mu_n}{v}} (\psi_m | \psi) \right) \\
            &\leq 2 \dee^{\Rok}_\mu(\psi_{m},\psi) + o(1) .
    \end{align*}
Hence
\begin{equation}
\label{ineq:tcbound}
    \tfrac{1}{\abs{S_n}} \TC\big( \{ {\psi_m^{\sigma_n}}(\mb{x}_n)(v) \st v \in S_n \} \big) \leq  \frac{3}{8} \shent_\mu(\phi) + o(1) .
\end{equation}
	
By empirical convergence, for large $n$ the model measure $\mu_n$ is mostly supported on $\calU$-microstates. So $(\mathrm{id}_{\A^{V_n}},\psi_m^{\sigma_n})_*\mu_n$ is mostly supported on $\calO$-microstates, and using Fano's inequality we see that
	\[ \frac{1}{\abs{V_n}} \shent( {\psi_{m}^{\sigma_n}}_*\mu_n) \leq \frac{1}{\abs{V_n}} \log \abs{\mathrm{proj}_n[\Omega(\sigma_n, \calO)]} + o(1) . \]
The total correlation bound (\ref{ineq:tcbound}) gives
	\begin{align*}
	\shent( {\psi_{m}^{\sigma_n}}_*\mu_n) =	\shent( \psi_{m}^{\sigma_n}({\bf x}_n))
			&\ge \shent( \{ (\psi_{m}^{\sigma_n}({\bf x}_n)(v) \st v \in S_n \}) \\
			&\geq \sum_{v \in S_n} \shent( (\psi_{m}^{\sigma_n}({\bf x}_n)(v) \}) -\abs{S_n} \big( \tfrac{3}{8} \shent_\mu(\phi) + o(1) \big).
	\end{align*}
	Since $(\mu_n)_n$ converges locally to $\mu$ and $\psi_m$ is a local function,
		\[ \frac{1}{\abs{S_n}} \sum_{v \in S_n} \shent( (\psi_{m}^{\sigma_n}({\bf x}_n)(v) ) = \shent_\mu(\psi_m) + o(1)
            \geq \tfrac{1}{2} \shent_\mu(\phi) + o(1). \]
	So
		\[ \frac{1}{\abs{V_n}} \log \abs{\mathrm{proj}_n[\Omega(\sigma_n, \calO)]} \geq \frac{\abs{S_n}}{\abs{V_n}} \left( \tfrac{1}{8}\shent_\mu(\phi) + o(1) \right) + o(1) , \]
    and for every $\calO \ni \lambda$
        \[ \liminf_{n \to \infty} \frac{1}{\abs{V_n}} \log \abs{\mathrm{proj}_n[\Omega(\sigma_n, \calO)]} \geq \left(\liminf_{n \to \infty} \frac{\abs{S_n}}{\abs{V_n}}\right) \tfrac{1}{8} \shent_\mu(\phi) . \]
    Since we chose $(S_n)_n$ independently of $\calO$, and $\liminf_{n \to \infty} \frac{\abs{S_n}}{\abs{V_n}} > 0$, taking the infimum over $\calO$ completes the proof.
\end{proof}

\section{Shattering}\label{S:shattering}

Let $(\A^\Gamma,\mu,T)$ be a shift $\Gamma$-system and $\Sigma=(\sigma_n)_n$ a sofic approximation, where ${\sigma_n:\Gamma \to \Sym(V_n)}$. On each model space $\A^{V_n}$ we have the normalized Hamming distance defined by
$$\dee^{(V_n)}(\mathbf{x},\mathbf{y})= |V_n|^{-1}|\{v\in V_n:~\mathbf{x}(v)\ne \mathbf{y}(v)\}|.$$

In this section we derive ergodic-theoretic consequences from the following phenomenon, which is at the heart of our study of parity check shifts.

\begin{defn}\label{defn:shatt}
The shift system has \textbf{totally shattered microstate spaces along $\Sigma$} if (i) it has microstates along $\Sigma$, and (ii) there exists a $\delta > 0$ for which the following holds. For every $\varepsilon > 0$ there exist a weak$^\ast$ neighbourhood $\calO$ of $\mu$ and a positive integer $n_0$ such that, for any $n\ge n_0$ and any two microstates $\mathbf{x},\mathbf{y} \in \Omega(\sigma_n,\calO)$, we have
$$\hbox{either} \quad \dee^{(V_n)}(\mathbf{x},\mathbf{y}) \ge \delta \quad \hbox{or}\quad \dee^{(V_n)}(\mathbf{x},\mathbf{y}) < \varepsilon.$$
We refer to any such $\delta$ as a \textbf{shatter distance} for the system along $\Sigma$.
\end{defn}





\subsection{Consequences of shattering}
\label{subs:shatt-cons}

In this section, fix a sofic approximation $\Sigma = (\sigma_n)_n$ by homomorphisms, and assume that $(\A^\Gamma,\mu,T)$ has totally shattered microstate spaces along $\Sigma$.

\begin{theorem}\label{T:shatt-cons}
For every $\varepsilon > 0$ there is a neighbourhood $\calO$ of $\mu$ for which the following holds.  Let $(\B^\Gamma,\nu,T)$ be another shift system that has microstates along $\Sigma$, let $(\tK^\Gamma, \k^\Gamma,T)$ be a Bernoulli shift, and let
\[\phi:\B^\Gamma \times \tK^\Gamma \to \A\]
be a measurable map. If
\begin{equation}\label{eq:nearly-factor}
\phi^\G_\ast (\nu\times \k^\Gamma) \in \calO
\end{equation}
then
\begin{equation}\label{eq:near-indep}
(\nu\times \k^\Gamma\times \k^\Gamma)\{(y,z,z'):\ \phi(y,z) \ne \phi(y,z')\} < \varepsilon.
\end{equation}
\end{theorem}

Intuitively, the conclusion is that if $\phi^\G$ is approximately a factor map onto $(\A^\Gamma,\mu,T)$, then it must be approximately independent from the second coordinate in $\B^\Gamma \times \tK^\Gamma$.  There are many ways to capture the latter assertion precisely, but~\eqref{eq:near-indep} turns out to be convenient during the proof.

The proof of Theorem~\ref{T:shatt-cons} has much in common with the main proof in~\cite{MR3693964}.  In that paper, Gamarnik and Sudan used a relative of shattering to prove an a.a.s. upper bound on the maximum size of an independent set on a random regular graph that can be constructed using a local algorithm.  The property they use there is now called the `overlap gap property,' and is actually a little weaker than being totally shattered.  See~\cite{gamarnik-survey} for a recent survey.  The reference \cite[Section 4]{lyons2017} explains how the absence of an approximating local algorithm for certain combinatorial problems implies that a resulting limit process is not weakly contained in a Bernoulli shift.

Our proof of Theorem~\ref{T:shatt-cons} needs a couple of known facts about microstate spaces for a product in which one factor is Bernoulli. We recall these as separate lemmas before starting the proof.

\begin{lemma}\label{lem:bern-fibres-good}
    Let $(\B^\G,\nu,T)$ be a shift system and let $(\tL^\G,\lambda^\G,T)$ a Bernoulli shift.  Assume that $\mb{y}_n \in \B^{V_n}$ is a sequence such that $P^{\s_n}_{\mb{y}_n} \to \nu$.  Then
    \[\lambda^{V_n}\big\{\mb{z}:\ P^{\s_n}_{(\mb{y}_n,\mb{z})}\in \calO\big\} \to 1\]
    for every neighbourhood $\calO$ of $\nu \times \lambda^\G$.     \hfill \qedsymbol
\end{lemma}

Lemma~\ref{lem:bern-fibres-good} is well-known as folklore in the study of sofic entropy, and a full proof can be found inside the proofs of some of its consequences in the literature.  The earliest and perhaps easiest to extract is inside the proof of the lower bound in~\cite[Theorem 8.1]{bowen-jams-2010}.

The next lemma is more specialized, but was also used in the second author's previous counterexample to the weak Pinsker conjecture for some sofic groups: it is a special case of~\cite[Proposition 7.9]{bowen2022}.  It refers to `hereditary' neighbourhoods of a shift-invariant measure.  We do not repeat the definition of these here; the only property we need is that they form a basis for the weak$^\ast$ topology.

\begin{lemma}\label{lem:bernconnected}
	Let $\calU \subset \calV$ be open neighborhoods of $\nu \times \kappa^\Gamma$ with $\calU$ hereditary and $\calV$ containing the closure of $\calU$. For every $\delta>0$, if $n$ is large enough, then for every $\mb{z}, \mb{z}' \in \tK^{V_n}$ and $\mb{y} \in \B^{V_n}$, if $(\mb{y},\mb{z}),\, (\mb{y}, \mb{z}')$ are in $\Omega(\sigma_n, \calU)$ then they are $\delta$-connected within $\Omega(\sigma_n,\calV)$ (this means there are $\mb{w}_1,\ldots, \mb{w}_k\in \Omega(\sigma_n,\calV)$ with $\mb{w}_1=(\mb{y},\mb{z})$,  $\mb{w}_k=(\mb{y},\mb{z}')$ and $d^{(V_n)}(\mb{w}_i, \mb{w}_{i+1})<\d$ for all $i$). \hfill \qedsymbol
\end{lemma}

\begin{proof}[Proof of Theorem~\ref{T:shatt-cons}]
Let $\delta$ be a shatter distance for the system along $\Sigma$.  Let $\varepsilon > 0$ be small enough that $2\varepsilon < \delta$, and now let $\calO_1$ and $n_0$ be a neighbourhood and positive integer as promised by Definition~\ref{defn:shatt} for this choice of $\varepsilon$.  Lastly choose a smaller neighbourhood $\calO$ of $\mu$ whose closure is contained in $\calO_1$.

Let $\phi:\B^\G\times \tK^\G\to \A$ be a measurable map such that $\phi_\ast^\G(\nu\times \k^\G) \in \calO$.  In the rest of the proof we show that~\eqref{eq:near-indep} holds for this $\calO$ and with $3\varepsilon$ in place of $\varepsilon$.

\vspace{7pt}

\emph{Step 1.}\quad Since $\A$ is finite and $\phi$ is measurable, for any $\eta > 0$ there is an approximating map
\[\psi:\B^\Gamma\times \tK^\Gamma \to \A\]
such that
\begin{equation}\label{eq:psi-phi}
    (\nu\times \k^\Gamma)\{\psi \ne \phi\} < \eta
\end{equation}
and such that (i) $\psi$ is a local map and (ii) $\psi$ depends on the coordinates in $\tK^\Gamma$ only through some finite measurable partition $\calP$ of $\tK$.  If we choose $\eta$ sufficiently small in terms of $\calO\subset \calO_1$ and $\varepsilon$, then~\eqref{eq:nearly-factor} and~\eqref{eq:psi-phi} imply that $\psi^\G_\ast(\nu\times \k^\Gamma)$ still lies in $\calO$, and also~\eqref{eq:psi-phi} implies that the desired conclusion~\eqref{eq:near-indep} holds for $\phi$ with error $3\varepsilon$ if it holds for $\psi$ with error $2\varepsilon$.

By replacing $\tK$ with the set of cells $\calP$, we have therefore reduced our work to the case when $\tK$ is finite and $\phi$ is $F$-local for some finite subset $F$ of $\G$. We assume this for the rest of the proof, and do not refer to $\psi$ again.

Having made these assumptions, let us note that the set
\[\Delta:= \{(y,z,z'):\ \phi(y,z) \ne \phi(y,z')\}\]
is closed and open in $\B^\G \times \tK^\G\times \tK^\G$.

\vspace{7pt}

\emph{Step 2.}\quad Since $\phi$ is local, pushing forward by the equivariant map $\phi^\Gamma$ acts continuously on probability measures.  We may therefore choose a neighbourhood $\calV$ of $\nu \times \k^\G$ such that $\phi^\G_\ast[\calV] \subset \calO$.  Since $\sigma_n$ is a homomorphism we have $P_{\phi^{\sigma_n}(\mb{x})}^{\sigma_n} = \phi^\Gamma_* P_{\mb{x}}^{\sigma_n}$, so this implies that
\begin{equation}\label{eq:inclusion}
\phi^{\s_n}[\Omega(\s_n,\calV)] \subset \Omega(\s_n,\calO) \quad \hbox{for every}\ n.
\end{equation}

In addition, since hereditary neighbourhoods form a basis, we may let $\calU$ be a hereditary neighbourhood of $\nu \times \k^\G$ whose closure is contained in $\calV$.

\vspace{7pt}

\emph{Step 3.}\quad By assumption, there is a sequence $\mb{y}_n \in \B^{V_n}$ such that $P^{\s_n}_{\mb{y}_n} \to \nu$.  Fix this for the rest of the proof.

For each $n$, consider $\mb{z}_n$ and $\mb{z}_n' \in \tK^{V_n}$ drawn independently at random according to $\k^{V_n}$, and consider the event
\[E := \big\{(\mb{z}_n,\mb{z}_n') \in \tK^{V_n}\times \tK^{V_n}:\ (\mb{y}_n,\mb{z}_n),(\mb{y}_n,\mb{z}_n')\in \Omega(\s_n,\calU)\big\}.\]
By Lemma~\ref{lem:bern-fibres-good}, we have
\begin{equation}\label{eq:prob-to-1}
(\k^{V_n}\times \k^{V_n})(E) \to 1.   
\end{equation}
When $E$ occurs, we can draw the following additional conclusions:
\begin{itemize}
\item The points $(\mb{y}_n,\mb{z}_n)$ and $(\mb{y}_n,\mb{z}_n')$ are $(\delta/|F|)$-connected within $\Omega(\sigma_n,\calV)$ for all sufficiently large $n$ (not depending on the specific values of $\mb{y}_n$, $\mb{z}_n$ or $\mb{z}_n'$), by Lemma~\ref{lem:bernconnected}.
\item Since $\phi$ is $F$-local, the map $\phi^{\s_n}$ is $|F|$-Lipschitz for the normalized Hamming metrics (Lemma~\ref{lem:localmaplipschitz}), so $\phi^{\s_n}(\mb{y}_n,\mb{z}_n)$ and $\phi^{\s_n}(\mb{y}_n,\mb{z}_n')$ both lie in $\Omega(\s_n,\calO)$ and are $\delta$-connected within that set, by the previous conclusion and~\eqref{eq:inclusion}.
\item By total shattering and our choice of $\calO$, we can now deduce that
\begin{equation}\label{eq:less-than-eps}
\dee^{(V_n)}\big(\phi^{\s_n}(\mb{y}_n,\mb{z}_n),\phi^{\s_n}(\mb{y}_n,\mb{z}_n')\big) < \varepsilon
\end{equation}
for all sufficiently large $n$.  Indeed, if $n$ is large enough and
\[\phi^{\s_n}(\mb{y}_n,\mb{z}_n) = \mb{w}_1,\ \dots,\ \mb{w}_l = \phi^{\s_n}(\mb{y}_n,\mb{z}_n')\]
is a sequence in $\Omega(\sigma_n,\calO)$ with all consecutive distances less than $\delta$, then total shattering implies that all of these distances are actually less than $\varepsilon$.  Then the triangle inequality implies that $\dee^{(V_n)}(\mb{w}_1,\mb{w}_3) < 2\varepsilon$, which is still less than $\delta$.  We may therefore invoke total shattering again to conclude that in fact $\dee^{(V_n)}(\mb{w}_1,\mb{w}_3) < \varepsilon$.  Now a simple induction shows that in fact $\dee^{(V_n)}(\mb{w}_1,\mb{w}_i) < \varepsilon$ for every $i$, giving~\eqref{eq:less-than-eps} when $i=l$.
\end{itemize}

Unpacking the definitions of normalized Hamming metric and empirical distribution, the left-hand side of~\eqref{eq:less-than-eps} is equal to
\[\frac{1}{|V_n|}\sum_{v \in V_n} 1_{\left\{\phi(\Pi^{\s_n}_v\mb{y}_n,\Pi^{\s_n}_v\mb{z}_n)\ne \phi(\Pi^{\s_n}_v\mb{y}_n,\Pi^{\s_n}_v\mb{z}_n')\right\}} = P^{\s_n}_{(\mb{y}_n,\mb{z}_n,\mb{z}_n')}(\Delta).\]
Therefore, in view of the conclusions above,~\eqref{eq:prob-to-1} implies that
\begin{equation}\label{eq:prob-to-1-2}
(\k^{V_n}\times \k^{V_n})\big\{P^{\s_n}_{(\mb{y}_n,\mb{z}_n,\mb{z}_n')}(\Delta) < \varepsilon\big\} \to 1.
\end{equation}

On the other hand, since $\Delta$ is a closed and open set, another appeal to Lemma~\ref{lem:bern-fibres-good} (this time for the larger product $\nu\times \k^\G\times \k^\G$) shows that
\begin{equation}\label{eq:prob-to-1-3}
(\k^{V_n}\times \k^{V_n}) \big\{P^{\s_n}_{(\mb{y}_n,\mb{z}_n,\mb{z}_n')}(\Delta) > (\nu\times \k^\G\times \k^\G)(\Delta) - \eps\big\} \to 1.
\end{equation}
The limits~\eqref{eq:prob-to-1-2} and~\eqref{eq:prob-to-1-3} can hold simultaneously only if
\[(\nu \times \k^\G\times \k^\G)(\Delta) < 2\varepsilon.\]
Since $\varepsilon$ was arbitrary, this completes the proof.
\end{proof}

Theorem~\ref{T:shatt-cons} has several more streamlined consequences.

\begin{cor}
\label{cor:shatt-cons}
Suppose that $(\A^\G,\mu,T)$ has totally shattered microstate spaces along $\Sigma$.
\begin{enumerate}
    \item If $(Y,\theta,S)$ is any probability-preserving $\G$-system that has microstates along $\Sigma$ (recall the definition of this property from Subsection~\ref{S:prelim-entropy}), and $(\tK^\G,\k^\G,T)$ is a Bernoulli shift, then any factor map
    \[(Y\times \tK^\G,\theta\times \k^\G,S\times T)\to (\A^\G,\mu,T)\]
    factorizes through the coordinate projection $Y \times \tK^\G \to Y$ up to agreement a.e.
    \item The system $(\A^\G,\mu,T)$ has no non-trivial direct Bernoulli factors.
    \item The system $(\A^\G,\mu,T)$ is not weakly contained in a Bernoulli shift unless it is actually a trivial system, meaning that $\mu = \delta_{(\dots,a,a,\dots)}$ for some $a \in \A$.
\end{enumerate}
\end{cor}

\begin{proof}
    \emph{Part 1.}\quad Let the factor map in question be $\xi^\Gamma$ for some measurable map ${\xi:Y\times \tK^\G \to \A}$.  Let $\varepsilon > 0$, and let $\calO$ be the neighbourhood of $\mu$ given by Theorem~\ref{T:shatt-cons} for this $\varepsilon$.

    By approximating the level sets of $\xi$ by measurable rectangles, for any $\eta > 0$ we can choose (i) a factor map $\pi$ from $(Y,\theta,S)$ to a finite-alphabet shift system $(\B^\Gamma,\nu,T)$ and (ii) a map $\phi:\B^\Gamma\times \tK^\Gamma \to \A$ such that
\begin{equation}\label{eq:choice-of-eta}
(\theta\times \kappa^\Gamma)\{(y,z):\ \phi(\pi(y),z) \ne \xi(y,z)\} < \eta.
\end{equation}
Provided we choose $\eta$ small enough, the inequality~\eqref{eq:choice-of-eta} implies that
$$\phi^\Gamma_\ast(\nu\times \kappa^\Gamma) \in \calO.$$
By shrinking further if necessary, we may also assume that $\eta < \varepsilon$.

Fix a choice of $\eta$ with these properties, and consider the resulting maps $\pi$ and $\phi$.  Since $(Y,\theta,S)$ has microstates along $\Sigma$, so does its factor $(\B^\G,\nu,T)$.  By our choice of $\eta$ we have arranged that this factor system and the map $\phi$ satisfy~\eqref{eq:nearly-factor}.  Therefore Theorem~\ref{T:shatt-cons} tells us that they also satisfy~\eqref{eq:near-indep}.  Combined with~\eqref{eq:choice-of-eta}, this gives
    \begin{multline*}
    (\theta\times \kappa^\Gamma\times \kappa^\Gamma)\{(y,z,z'):\ \xi(y,z)\neq \xi(y,z')\} \\ \le 2\eta + (\nu\times \kappa^\Gamma\times \kappa^\Gamma)\{(y_1,z,z'):\ \phi(y_1,z)\neq \phi(y_1,z')\}
    < 2\eta + \varepsilon < 3\varepsilon.
    \end{multline*}
    Since $\varepsilon$ is arbitrary, the left-hand side here must actually equal $0$, so in fact
    \[\xi(y,z) = \xi(y,z') \quad \hbox{for}\ (\theta\times \k^\G\times \k^\G)\hbox{-a.e.}\ (y,z,z').\]
    Therefore, possibly after adjusting on a set of measure zero, $\xi$ depends on only the first coordinate in $Y\times \tK^\G$.

    \vspace{7pt}

    \emph{Part 2.}\quad If
    \[\xi^\Gamma:(Y,\theta,S)\times (\tK^\G,\k^\G,T)\to (\A^\G,\mu,T)\]
    is an isomorphism, then the system $(Y,\theta,S)$ is a factor of $(\A^\G,\mu,T)$, and so it inherits the property of having microstates along $\Sigma$.  Therefore Part 1 shows that, up to agreement a.e., $\xi$ depends on only the first coordinate in $Y\times \tK^\Gamma$.  Since $\xi^\Gamma$ is an isomorphism, this is possible only if the Bernoulli factor $(\tK^\Gamma,\kappa^\Gamma,T)$ is trivial.

    \vspace{7pt}

    \emph{Part 3.}\quad Towards a contradiction, suppose that $(\tK^\G,\k^\G,T)$ is a Bernoulli shift and that $\phi_m:\tK^\G \to \A$ is a sequence of measurable maps such that
    \begin{equation}\label{eq:weak-cont}
    \phi^\G_{m\ast}\k^\G \to \mu.
    \end{equation}
    Apply Theorem~\ref{T:shatt-cons} by inserting a trivial one-point system in the place of $(\B^\G,\nu,T)$.  The conclusion is that
    \[(\k^\G\times \k^\G)\{(z,z'):\ \phi_m(z) \ne \phi_m(z')\} \to 0.\]
    However, this implies that the distribution of $\phi_m$ is converging to $\delta_a$ for some $a\in \A$.  Combined with~\eqref{eq:weak-cont}, it follows that $\mu = \delta_{(\dots,a,a,\dots)}$.
    \end{proof}

\subsection{Shattering and Bernoulli splittings}



Suppose the shift system $(\A^\Gamma, \mu, T)$ has totally shattered microstate spaces along $\Sigma$. Note as soon as $\varepsilon < \delta/2$, if $\calO$ is the neighborhood of $\mu$ given by the definition then the relation ``$\dee^{(V_n)} < \delta$'' restricted to the microstate space $\Omega(\sigma_n, \calO)$ is an equivalence relation, so it partitions the microstate space into small-diameter, well-separated clusters.  

In this section, we give a second proof of Corollary~\ref{cor:shatt-cons}(2) and sketch a third. Both of these approaches are based on the clusters of microstates: one uses the number of clusters, which is one of the sofic homological invariants introduced in \cite{bowen2022}, and the other uses the sizes of clusters. Both can be compared to the use of the ``overlap gap property'' introduced in \cite{MR3693964} to prove an a.a.s. upper bound on the size of an independent set on a random regular graph, when the independent set is required to be constructed from a local algorithm. 

The size of clusters can be used as follows: if a direct Bernoulli factor exists, it can be shown that its entropy rate is a uniform lower bound for the exponential size of all microstate clusters, while if a system has totally shattered microstate spaces then all its clusters are of subexponential size. This implies Corollary~\ref{cor:shatt-cons}(2).

In the rest of the section, we give a proof using the number of clusters. First, we give relevant definitions.

If $\calO$ is a subset of a metric space and $x,y \in \calO$, we say $x,y$ are {\bf $\delta$-connected within $\calO$} if there is a sequence of points $z_1, \ldots, z_l \in \calO$ with $z_1 = x$, $z_l = y$, and $\dee(z_i, z_{i+1}) < \delta$ for each $i$. This defines an equivalence relation on $\calO$, which we denote $\cluster(\calO,\delta)$.

Given a map $\sigma \colon \Gamma \to \Sym(V)$, a labeling $\mb{x} \in \A^{V}$, an open set $\calO \subset \Prob(\A^\Gamma)$, and $\delta>0$, let
    \[ \cluster(\sigma, \mb{x}, \calO, \delta) = [x]_{\cluster(\Omega(\sigma,\calO),\delta)} = \{ \mb{z} \in \A^{V} \st \mb{z} \text{ is $\delta$-connected to $\mb{x}$ within } \Omega(\sigma,\calO) \} . \]
where $\Omega(\sigma,\calO) \subset \A^{V}$ has the normalized Hamming metric $\dee^{(V)}$. In particular, $\cluster(\sigma,\mb{x}, \calO, \delta) \subset \Omega(\sigma,\calO)$.

Now if in addition to the above we have some $\calO' \subset \calO$, the quotient
    \[ \bigslant{\Omega(\sigma_n, \calO')}{\cluster(\Omega(\sigma_n, \calO), \, \delta)} \]
is the set of clusters of $\calO'$-microstates that are $\delta$-connected within $\Omega(\sigma_n, \calO)$. We define
    \[ b_{0,\Sigma}(\mu) = \sup_{\calO \ni \mu} \sup_{\delta > 0} \inf_{\mu \in \calO' \subset \calO} \limsup_{n \to \infty} \frac{1}{\abs{V_n}} \log \abs*{\bigslant{\Omega(\sigma_n, \calO')}{\cluster(\Omega(\sigma_n, \calO), \, \delta)}} . \]
Informally, we first set coarseness parameters $\calO, \delta$ which divide microstate spaces into clusters. We consider the exponential growth rate of the number of these clusters which contain $\calO'$-good microstates for $\mu$. At first, it might seem more natural to consider directly the growth rate of the number of $\delta$-connected clusters within a single $\Omega(\sigma_n, \calO)$:
    \[ \limsup_{n \to \infty} \frac{1}{\abs{V_n}} \log \abs*{\bigslant{\Omega(\sigma_n, \calO)}{\cluster(\Omega(\sigma_n, \calO), \, \delta)}}, \]
then take $\delta$ to $0$ and in some sense $\calO$ to $\mu$. But there is no monotonicity in $\calO$: shrinking the neighborhood of $\mu$ removes some microstates from $\Omega(\sigma_n, \calO)$, but this can both remove some clusters and/or break a cluster into multiple pieces. Considering pairs $\calO' \subset \calO$ is one natural way around this. See also the discussion of a related definition of connected model spaces in \cite[Section~2.2]{austin2016a}.

In \cite{bowen2022}, $b_{0,\Sigma}(\mu)$ is called the $0$th Betti number of $\mu$. If $X$ is totally disconnected and $\mu \in \Prob^\Gamma(X^\Gamma)$ then $b_{0,\Sigma}(\mu)$ is a measure-conjugacy invariant \cite[Corollary~4.2]{bowen2022}.

It follows directly from the definition that $b_{0,\Sigma}(\mu) \le \h_\Sigma(\mu)$ (a similar inequality holds for higher-dimensional sofic homology theories \cite[Lemma~7.13]{bowen2022}).


\begin{lemma}
\label{lem:shatteredbetti}
    If $(\A^\Gamma, \mu)$ has totally shattered microstate spaces over $\Sigma$, then $b_{0,\Sigma}(\mu) = \h_\Sigma(\mu)$.
\end{lemma}
\begin{proof}
    Recall that in general $b_{0,\Sigma}(\mu) \leq \h_\Sigma(\mu)$, so we only have to prove that having totally shattered microstate spaces implies the reverse inequality.

    Let $\delta>0$ be as in the definition of totally shattered microstate spaces. Given $\varepsilon< \delta/2$, there exists a neighborhood $U$ of $\mu$ such that for all large $n$, every $\mb{x},\mb{y} \in \Omega(\sigma_n, U)$ have $\dee^{(V_n)}(\mb{x},\mb{y}) \in [0,\varepsilon) \cup [\delta, \infty)$. In particular, for every $\mb{x} \in \Omega(\sigma_n, U)$
        \[ \cluster(\sigma_n, \mb{x}, U, \varepsilon) \subseteq \ball{\varepsilon}{\mb{x}}, \]
    where $\ball{\varepsilon}{\mb{x}}$ is the radius-$\varepsilon$ ball around $\mb{x}$. Thus
         \[ \abs{\cluster(\sigma_n, \mb{x}, U, \varepsilon)} \leq \abs{\ball{\varepsilon}{\mb{x}}} \le \exp\big( \abs{V_n} (\shent(\varepsilon) + \varepsilon \log \abs{\A} + o(1) )\big) . \]
     
    This leads to a lower bound on the number of clusters for all $\calO' \subset U$:
        \[ \abs*{\bigslant{\Omega(\sigma_n, \calO')}{\cluster(\Omega(\sigma_n, U), \, \varepsilon)}} \geq \abs{\Omega(\sigma_n, \calO')} \exp\big( -\abs{V_n} (\shent(\varepsilon) + \varepsilon \log \abs{\A} + o(1))\big) . \]
    Hence
        \[ \inf_{\mu \in \calO' \subset U} \limsup_{n \to \infty} \frac{1}{\abs{V_n}} \log \abs*{\bigslant{\Omega(\sigma_n, \calO')}{\cluster(\Omega(\sigma_n, U), \, \varepsilon)}} \geq \h_\Sigma(\mu) - \big( \shent(\varepsilon) + \varepsilon \log\abs{\A} \big) , \]
    and taking the supremum over $\varepsilon>0$ and $U \ni \mu$ completes the proof.
\end{proof}

\begin{proof}[Alternate proof of Cor.~\ref{cor:shatt-cons}(2)]
To obtain a contradiction suppose that $(\A^\Gamma,\mu,T)$ is measurably conjugate to the direct product of a nontrivial Bernoulli shift $(\tK^\Gamma,\k^\Gamma,T)$ with another shift system $(\B^\Gamma,\nu,S)$ where $\B$ is a compact metrizable space. We may assume $\B$ is totally disconnected without loss of generality because any dynamical system is measurably conjugate to a system of this form; this assumption is used in results of \cite{bowen2022} cited below.

Sofic entropy is additive under taking direct products with a Bernoulli shift \cite[Theorem 8.1]{bowen-jams-2010}, so
\begin{eqnarray}\label{E:shattered}
\h_\Sigma(\A^\Gamma,\mu,T) = \shent(\tK,\kappa) + \h_\Sigma(\B^\Gamma,\nu,S) > \h_\Sigma(\B^\Gamma,\nu,S).
\end{eqnarray}

Theorem 7.8 of \cite{bowen2022} implies that the $0$-dimensional sofic homology theories of $(\A^\Gamma,\mu,T)$ and $(\B^\Gamma,\nu,S)$ are equivalent. In particular, this implies that the exponential rate of growth of the number of clusters in the microstate spaces of the two actions are the same. In the notation of \cite{bowen2022}, this means
$$b_{0,\Sigma}(\A^\Gamma,\mu,T) = b_{0,\Sigma}(\B^\Gamma,\nu,S)\le \h_\Sigma(\B^\Gamma,\nu,S)$$
where the last inequality holds by \cite[Lemma 7.13]{bowen2022}.

Because $(\A^\Gamma,\mu,T)$ has totally shattered microstate spaces, by Lemma~\ref{lem:shatteredbetti} we have $b_{0,\Sigma}(\A^\Gamma,\mu,T) = \h_\Sigma(\A^\Gamma,\mu,T)$. Combined with the previous inequality this implies
$$\h_\Sigma(\A^\Gamma,\mu,T) = b_{0,\Sigma}(\A^\Gamma,\mu,T) \le \h_\Sigma(\B^\Gamma,\nu,S)$$
which contradicts (\ref{E:shattered}). This contradiction finishes the proof.
\end{proof}

\part{Parity check subshifts}\label{P2}

In this part, we fix natural numbers $d,k$ with $k > d \geq 3$ and let $\Gamma=\Gamma_{d,k}$ be the $d$-fold free product of order-$k$ cyclic groups:
\[\Gamma := \langle s_1,\dots,s_d:\ s_1^k = \dots = s_d^k = e\rangle = \underbrace{\ZZ_k\ast \cdots \ast \ZZ_k}_{d}.\]
Let $X \le \ZZ_2^\Gamma$ be the closed subgroup defined by
\[X = \left\{x \in \ZZ_2^\Gamma:\ \sum_{j=0}^{k-1}x_{gs_i^j} = 0 \ \forall g\in \Gamma,\ i = 1,\dots,d\right\},\]
and let $\mu = m_X$ be the Haar probability measure on $X$.

If $F$ is any subset of $\G$, let $X_F$ be the image of $X$ under the coordinate projection map $\ZZ_2^\Gamma \to \ZZ_2^F$, and let $m_F$ be the pushforward of $m_X$ under this projection.  We mostly use these notations when $F$ is $B_r$, the ball of radius $r$ centered at the identity in the Cayley graph of $\Gamma$, where we use $\{s_i^j:~ 1\le i \le d, 1\le j \le k-1\}$ for the generating set.

\section{LDPC codes and measures on microstate spaces}\label{sec:LDPC}

\subsection{The use of LDPC codes}\label{S:outline}

Our proofs of the main theorems are considerably simplified by using the special structure of the system as a subgroup of $\ZZ_2^\Gamma$.  We largely do so via the corresponding finitary codes constructed over the sofic approximations.  Given a $k$-uniform homomorphism $\sigma:\Gamma_{d,k}\to \Sym(V)$, let
\[X_\sigma := \left\{\mathbf{x} \in \ZZ_2^V:\ \sum_{j=0}^{k-1}x_{\sigma_i^j(v)} = 0\ \forall v\in V,\ i=1,\dots,d\right\}.\]
Let $\mu_\sigma$ be the uniform distribution on $X_\sigma$.

These are the obvious finitary analogs of $X$ and $m_X$ themselves.  It turns out that these finitary constructs can be used as better and better approximations to the infinitary system and measure, and their linear structure makes them easier to analyze than the `looser' sets $\Omega(\sigma,U)$ that are used to define sofic entropy.

In fact, sets such as $X_\sigma$ are classical objects in coding theory.  They are linear codes over the field $\ZZ_2$, each defined by a collection of parity-check constraints.  Since each of those parity-checks involves only $k$ vertices, and $k$ is fixed as $n$ grows, these are examples of low-density parity-check (`LDPC') codes.  Such codes were introduced in Gallager's PhD thesis~\cite{gallager1962,gallager1963}.  After many years of relative neglect, they were re-discovered independently by MacKay in the late 1990s~\cite{mackay1999}, and they are now a textbook family of codes with desirable properties.  A good basic reference for their theory is~\cite[Chapter 47]{mackay2003}, and a more dedicated treatment is~\cite{richardson2008}.  In fact, the essence of the parity-check subshift $X$ itself already appears in those sources too, playing the role of an `idealized limit code' on which to investigate the performance of local decoding algorithms: see, for instance,~\cite[Figure 47.11]{mackay2003} and the discussion around it.

\subsection{Outline of the rest of the paper}

Our use of the measures $\mu_\sigma$ to analyze the system $(X,m_X,T)$ rests on the following main results.  
As before, we confine the index $n$ to multiples of $k$. 
Abbreviate $X_{\sigma_n}$ to $X_n$ and $\mu_{\sigma_n}$ to $\mu_n$. 






Recall that $\PP_n$ is the uniform probability measure on $\Hom_{\unif}(\Gamma, \Sym(n))$ which is the set of $k$-uniform homomorphisms $\s:\Gamma \to \Sym(n)$. 

\begin{theorem}\label{thm:umm}
    There are subsets
    \[\Omega'_n \subseteq \Hom_{\mathrm{unif}}(\Gamma,\Sym(V_n))\]
    such that $\PP_n(\Omega'_n) \to 1$ and the following holds.  If $\sigma_n\in \Omega_n'$ for each $n$, then 
    \begin{enumerate}
        \item $(\mu_n)_n$ has property M;
        \item $(\mu_n)_n$ converges locally and empirically to $m_X$;
        \item $(X,m_X,T)$ has totally shattered microstate spaces along $\Sigma=(\sigma_n)_n$.
    \end{enumerate}
\end{theorem}

Our proof of Theorem~\ref{thm:umm} relies on the linear structure of the codes $X_\sigma$. 

In the next few sections we introduce notation to formulate Proposition \ref{prop:few-extra-cancellations} which, roughly speaking, rules out near-cancellations among parity-check words of a typical $\s_n \sim \PP_n$. Section \S \ref{subs:cancel-to-umm} proves Theorem~\ref{thm:umm}(1) from Proposition \ref{prop:few-extra-cancellations}. The rest of \S \ref{sec:propM} proves Proposition \ref{prop:few-extra-cancellations}. 

Item (2) of Theorem~\ref{thm:umm} is proven in \S \ref{S:lde}. Its proof relies on item (1). Item (3) is proven in \S \ref{S:mainthm2pf}. Its proof does not refer to items (1) or (2). Theorem~\ref{mainthm2} is an immediate consequence of item (3) and Corollary~\ref{cor:shatt-cons}(2).

Theorem \ref{mainthm3} is proven in \S \ref{sec:mainthm3pf}. Part (b) of that theorem follows readily from items (1) and (2) of Theorem~\ref{thm:umm} together with Theorem \ref{T:wmmcpe}. Part (a) (which computes the sofic entropy value) uses item (2) and the Bethe-Kikuchi entropy theory of \S \ref{sec:B-K}.

\subsection{Random factor graphs}\label{subs:factor-graphs}

In this subsection, fix a size $n$ that is divisible by $k$ and suppress it from the notation: thus, for instance, $\PP$ stands for $\PP_n$.  Most of our work towards Theorem~\ref{thm:umm} consists of estimates of various probabilities under $\PP$.

Several of these estimates involve a sum or union bound over possible subfamilies of the set of all hyper-edges of the form~\eqref{eq:hyper-edge}.  We need to be able to move the sum outside an expectation, and for this purpose the sum must be over a range which is fixed, not random.  For this reason, it is convenient to augment the information in $\sigma_n$ with a labeling of the family of hyper-edges~\eqref{eq:hyper-edge} by a fixed index set.

Let $E_1$, \dots, $E_d$ be disjoint sets, each of size $n/k$, and let $E := E_1\cup\cdots\cup E_d$.  Taking some terminology from coding theory, we refer to the elements of $E$ as \textbf{check nodes}: this is explained further in Subsection~\ref{subs:matrices} below.  Let $\sigma \in \Hom_{\mathrm{unif}}(\Gamma,\Sym(V))$ with $|V|=n$.  Fix $i \in [d]$, and consider that $V$ is partitioned into the orbits of the generator $\sigma(s_i)$. Each orbit corresponds to a hyper-edge as in \eqref{eq:hyper-edge}. Since each hyper-edge has size $k$, there are $n/k$ of them, and so there exists a bijection between $E_i$ and this family of hyper-edges.  Let us choose such a bijection uniformly and independently at randomly for each $i$, and record the result as a subset $H\subseteq E\times V$: a pair $(e,v) \in E_i\times V$ lies in $H$ if $e$ is attached by the $i^{\mathrm{th}}$ bijection to the hyper-edge that contains $v$.

We regard $H$ as a bipartite graph on the disjoint union of $E$ and $V$.  As such, each check node in $E$ has exactly $k$ neighbours in $V$, and each vertex in $V$ has exactly one neighbour in each of the subsets $E_i$ (and thus $d$ neighbours in total).  It follows that each intersection $H\cap (E_i\times V)$ is equivalent to a partition of $V$ into parts labelled by $E_i$.  In the sequel, we borrow some more terminology from coding theory and refer to any such bipartite graph $H$ on $E$ and $V$ as a \textbf{factor graph}: see, for instance,~\cite[Sections 26.1 and 47.2]{mackay2003}.  Beware that this is actually a slight deviation from standard usage, which would not insist that $H$ be a union of the partitions $H\cap (E_i\times V)$, but here we do take this as part of the definition of a `factor graph'.

More generally, if $F \subseteq E$, then a \textbf{partial factor graph} on $F$ and $V$ is a bipartite graph $M \subseteq F\times V$ such that every check node in $F$ has precisely $k$ neighbours in $V$ and every vertex in $V$ is joined to at most one check node in each intersection $F\cap E_i$. Equivalently, there exists a factor graph $H$ such that $M = H \cap (F\times V)$. In particular, if $F = E_i$, then a partial factor graph is simply a partition of $V$ into $k$-sets that are labelled by $E_i$.

If $H$ is a factor graph on $E$ and $V$ and $F\subseteq E$, then the \textbf{vertex neighbourhood} of $F$ is the set
\[\Ver(H;F) := \{v \in V:\ (e,v) \in H\ \hbox{for some}\ e \in F\}.\]
We also use the notation $\Ver(M;F) $ for a partial factor graph $M$ in the same way.

Similarly, the \textbf{check-node neighbourhood} of $U\subseteq V$ is the set
\[\Chec(H;U) := \{e \in E:\ (e,v) \in H\ \hbox{for some}\ v \in V\}.\]

We may iterate these definitions to define neighbourhoods with larger radii.  In general, for $F \subseteq E$ or $U\subseteq V$ we define
\begin{eqnarray*}
\Ver_1(H;F) &:=& \Ver(H;F),\\
\Chec_1(H;U) &:=& \Chec(H;U),\\
\Chec_1(H;F) &:=& \Chec_1\big(\Ver_1(H;F)\big) \\ 
\hbox{and} \quad \Ver_1(H;U) &:=& \Ver_1\big(\Chec_1(H;U)\big).
\end{eqnarray*}
Then for integers $r > 1$, we make the recursive definitions
\[\Ver_r(H;F) := \Ver_1(H;\big( \Ver_{r-1}(H;F)\big),\]
and the same with $U$ in place of $F$ or $\Chec$ in place of $\Ver$.  In graph theoretic terms, if $F \subseteq E$, then $\Ver_r(H;F)$ is the set of all vertices whose graph distance in $H$ is at most $2r-1$ from $F$, $\Chec_r(H;F)$ is the set of all check nodes whose graph distance in $H$ is at most $2r$ from $F$, and similarly for subsets of $V$.

These neighbourhoods are compatible with the Cayley graph neighbourhoods $B_r = \ball{e}{r}$ induced by the word metric on $\Gamma$ (with respect to the generating set $\{s_i^j:~ 1 \le i \le d, 1\le j \le k-1\})$. Specifically, if $H$ arises from $\sigma$ through the construction above, and $U\subseteq V$, then
\[\sigma^{B_r}(U) := \{\sigma^g(u):\ g\in B_r,\ u \in U\} =  \Ver_r(H;U).\]

Now consider $\sigma \sim \PP$ and generate $H$ from it as described above.  Then $H$ results from two random steps: the choice of $\sigma$, and then the choice of a bijection for each $i$.  Let $\tilde{\PP}$ be the joint distribution of $(\sigma,H)$ after this construction.  This is a probability measure on
\[\tilde{\Omega} := \Hom_{\mathrm{unif}}(\Gamma,\Sym(V)) \times \{0,1\}^{E\times V}.\]
It is a coupling of $\PP$ to the law of $H$ described above, and it has the following simple properties:
\begin{itemize}
    \item Given $\sigma$, the conditional distribution of $H$ is uniform over $((n/k)!)^d$ choices of bijections.
    \item Given $H$, the conditional distribution of $\sigma$ is uniform over all choices of cyclic orderings for each hyper-edge of the multi-hyper-graph: there are ${((k-1)!)^{dn/k}}$ such choices for any $H$.
    \item Under the marginal distribution of $H$, the intersections $H\cap (E_i\times V)$ are independent as $i$ varies.
\end{itemize}

For $(\sigma,H)$ drawn from $\tilde{\PP}$, the underlying multi-hyper-graph may be read off from either coordinate: it is the multi-set of all $\sigma(s_i)$ orbits for $i\in [d]$, and it is also the collection of all vertex neighbourhoods of the check nodes according to $H$.  The labeling of these hyper-edges by the fixed set $E$ that is given by $H$ is convenient for union bounds and other forms of counting.  Most of our probabilistic estimates concerning these hyper-graphs later refer to $\tilde{\PP}$ rather than $\PP.$

\begin{remark}
The random hyper-graphs on $V$ that arise from $\sigma \sim \PP$ are $k$-uniform and $d$-regular.

Models of such random hyper-graphs already have an established place in the literature on probabilistic combinatorics: see, for instance,~\cite[Subsection 3.5]{wormald1999} and the references given there.  However, most of that literature is dedicated to a \emph{uniform} random choice of such a hyper-graph, and this is not the same as the distribution that results from a random choice of $\sigma$.  The point is that if the hyper-graph is generated by a homomorphism $\sigma$, then its hyper-edges can be classified according to which generator $s_i$ gave rise to them.  For each single $i$, the corresponding hyper-edges form a partition of $V$ into $k$ sets.  In general, a $k$-uniform $d$-regular hyper-graph need not be a union of partitions.

In case $k=2$, the difference here is between a uniformly random $d$-regular graph and the sum of $d$ independent random matchings.  In this case the difference has been studied in some depth, with the outcome that these models are `contiguous': see, for instance,~\cite{greenhilljansonkimwormald2002} or\cite[Subsection 4.3]{wormald1999}.   The relation of contiguity is strong enough to allow us to transfer most phenomena of interest from one model to the other.

However, for $k > d \ge 3$, it turns out that contiguity fails.  While we have not found a reference for this fact in the literature, it follows fairly easily from some other standard results, so we explain this in Appendix~\ref{sec:non-contig}.

For uniformly random $k$-uniform $d$-regular hyper-graphs, estimates on the typical behaviour of the resulting LDPC codes are widely available in the coding theory literature.  For example, the typical rate of the resulting LDPC code is known very exactly from~\cite{millercohen2003} (we cite this in our proof of non-contiguity in Appendix~\ref{sec:non-contig}).  These known results would include most of the facts we need here, were it not for the difference in the underlying random hyper-graph model.  However, we have not found the analogous estimates for our distributions $\PP$, so we must develop them here from scratch as necessary.  Nevertheless, the conclusions generally look the same as in those previous works, and we have been guided by those throughout.
\end{remark}

\subsection{Parity-check matrices}\label{subs:matrices}

To prove the desired properties of the random measures $\mu_{\sigma_n}$, we make careful use of the linear structure of the codes $X_n$.  As is standard in coding theory, this structure is conveniently summarized by a parity-check matrix.

To introduce this point of view, we start with some more notation.  If a ground set $A$ is understood and $B \subseteq A$, then $\mathbf{e}_B$ denotes the mod-$2$ indicator function of $B$: this is the element of $\ZZ_2^A$ with entries equal to $\mathtt{1}$ precisely at the indices in $B$.

Now consider a vertex set $V$ of size $n$ which is a multiple of $k$, and let $E = E_1\cup \cdots \cup E_d$ as in Subsection~\ref{subs:rndm-sofic-approx}.  Let $H$ be a factor graph on $E$ and $V$, and turn it into the $(E\times V)$-matrix $\mathbf{H} = \mathbf{e}_H$, which also takes values in $\ZZ_2$.  If this generates the same hyper-graph as $\sigma \in \Hom_{\unif}(\Gamma, \Sym(V))$, then our code is given by
\[X_\sigma = \{\mathbf{x} \in \ZZ_2^V:\ \mathbf{H}\mathbf{x} = \boldsymbol{0}\} = \ker \mathbf{H}.\]
In this role $\mathbf{H}$ is called the \textbf{parity-check matrix} (for an introduction, see, for instance,~\cite[Chapter 9]{cameronvanlint1991},~\cite[Section 7.11]{coverthomas2006}, or~\cite[Chapter 1]{mackay2003}, especially the discussion of Exercise 1.9).  This is why we refer to the elements of $E$ as `check nodes': each corresponds to a row of $\mathbf{H}$, which every codeword $\mb{x}$ must be orthogonal to.  The representation of a code using a set of connections between vertices and check nodes is a common example of a `factor graph representation'~\cite[Section 47.2]{mackay2003}, hence our use of the term `factor graph' for $H$.

Because $X_\sigma$ is a linear subspace of $\ZZ_2^V$, it may equivalently be specified via its dual code
$$X_\sigma^\perp := \{\mathbf{y} \in \ZZ_2^V:\ \langle\mathbf{y}, \mathbf{x}\rangle = 0 \pmod{2} \ \forall \mathbf{x} \in X_\sigma \}.$$
Below, we refer to elements of $X_\sigma^\perp$ as \textbf{parity checks}.
By the construction of $X_\sigma$, this $X_\sigma^\perp$ is precisely the linear subspace of $\ZZ_2^V$ spanned by the rows of $\mathbf{H}$: that is, by the vectors $\mathbf{e}_{\Ver(H;\{e\})}$ for $e \in E$. More succinctly,
 \[X_\sigma^\perp = \mathrm{img}\,(\mathbf{H}^\mathsf{T}).\]

Theorem~\ref{thm:umm} (1,2) are proved by counting relations or `near relations' among the rows of $\mathbf{H}$.  The second of these theorems requires the more complicated calculation.  It is based on Proposition~\ref{prop:few-extra-cancellations} below.  With that proposition in hand, we can also prove Theorem~\ref{thm:umm}(2); see Section~\ref{S:lde}.

\section{Property M} 
\label{sec:propM}

In this section we again suppress the subscript $n$ from our notation, as in Subsection~\ref{subs:factor-graphs}.  As before this index will tend to $\infty$ along multiples of $k$.

The main part of the proof of Theorem~\ref{thm:umm} parts (1) and (2) is another, more technical proposition.  Put roughly, it rules out most `near cancellations' among the parity-check words of a typical $(\sigma,H)$ drawn at random.  However, for the application to Theorem~\ref{thm:umm}, we need such a result not only when $(\sigma,H) \sim \tilde{\PP}$, but also after conditioning $(\sigma,H)$ on a small fraction of the vertices and check nodes.

We formulate this technical proposition next. For each $i \in [d]$ let $F_i \subseteq E_i$ and let $W_i := \Ver(H;F_i)$.  Let $w_i:= |W_i| = k|F_i|$.  Let $F:= F_1\cup\cdots\cup F_d$, $W := W_1 \cup \cdots \cup W_d$, and $w:= |W|$; see Figure~\ref{fig:factorgraphs}. Then
\[\max\{w_1,\dots,w_d\} \le w \le w_1 + \cdots + w_d.\]

\begin{figure}
\begin{center}
	\includegraphics{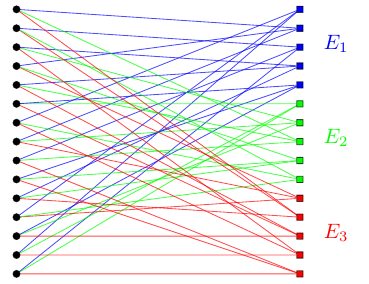}
    \ 
	\includegraphics{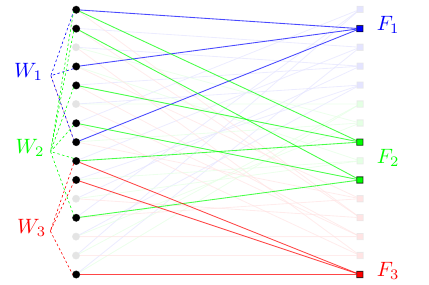}
	\caption{Diagrams showing full factor graph $H$ (left) and partial factor graph $M$ (right). The square vertices on the right of each graph are the check nodes, colored according to their membership in the sets $E_1, E_2, E_3$. The distribution $\tilde{\PP}^M$ draws a new pair $(\sigma, H)$ conditioned on the edges on the right being present.}
 \label{fig:factorgraphs}
 \end{center}
\end{figure}

Let $M$ be a partial factor graph on $F$ and $V$, and let $\tilde{\PP}^M$ be the distribution obtained by conditioning $\tilde{\PP}$ on the event that $H\cap (F\times V) = M$.  The key to Theorem~\ref{thm:umm} is that, if the sets $F_i$ are small enough, then after this conditioning, the rest of $H$ is very unlikely to create many new parity checks that involve only vertices inside $W$. Roughly speaking, this means that if ${\bf x} \in \ZZ^W_2$ satisfies all of the parity checks arising from $F$ then, with high probability, it admits an extension satisfying all of the parity checks.

Here and in the rest of the paper, if $F\subseteq E$ and $U\subseteq V$, then $\mathbf{H}_{F\times U}$ denotes the submatrix of $\mathbf{H}$ indexed by these sets, and we use analogous notation for the transpose $\mathbf{H}^{\mathsf{T}}$ and for vectors indexed by either $E$ or $V$.  We identify $\ZZ_2^{E\setminus F}$ as a subspace of $\ZZ_2^E$ in the obvious way, and so any $\mathbf{y} \in \ZZ_2^{E\setminus F}$ may be written as a tuple $(\mathbf{y}_1,\dots,\mathbf{y}_d)^{\mathsf{T}}$ with $\mathbf{y}_i \in \ZZ_2^{E_i\setminus F_i}$.

\begin{prop}[`Few additional checks inside $W$']\label{prop:few-extra-cancellations}
For every $K > 0$ and $\varepsilon > 0$ the following holds for all sufficiently small $\delta$ (depending on $d$, $k$, $\varepsilon$ and $K$).  Let $F_i$, $W_i$, $w_i$ and $M$ be as above.  If
\begin{equation}\label{eq:small-enough}
\delta n \le w_1 + \dots + w_d \le K \delta n,
\end{equation}
then 
\begin{multline*}
\tilde{\PP}^M\Big(\exists \mathbf{y} \in \ZZ_2^{E\setminus F}\ \hbox{such that}\ \min\{|\mathbf{y}_i|,|\mathbf{e}_{E_i\setminus F_i} - \mathbf{y}_i|\} \ge \varepsilon\delta\frac{n}{k}\ \hbox{for some $i$}\\ \hbox{ and}\ (\mathbf{H}^\mathsf{T}\mathbf{y})_{V\setminus W} = \boldsymbol{0}\Big) \to 0
\end{multline*}
as $n\to\infty$. The rate of convergence depends only on $d$, $k$ and $\delta$, and is independent of $\delta$ provided $\delta$ is small enough and also bounded away from zero.
\end{prop}

We show how Proposition~\ref{prop:few-extra-cancellations} implies Theorem~\ref{thm:umm} in the next subsection, and then we prove Proposition~\ref{prop:few-extra-cancellations} itself in Subsection~\ref{subs:few-extra-cancellations}.

\begin{remark}
    Proposition~\ref{prop:few-extra-cancellations} has a similar flavor to the analysis of the satisfiability threshold for the random combinatorial problem known as random XORSAT.  But our setting has the additional complication that we must analyse conditional probabilities given the behaviour of the random factor graph $H$ on the small sets $F_i$, which seems to make our situation less `homogeneous'.  The random XORSAT model is treated in~\cite[Chapter 18]{mezard2009} using the paradigm of `belief propagation'.  It is possible that that approach could also be brought to bear in the situation in Proposition~\ref{prop:few-extra-cancellations}, possibly leading to an alternative proof, but we have not pursued this idea further.
\end{remark}

\subsection{Property M from few additional checks}\label{subs:cancel-to-umm}

To deduce Theorem~\ref{thm:umm} from Proposition~\ref{prop:few-extra-cancellations}, we also need two simpler calculations that we present as separate lemmas.

The first concerns the following construction.  Fix a subset $R$ of $V$ and let $r > 0$.  For any factor graph $H\subset E\times V$, we can form the neighbourhood $F := \Chec_r(H;R)$, and then the intersection $M := H \cap (F\times V)$ is a partial factor graph on $F$ and $V$.  Let us call the pair $(F,M)$ \textbf{possible with respect to $(R,r)$} if they can arise from a factor graph in this way.

\begin{lemma}\label{lem:cond-dist}
Let $R$ and $r$ be as above, let $(F,M)$ be possible with respect to $(R,r)$, and let $(\sigma,H) \sim \tilde{\PP}$.  If the event
\[\{H:\ H\cap (F\times V) = M\}\]
occurs, then so does the event
\[\{H:\ \Chec_r(H;R) = F\}.\]
If we condition on the former event then the subsets $H\cap ((E_i\setminus F)\times V)$ are independent for different $i$, and each is a uniform random labelled partition of $V\setminus \Ver(M; E_i \cap F)$ into $k$-sets.
\end{lemma}

\begin{proof}
First, the definition of the neighbourhood $\Chec_r(H;R)$ depends only on those edges of the bipartite graph $H$ that connect this neighbourhood to $V$.  Since we are told that $(F,M)$ is possible with respect to $(R,r)$, all those edges must already be visible in $M$, and so knowing that $H\cap (F\times V) = M$ is enough to tell us that $\Chec_r(H;R) = F$.

Now the conditional probability in question is $\tilde{\PP}^M$, as introduced previously.  Since $\tilde{\PP}$ is a uniform distribution, $\tilde{\PP}^M$ is the uniform distribution over all factor graphs for which the event holds.  However, if the event holds then (i) $H \cap (F\times V)$ is uniquely determined, (ii) each $H\cap ((E_i\setminus F)\times V)$ must consist of a labelled partition of $V\setminus \Ver(M; E_i \cap F)$ into $k$-sets, and (iii) any tuple of such labelled partitions is still possible.  So this conditional distribution is simply the uniform distribution over the Cartesian product set of tuples of such labelled partitions. This implies the desired joint distribution for these sets.
\end{proof}

The next lemma may be well-known in coding theory, but we have not found a convenient reference.

\begin{lemma}\label{lem:weight-and-dim}
    Let $A$ be any finite non-empty index set, $\delta < 1/3$, and let $Y$ be a linear subspace of $\ZZ_2^A$ such that
    \begin{equation}\label{eq:big-or-small}
    \hbox{either}\quad |\mathbf{y}| \le \delta |A| \quad \hbox{or} \quad |\mathbf{y}| \ge (1-\delta)|A| \quad \hbox{for every} \quad \mathbf{y} \in Y.
    \end{equation}
    Then $\dim Y \le 2\delta |A| + 1$.
\end{lemma}

\begin{proof}
    Let $Z$ be the subset of all $\mathbf{y} \in Y$ for which $|\mathbf{y}| \le \delta |A|$.  We prove that $Z$ is a linear subspace, $\dim Y/Z \le 1$, and $\dim Z \le 2\delta |A|$.

    First, $Z$ clearly contains $\boldsymbol{0}$, and if $\mathbf{y},\mathbf{y}' \in Z$ then $\mathbf{y} + \mathbf{y}' \in Y$ and 
    \[|\mathbf{y} + \mathbf{y}'| \le 2\delta |A|.\]
    Since $2\delta < 1 - \delta$, by~\eqref{eq:big-or-small} this forces $\mathbf{y} + \mathbf{y}' \in Z$, so $Z$ is a subspace.
    
    Next, if $\mathbf{y},\mathbf{y}' \in Y\setminus Z$, then
    \[|\mathbf{y} - \mathbf{y}'| \le |(\mathtt{1},\dots,\mathtt{1}) - \mathbf{y}| + |(\mathtt{1},\dots,\mathtt{1}) - \mathbf{y}'| \le 2\delta |A|,\]
    so again we must in fact have $\mathbf{y} - \mathbf{y}' \in Z$.  Therefore $\dim Y/Z \le 1$.

    Finally, call an index $i \in A$ \textbf{proper} if some member of $Z$ is non-zero in this coordinate.  Now choose $\mathbf{z}$ from $Z$ uniformly at random.  If $i$ is proper, then the coordinate $z_i$ is equally likely to be $0$ or $1$, so the expectation of $|\mathbf{z}|$ is half the number of proper coordinates.  Therefore the number of proper coordinates is at most $2\delta |A|$, and this number is an upper bound on $\dim Z$.
\end{proof}

Property M (Theorem~\ref{thm:umm}) follows from the following result, which is also used to prove local and empirical convergence:

\begin{theorem}
\label{thm:propertyM}
    Given $r \in \NN$ and $\eta > 0$, for all small enough $\delta>0$ the following holds: if for each $n$, $R_n \subset V_n$ is a subset of size $\lceil \delta n \rceil$, then with high probability as $n \to \infty$
        \[ \shent\! \big( (\mu_n)_{\sigma_n^{B_r}(R_n)} \big) \geq (1 - \eta) \, \shent \!\big( m_{B_r} \big)|R_n|. \]
\end{theorem}


\begin{proof}[Proof of Theorem \ref{thm:propertyM} assuming Proposition \ref{prop:few-extra-cancellations}]
Beware that we continue to suppress $n$ from subscripts where possible.  It should be understood that the following argument and construction are carried out for each $n$ that is divisible by $k$.

Fix $\varepsilon > 0$. For a small positive $\delta$ to be specified shortly, let $R$ be any fixed choice of a subset of $V$ of size $\lceil\delta n\rceil$, and now consider the following three subsets of $(\sigma,H)$ in $\tilde{\Omega}$:
\begin{itemize}
    \item[i.] (Few additional checks) $(\sigma,H)$ is in $\Omega^1$ if for
    \begin{quote}
    any $\mathbf{y} \in \ZZ_2^{E\setminus \Chec_r(H;R)}$  such that $(\mathbf{H}^\mathsf{T}\mathbf{y})_{V\setminus \Ver_r(H;R)} = \boldsymbol{0}$ also has
    \[\min\{|\mathbf{y}_i|,|\mathbf{e}_{E_i\setminus F_i} - \mathbf{y}_i|\} < \varepsilon\delta\frac{n}{k} \qquad \hbox{for every}\ i \in [d].\]
    \end{quote}
    where $F_i = E_i \cap \Chec_r(H;R)$.
    \item[ii.] (Most vertices in $R$ well-separated) $(\sigma,H)$ is in $\Omega^2$ if the set
    \[S_1 := \big\{v \in R:\ \Ver_r(H;v)\cap \Ver_r(H;R \setminus v) = \emptyset\big\}.\]
    has $|S_1| > (1-\varepsilon)|R|$.
    \item[iii.] (Most vertices in $R$ not close to any short loops) $(\sigma,H)$ is in $\Omega^3$ if the set
    \[S_2 := \big\{v \in R:\ \hbox{the orbit map}\ B_r\to \Ver_r(H;v)\ \hbox{is injective}\big\}.\]
    has $|S_2| > (1-\varepsilon)|R|$. (Here, the orbit map sends $\gamma \in B_r \subset \Gamma$ to $\sigma(\gamma)v$.)
\end{itemize}
Let $S := S_1\cap S_2$ and $\Omega' := \Omega^1 \cap \Omega^2\cap \Omega^3$.  Think of $S$ as the result of expurgating the `bad' vertices from $R$, and $\Omega'$ as the event that there are only a few of these.

In the remainder of the proof we show that, provided $\delta$ was chosen small enough, we have $\tilde{\PP}(\Omega')\to 1$ and the set $R$ satisfies the desired entropy bound.

\vspace{7pt}

\emph{Step 1: $\tilde{\PP}(\Omega')\to 1$.}\quad First, we have $\tilde{\PP}(\Omega^2) \to 1$ as $n\to\infty$ for any sufficiently small $\delta$ in terms of $d$, $k$, $r$ and $\varepsilon$. To see this, observe that although we are fixing $R$ and choosing $(\sigma,H)$ at random, we obtain the same distribution on $\abs{S_1}$ if we choose $R$ uniformly at random among all subsets of $V$ of cardinality $\lceil\delta n\rceil$, independently of $(\sigma,H)$. Thus it suffices to show that on any $d$-regular $k$-uniform hyper-graph on $n$ vertices, if a subset $R \subset V$ of $\lceil\delta n\rceil$ vertices is chosen uniformly at random then the probability that there are more than $\varepsilon \delta n$ vertices of $R$ which are $\le 2r$ distance from another vertex in $R$ tends to zero as $n\to\infty$. 

If some number, say $x$, of vertices have already been chosen, then the probability that the next vertex is not within distance $2r$ of the previously selected vertices is at least 
$$1- \frac{(kd)^{2r}x}{n-x} \ge 1- \frac{(kd)^{2r}\lceil\delta n\rceil}{n - \lceil\delta n\rceil}\ge 1- (kd)^{2r} \frac{\delta}{1-\delta}.$$ 
This is because the number of vertices in the $(2r)$-neighbourhood of a given vertex is at most $(kd)^{2r}$. Thus $\tilde{\PP}(\Omega^2)$ is at least the probability that in $\lceil\delta n\rceil$ Bernoulli trials with success probability 
$1- (kd)^{2r} \, \delta/(1-\delta) $, there are at least $(1-\varepsilon)\lceil \delta n\rceil$ successes. This occurs with overwhelming probability as $n\to\infty$ as long as $\varepsilon> (kd)^{2r} \, \delta/(1-\delta)$, which we may assume by choosing $\delta$ sufficiently small. 

Second, for any $\delta > 0$ we have $\tilde{\PP}(\Omega^3)\to 1$ as $n\to\infty$ as an immediate corollary of Proposition~\ref{prop:probably-sofic}.

So now let us show that $\tilde{\PP}(\Omega^1) \to 1$ as $n\to\infty$.  This is our application of Proposition~\ref{prop:few-extra-cancellations}. We make contact with that proposition by conditioning on the partial factor graph $H\cap (\Chec_r(H;R)\times V)$ and using the law of total probability.

Let $(F,M)$ be a possible pair with respect to $(R,r)$ as in Lemma~\ref{lem:cond-dist}, and let $\tilde{\PP}^M$ be the result of conditioning $\tilde{\PP}$ on the event $H\cap (F\times V) = M$, as previously.  Let $W_i := \Ver(M;F_i)$ and let $W := W_1\cup \cdots \cup W_d$.  By the law of total probability, $\tilde{\PP}(\tilde{\Omega}\setminus \Omega^1)$ is equal to the sum
\begin{multline*}
\sum_{\mathrm{possible}\,F,M}\tilde{\PP}\big(H\cap (F\times V) = M\big)\cdot \tilde{\PP}^M\Big(\exists \mathbf{y} \in \ZZ_2^{E\setminus F}\ \hbox{such that}\\
\min\{|\mathbf{y}_i|,|\mathbf{e}_{E_i\setminus F_i} - \mathbf{y}_i|\} \ge \varepsilon\delta\frac{n}{k}\ \ \hbox{for some $i$}\ \hbox{and}\ (\mathbf{H}^\mathsf{T}\mathbf{y})_{V\setminus W} = \boldsymbol{0}\Big).
\end{multline*}
Let $K := (dk)^{r+1}+1$. Our construction of the possible pair $(F,M)$ gives that
\[ \delta n\le |R| \le |W| \le \sum_{i=1}^d |W_i| \le (dk)^{r+1}|R| \le K\delta n,\]
when $\delta n >1$ and so the condition~\eqref{eq:small-enough} is satisfied for this value of $K$ by the second factor in every term of the sum above.  Therefore, by Proposition~\ref{prop:few-extra-cancellations}, if $\delta$ is small enough in terms of $d$, $k$, $\varepsilon$ and $r$, then this sum is a convex combination of quantities that converge to $0$ at a rate depending only on $d$, $k$ and $\delta$, and hence so does the whole expression.

This proves that $\tilde{\PP}(\Omega')\to 1$ for any sufficiently small $\delta$.  Fix such a $\delta$ for the rest of the proof.

\vspace{7pt}

\emph{Step 2.}\quad To finish the proof, we show that if $\sigma \in \Omega'$ then (i), (ii) and (iii) imply that $S$ has the properties required to witness property M as in Definition~\ref{dfn:M}.

First, $S_1$ is $(2r)$-separated by property (ii), and hence so is $S$.

Second, properties (ii) and (iii) together give
\[\frac{|S|}{|V|} \ge (1 - 2\varepsilon)\frac{|R|}{|V|} \ge (1 - 2\varepsilon)\delta,\]
which is uniformly positive in $n$ provided $\varepsilon < 1/2$.

Finally, we must prove~\eqref{eq:high-ent}.  Because $\mu_\sigma$ is the uniform distribution on the linear subspace $X_\sigma$ of $\ZZ_2^V$, the projection of $\mu_\sigma$ over the subset of vertices $\sigma^{B_r}(S) = \Ver_r(H;S)$ is the uniform distribution on the linear space
\[Z:= \{\mathbf{x}_{\Ver_r(H;S)}:\ \mathbf{x} \in X_\sigma\}.\]
Therefore, after ignoring a factor of $\log 2$, we need a lower bound on $\dim Z$.

By property (iii), each of the neighbourhoods $\Ver_r(H;v)$ for $v \in S$ is a bijective copy of $B_r \subseteq \Gamma$, and by property (ii) these neighbourhoods are disjoint.  Therefore $Z$ is naturally identified with a linear subspace of
$$\prod_{v \in S} \ZZ_2^{\Ver_r(H;v)}\cong (\ZZ_2^{B_r})^S.$$
As a linear subspace of $\ZZ_2^{\Ver_r(H;S)}$, $Z$ is determined by its dual code $Z^\perp$: that is, its own collection of parity-check words in $\ZZ_2^{\Ver_r(H;S)}$~\cite[Section 13.10]{mackay2003}.  Since $Z$ is the projection of $X_\sigma$, each word in $Z^\perp$ becomes a word in $X^\perp_\sigma$ when it is extended by $0$ to $V\setminus \Ver_r(H;S)$.  Therefore $Z^\perp$ is identified with the set of all mod-$2$ sums of the rows of $\mathbf{H}$ that vanish on $V\setminus \Ver_r(H;S)$, and so the rank-nullity formula gives
\begin{equation}\label{eq:rank-nullity}
\dim Z = \dim (\ZZ_2^{B_r})^S - \dim \ker \mathbf{H}^{\mathsf{T}}_{(V\setminus \Ver_r(H;S))\times E}.
\end{equation}

So let us consider the ways in which a mod-$2$ sum of the rows of $\mathbf{H}$ can vanish on $V\setminus \Ver_r(H;S)$.  First, each individual row of $\mathbf{H}$ that is indexed by an element of $\Chec_r(H;S)$ vanishes outside $\Ver_r(H;v)$ for some single element $v \in S$.  Since each of these sets is a bijective copy of $B_r$, these checks by themselves show that $Z$ is actually a linear subspace of $X_{B_r}^S \subset (\ZZ_2^{B_r})^S$.

On the other hand, given any other mod-$2$ sum of rows of $\mathbf{H}$ which vanishes on $V\setminus \Ver_r(H;S)$, we may remove any summands indexed by $\Chec_r(H;S)$ without losing that property.  So now consider a mod-$2$ sum of rows of $\mathbf{H}$ that is supported in $\Ver_r(H;S)$ and uses no rows indexed by $\Chec_r(H;S)$.  It might use only rows indexed by $\Chec_r(H;R)\setminus \Chec_r(H;S)$.  Or it might include at least one summand indexed by an element of $E\setminus \Chec_r(H;R)$, in which case those summands by themselves define a non-zero vector in
$$D := \ker \mathbf{H}^{\mathsf{T}}_{(V\setminus \Ver_r(H;R)) \times (E\setminus \Chec_r(H;R))},$$
because the rows of $\mathbf{H}$ indexed by $\Chec_r(H;R)$ are supported in $\Ver_r(H;R)$ and so ignoring them cannot change the support in $V\setminus \Ver_r(H;R)$.

Combining these possibilities,~\eqref{eq:rank-nullity} implies the lower bound
\begin{equation}\label{eq:rank-nullity-2}
\dim Z \ge |S|\cdot \dim X_{B_r} - |\Chec_r(H;R)\setminus \Chec_r(H;S)| - \dim D.
\end{equation}
By properties (ii) and (iii), $|S| \ge (1-2\eps)|R|$. Since ${|\Chec_r(H;R \setminus S)| \le (dk)^{r-1}k |R\setminus S|}$, this implies
\[|\Chec_r(H;R)\setminus \Chec_r(H;S)| \le |\Chec_r(H;R\setminus S)| \le Kk\varepsilon|S|.\]
On the other hand, by property (i), if $\mathbf{y} \in D$ and we write $\mathbf{y} = (\mathbf{y}_1,\dots,\mathbf{y}_d)^{\mathsf{T}}$, then for each $i$ we have $\mathbf{y}_i \in \ZZ_2^{E_i\setminus \Chec_r(H;R)}$ and
\[\hbox{either}\ |\mathbf{y}_i| < \varepsilon\delta\frac{n}{k} \quad \hbox{or} \quad |\mathbf{e}_{E_i\setminus \Chec_r(H;R)} - \mathbf{y}_i| < \varepsilon\delta\frac{n}{k}.\]
Therefore, provided $\varepsilon\delta < 1/3$, Lemma~\ref{lem:weight-and-dim} gives
\[\dim\{\mathbf{y}_i:\ \mathbf{y} \in D\} \le 2\varepsilon\delta\frac{n}{k} + 1\quad \hbox{for each}\ i,\]
and hence $\dim D \le 2\varepsilon\delta nd/k+d$.

Inserting these bounds into~\eqref{eq:rank-nullity-2}, we finally arrive at
\begin{align*}
\dim Z &\ge |S|\cdot \dim X_{B_r} - Kk\varepsilon|S| - 2\varepsilon\delta\frac{nd}{k}-d\\
&\ge \left(\dim X_{B_r} - Kk\varepsilon - \frac{2\varepsilon d}{k} - o(1)\right)\cdot |S|.
\end{align*}

This gives us a lower bound on the desired joint entropy:
\begin{align*}
    \shent\! \big( (\mu_n)_{\sigma_n^{B_r}(R)} \big)
        &\geq \shent\! \big( (\mu_n)_{\sigma_n^{B_r}(S)} \big) \\
        &= \log 2 \cdot \dim Z \\
        &\geq \left( \log2 \cdot \dim X_{B_r} - \varepsilon \log 2 \cdot \Big(Kk - \frac{2 d}{k}-o(1)\Big) \right)\cdot (1-2\varepsilon)|R| \\
        &\geq \left( \shent(m_{B_r}) - \varepsilon C \right)|R|
\end{align*}
where $C = \log 2 \cdot (Kk - \frac{2 d}{k}) + 2\shent(m_{B_r})$.
Since $C$ depends only on $d,k$ and $r$, we could have taken at the beginning $\varepsilon = \eta/C$, so this proves~\eqref{eq:high-ent}.
\end{proof}

\begin{remark}
Proposition~\ref{prop:few-extra-cancellations} shows that, with high conditional probability in the choice of $\mathbf{H}$, the only parity checks among the bits in $W$ that are created by $\mathbf{H}$ are (i) those created by the rows in $F$, and possibly (ii) a few others that are generated by some vectors $\mathbf{y}$ whose support is either extremely small or extremely close to the whole of $E\setminus F$.  Since the number of possible vectors of type (ii) is very small compared with those of type (i), this implies that there are few enough `spurious' parity checks among the bits in $W$ to give Theorem~\ref{thm:umm}.  However, we do not expect that there are \emph{no} extra parity checks of type (ii): just by chance, one should typically find as many as a very small multiple of $n$ of these. 
\end{remark}

\begin{proof}[Proof of Theorem~\ref{thm:umm} item (1) from Theorem~\ref{thm:propertyM}]

We have to prove there are subsets
    \[\Omega'_n \subseteq \Hom_{\mathrm{unif}}(\Gamma,\Sym(V_n))\]
    such that $\PP_n(\Omega'_n) \to 1$ and if $\sigma_n\in \Omega_n'$ for each $n$, then $(\mu_n)_n$ has property M with respect to $\Sigma=(\sigma_n)_n$. This means that for every $\eps>0$ and $0<r<\infty$ there is a sequence of subsets $S_n \subset V_n$ such that 
$$\liminf \frac{|S_n|}{|V_n|} >0$$
and
\begin{equation}\label{eq:high-ent2}
\shent( (\mu_n)_{\sigma_n^{B_r}(S_n)} ) \ge |S_n| (\shent(m_{B_r}) - \eps)
\end{equation}
for all $n$.

    Let $r,\varepsilon>0$. By Theorem~\ref{thm:propertyM}, if $\delta>0$ is sufficiently small and for each $n$, $S_n \subset [n]$ is an arbitrary set of size $\lceil \delta n \rceil$, then there is some sequence $(\Omega'_n)_n$ with $\PP(\Omega'_n) \to 1$ such that the desired entropy inequality holds when $\sigma_n \in \Omega'_n$, and $\liminf \frac{\abs{S_n}}{n} = \liminf \frac{\lceil \delta n \rceil}{n} = \delta > 0$.

    Then, by diagonalizing over the countably many choices of $r \in \NN$ and $\varepsilon \in \{\frac{1}{2},\frac{1}{3},\frac{1}{4},\ldots\}$, we can get a single sequence of sets $(\Omega'_n)_n$ with $\PP(\Omega'_n) \to 1$ such that for any $r \in \NN$ and $\varepsilon>0$ there is a sequence $(S_n)_n$ with the desired properties.
\end{proof}

\subsection{Proof that there are few additional checks}\label{subs:few-extra-cancellations}

This subsection proves Proposition~\ref{prop:few-extra-cancellations}.  Let the sets $F_i$, $W_i$ and $M$ and parameters $K$, $\varepsilon$, $w_i$ and $w$ be as in that statement. For $\mathbf{y} \in \ZZ_2^{E\setminus F}$, let
\[G_{\mathbf{y}} := \left\{ (\sigma, H) \st (\mathbf{H}^{\mathsf{T}}\mathbf{y})_{V\setminus W} = \boldsymbol{0}\right\} \subset \tilde{\Omega}_n.\]
Write $\mathbf{y}=(\mathbf{y}_1, \ldots, \mathbf{y}_d)^\mathsf{T}$ with $\mathbf{y}_i \in  \ZZ_2^{E_i\setminus F_i}$. We will focus on those vectors $\mathbf{y}$ which satisfy restrictions on the cardinalities $|\mathbf{y}_i|$ coming from Proposition~\ref{prop:few-extra-cancellations}. To be precise, let $\mathcal{R}=\mathcal{R}(\varepsilon,\delta)$ be the set of all non-negative integer tuples $\mathbf{\ell} = (\ell_1,\dots,\ell_d)$ that satisfy
\begin{equation}\label{eq:allowed-ys}
\ell_i \le \frac{n-w_i}{k} \quad \hbox{for every} \ i \quad \hbox{and} \quad \min\left\{\ell_i,\frac{n-w_i}{k} - \ell_i\right\} \ge \varepsilon\delta\frac{n}{k} \quad \hbox{for some}\ i.
\end{equation}
The main conclusion of Proposition~\ref{prop:few-extra-cancellations} is equivalent to:
    \[\tilde{\PP}^M\left(\bigcup_{\mathbf{y}:\ (|\mathbf{y}_1|,\dots,|\mathbf{y}_d|) \in \mathcal{R}}G_{\mathbf{y}}\right) \to 0\]
as $n\to\infty$. Our proof uses a simple union bound over $\mathbf{y}$. We will derive estimates on $\tilde{\PP}^M(G_{\mathbf{y}})$ that depend on $\varepsilon$ and $\delta$ for $\mathbf{y} \in\mathcal{R}(\varepsilon,\delta)$, and then use these to conclude Proposition~\ref{prop:few-extra-cancellations} provided $\delta$ is chosen correctly.

Fix a vector $\mathbf{y} \in \ZZ_2^{E\setminus F}$ with $(|\mathbf{y}_1|,\dots,|\mathbf{y}_d|) \in \mathcal{R}$, and set $r_i := k|\mathbf{y}_i|$ and $\mathbf{r} := (r_1,\dots,r_d)$.  Let $Y_i$ be the support of the random vector $\mathbf{H}^\mathsf{T}\mathbf{y}_i$ for each $i$. This means $\mathbf{H}^\mathsf{T}\mathbf{y}_i = \mathbf{e}_{Y_i}$.  These random sets are independent by the independence in Lemma~\ref{lem:cond-dist}.  Each $Y_i$ is uniformly random among all subsets of $V\setminus W_i$ of size $r_i$. 

Each $Y_i$ may have non-empty intersection with $W\setminus W_i$.  We bound $\tilde{\PP}^M(G_{\mathbf{y}})$ by breaking into a few further cases depending on the sizes of these intersections.  To this end, let $Z_i := Y_i\setminus W$.  Under $\tilde{\PP}^M$, the cardinality $|Z_i|$ is a random quantity obtained by sampling $r_i$ points from $V\setminus W_i$ without replacement and counting how many of them land in $V\setminus W$.  This random quantity has the hypergeometric distribution with parameters $n-w_i$, $n-w$ and $r_i$: see, for instance,~\cite[Section II.6]{fellerVolI}. Its possible values are the integers $s_i$ that satisfy
\begin{equation}\label{eq:bounds-on-s}
\max(0,r_i-(w-w_i)) \le s_i \le \min(r_i,n-w),
\end{equation}
and for such values we have
\[\tilde{\PP}^M(|Z_i| = s_i) = \frac{{n-w \choose s_i}{w-w_i \choose r_i-s_i}}{{n-w_i \choose r_i}}.\]
By the standard exponential estimate for binomial coefficients (see, for instance,~\cite[Example 11.1.3]{coverthomas2006}), this ratio is at most
\begin{equation}\label{eq:hypergeo}
\exp\left(\shent\left(\frac{s_i}{n-w}\right)(n-w) + \shent\left(\frac{r_i-s_i}{w-w_i}\right)(w-w_i)  - \shent\left(\frac{r_i}{n-w_i}\right)(n-w_i) + o(n)\right),
\end{equation}
where quality of the error term does not depend on any other parameters.

Towards Proposition~\ref{prop:few-extra-cancellations}, we estimate the probability of $G_{\mathbf{y}}$ after further conditioning on the tuple of cardinalities $|Z_i|$, and then combine this estimate with~\eqref{eq:hypergeo} using the law of total probability. That refined conditional probability estimate depends on the following lemma.

Let $\mathtt{Even}(d)$ be the subset of all strings in $\{0,1\}^d$ that have even weight.

\begin{lemma}\label{lem:sup-on-evens}
    Let $q_1$, \dots, $q_d$ be probability distributions on $\{0,1\}$, and assume that $q$ is a coupling of $q_1$, \dots, $q_d$ that is supported on ${\mathtt{Even}}(d)$.  Then
    \[\shent(q) \le \left(1 - \frac{1}{d}\right)\big(\shent(q_1) + \dots + \shent(q_d)\big).\]
\end{lemma}

\begin{proof}
    By permuting indices, we may assume without loss of generality that
    \[\shent(q_1) \ge \dots \ge \shent(q_d).\]
    Let $\mathbf{z} = (z_1,\dots,z_d)$ be the identity map on $\{0,1\}^d$, and regard it as a random binary string with distribution $q$.  Then $\mathbf{z}$ has even weight almost surely, and hence the coordinates $z_2$, \dots, $z_d$ determine $z_1$ almost surely.  Therefore
    \[\shent(q) = \shent_q(\mathbf{z}) = \shent_q(z_2,\dots,z_d) \le \shent_q(z_2) + \cdots + \shent_q(z_d) = \shent(q_2) + \cdots + \shent(q_d).\]
    Because of our ordering of the indices, this is at most the desired upper bound.
\end{proof}

Now let $\mathcal{S}(\mathbf{r})$ be the set of all integer tuples $\mathbf{s} = (s_1,\dots,s_d)$ that satisfy~\eqref{eq:bounds-on-s} for every $i$.

\begin{lemma}\label{lem:unbiased-estimate}
 Fix a vector $\mathbf{y} \in \ZZ_2^{E\setminus F}$ with $(|\mathbf{y}_1|,\dots,|\mathbf{y}_d|) \in \mathcal{R}$ as above.   If $\mathbf{s} \in \mathcal{S}(\mathbf{r})$, then
    \[\tilde{\PP}^M\left(G_{\mathbf{y}}\ \big|\ |Z_i|=s_i\ \hbox{for each}\ i\right) \le \exp\left(-\frac{1}{d}\sum_{i=1}^d\shent\left(\frac{s_i}{n-w}\right)\cdot (n-w) + o_d(n)\right).\]
\end{lemma}

\begin{proof}
    To lighten notation, within this proof let
    \[\tilde{\PP}^{M,\mathbf{s}} := \tilde{\PP}^M\big(\,\cdot\,\big|\  |Z_i|=s_i\ \hbox{for each}\ i\big).\]

    Record the random sets $Z_1$, \dots, $Z_d$ into the random vector
    \[\boldsymbol{\eta} = (\eta_v)_{v\in V\setminus W} \in (\{0,1\}^d)^{V\setminus W} \quad \hbox{where}\ \eta_v := (1_{Z_1}(v),\dots,1_{Z_d}(v)).\]
    Let $P_{\boldsymbol{\eta}} = \frac{1}{\abs{V \setminus W}} \sum_{v \in V \setminus W} \delta_{\eta_v}$ be the empirical distribution of $\boldsymbol{\eta}$.  This is a probability distribution on $\{0,1\}^d$, and the event $G_{\mathbf{y}}$ occurs if and only if this probability distribution is supported on the subset $\mathtt{Even}(d)$.  The marginals of $P_{\boldsymbol{\eta}}$ are $(q_i,1-q_i)$ for $i=1,2,\dots,d$, where $q_i := s_i/(n-w)$. 
    Moreover, $P_{\boldsymbol{\eta}}$ must take values that are multiples of $1/(n-w)$, and the total number of such possible distributions is at most
\[(n-w+1)^{2^d} \le n^{2^d}= e^{o_d(n)}\]
(here and in some subsequent steps we generally loosen $o_d(n-w)$ to $o_d(n)$). Therefore
\begin{multline}\label{eq:P-good}
\tilde{\PP}^{M,\mathbf{s}}(G_{\mathbf{y}}) \le e^{o_d(n)}\max\big\{\tilde{\PP}^{M,\mathbf{s}}(P_{\boldsymbol{\eta}} = q):\ q\ \hbox{a coupling}\\ \hbox{of}\ q_1,\dots,q_d\ \hbox{such that}\ q(\mathtt{Even}(d)) = 1\big\}.
\end{multline}
For a distribution $q$ as above, the set of vectors $\boldsymbol{\eta}$ that give $P_{\boldsymbol{\eta}} = q$ are the `type class' of $q$, and their number is simply bounded using the entropy of $q$:
\[|\{\boldsymbol{\eta} \in (\{0,1\}^d)^{V\setminus W}:\ P_{\boldsymbol{\eta}} = q\}| \le e^{\shent(q)\cdot (n-w)}\]
(see, for instance,~\cite[Theorem 11.1.3]{coverthomas2006}, except note that Cover and Thomas use $\log_2$ rather than natural logarithms to define $\shent$). By Lemma~\ref{lem:sup-on-evens}, this upper bound is always at most
\begin{equation}\label{eq:card-even}
\exp\left(\left(1 - \frac{1}{d}\right)\cdot \big(\shent(q_1) + \cdots + \shent(q_d)\big)\cdot (n-w)\right).   \end{equation}

On the other hand, under the conditional probability measure $\tilde{\PP}^{M,\mathbf{s}}$, the set $Z_i$ is a uniform random subset of $V\setminus W$ of size $s_i$, and these random sets are still independent.  Therefore the probability of any particular $d$-tuple of sets of these sizes occurring is
\[\prod_{i=1}^d{n-w \choose s_i}^{-1},\]
and by another use of standard exponential estimates on binomial coefficients~\cite[Example 11.1.3]{coverthomas2006} this is at most
\[\exp\left(-\big(\shent(q_1) + \cdots + \shent(q_d)\big)\cdot (n-w) + o_d(n)\right).\]

Multiplying by the cardinality upper bound~\eqref{eq:card-even}, we obtain
\[\tilde{\PP}^{M,\mathbf{s}}(P_{\boldsymbol{\eta}} = q) \le \exp\left(-\frac{1}{d}\big(\shent(q_1) + \cdots + \shent(q_d)\big)\cdot (n-w) + o_d(n)\right)\]
for any such coupling $q$.  Since this upper bound is independent of the particular coupling $q$, and the extra factor in~\eqref{eq:P-good} is sub-exponential, this gives the result.
\end{proof}


\begin{lemma}\label{lem:LOTP}
    Fix $k$ and $d$ as before, and define the function
    \[f(t,\alpha',\alpha'') := (1 - k^{-1})\shent\big((1-t)\alpha' + t\alpha''\big) - (1-t)(1 - d^{-1})\shent(\alpha') - t\shent(\alpha'')\]
    for $0 \le \alpha',\alpha'',t \le 1$. Then
    \[\tilde{\PP}^M\left(\bigcup_{\mathbf{y}:\ (|\mathbf{y}_1|,\dots,|\mathbf{y}_d|) \in \mathcal{R}}G_{\mathbf{y}}\right) \le \sum_{\mathbf{r} \in k\mathcal{R},\ \mathbf{s} \in \mathcal{S}(\mathbf{r})}\exp \left(-\sum_{i=1}^df_i(r_i,s_i)\cdot (n-w_i) + o_d(n)\right),\]
    where
    \[f_i(r_i,s_i) = f\left(\frac{w-w_i}{n-w_i},\frac{s_i}{n-w},\frac{r_i-s_i}{w-w_i}\right).\]
    \end{lemma}

\begin{proof}
For each $\mathbf{y}$, we let $r_i := k|\mathbf{y}_i|$ and bound $\tilde{\PP}^M(G_{\mathbf{y}})$ from above using~\eqref{eq:hypergeo}, Lemma~\ref{lem:unbiased-estimate} and the law of total probability.  The resulting upper bound is
\begin{multline*}
  \tilde{\PP}^M\left(G_{\mathbf{y}}\right) \le  \sum_{\mathbf{s} \in \mathcal{S}(\mathbf{r})}\exp\Bigg(-\frac{1}{d}\sum_{i=1}^d\shent\left(\frac{s_i}{n-w}\right) (n-w) + \sum_{i=1}^d\shent\left(\frac{s_i}{n-w}\right)(n-w) \\
    + \sum_{i=1}^d\shent\left(\frac{r_i-s_i}{w-w_i}\right)(w-w_i) - \sum_{i=1}^d\shent\left(\frac{r_i}{n-w_i}\right)(n-w_i) + o_d(n)\Bigg).
\end{multline*}

On the other hand, the number of vectors $\mathbf{y} \in \ZZ_2^{E\setminus F}$ with given weights $|\mathbf{y}_i| = r_i/k$ is at most
\[\exp\left(\sum_{i=1}^d\shent\left(\frac{r_i}{n-w_i}\right)\frac{n-w_i}{k}\right).\]
Therefore the sum of $\tilde{\PP}^M(G_{\mathbf{y}})$ over all $\mathbf{y}$ satisfying $(|\mathbf{y}_1|,\dots,|\mathbf{y}_d|) \in \mathcal{R}$ is at most
\begin{equation}\label{eq:union-bound}
\sum_{\mathbf{r} \in k\mathcal{R},\ \mathbf{s} \in \mathcal{S}(\mathbf{r})}\exp \left(-\sum_{i=1}^df_i(r_i,s_i)\cdot (n-w_i) + o_d(n)\right),
\end{equation}
where
\begin{align*}
f_i(r_i,s_i) &:= -\frac{1}{k}\shent\left(\frac{r_i}{n-w_i}\right) + \frac{1}{d}\cdot \frac{n-w}{n-w_i}\cdot \shent\left(\frac{s_i}{n-w}\right) \\
    &\qquad -\frac{n-w}{n-w_i}\cdot \shent\left(\frac{s_i}{n-w}\right) - \frac{w-w_i}{n-w_i}\cdot \shent\left(\frac{r_i-s_i}{w-w_i}\right) + \shent\left(\frac{r_i}{n-w_i}\right)\\
    &= \left(1 - \frac{1}{k}\right)\cdot \shent\left(\frac{r_i}{n-w_i}\right)\\
    &\qquad -\left(1-\frac{1}{d}\right)\cdot \frac{n-w}{n-w_i}\cdot \shent\left(\frac{s_i}{n-w}\right) - \frac{w-w_i}{n-w_i}\cdot \shent\left(\frac{r_i-s_i}{w-w_i}\right)\\
    &= f\left(\frac{w-w_i}{n-w_i},\frac{s_i}{n-w},\frac{r_i-s_i}{w-w_i}\right). \qedhere
\end{align*}
\end{proof}

We are nearly ready to prove Proposition~\ref{prop:few-extra-cancellations}.  For that proof, we must combine Lemma~\ref{lem:LOTP} with an elementary but rather fiddly estimate.  That estimate refers to the functions
\[\gamma_1(t) := \frac{1}{(\log(1/t))^{1/3}} \quad \hbox{and} \quad \gamma_2(t) := \frac{1}{(\log(1/t))^{2/3}},\]
both for $0 < t < 1$. The exponents $1/3$ and $2/3$ are not particularly special here: all we really need is the ordering $0 < 1/3 < 2/3 < 1$. The next lemma gives a collection of simple bounds on the quantity $f(t,\alpha',\alpha'')$ for different ranges of the arguments.  Each part requires that $t$ is sufficiently small in terms of $d$ and $k$. The quantities $t_0^{\mathrm{(a)}}, t_0^{\mathrm{(b)}}$, $t_0^{\mathrm{(c)}}$ and  $t_0^{\mathrm{(d)}}$ are unspecified positive numbers that are sufficiently small in terms of only $d$ and $k$. Recall we assume $k > d \geq 3$.

\begin{lemma}\label{lem:precalc}
Write $\alpha := (1-t)\alpha' + t\alpha''$. The function $f$ from Lemma~\ref{lem:LOTP} satisfies the following.
\begin{enumerate}
\item[a.] If $t < t_0^{\mathrm{(a)}}$ and either $\alpha'' \le \alpha' \le 1/2$ or $\alpha'' \ge \alpha' > 1/2$, then
\[f(t,\alpha',\alpha'') \gtrsim_{d,k}
\shent(\alpha).\]
(This includes the assertion that the left-hand side is non-negative; see Section~\ref{sec:notation}.)
\item[b.] If $t < t_0^{\mathrm{(b)}}$ and $t\gamma_p(t) \le \alpha' \le 1 - t\gamma_p(t)$, then
\[f(t,\alpha',\alpha'') \gtrsim_{d,k} t\cdot (\log(1/t))^{1-p/3}.\]
\item[c.] If $t < t_0^{\mathrm{(c)}}$ and $\gamma_p(t) \le \alpha'' \le 1 - \gamma_p(t)$, then
\[f(t,\alpha',\alpha'') \gtrsim_{d,k} t\cdot (\log(1/t))^{1-p/3}.\]
\item[d.] If $t < t_0^{\mathrm{(d)}}$ and
\[ [\ \alpha' < t\gamma_p(t)\ \hbox{or}\ \alpha' > 1-t\gamma_p(t) ]\ \hbox{and}\ [\ \alpha'' < \gamma_p(t)\ \hbox{or}\ \alpha'' > 1 - \gamma_p(t)\ ],\]
then
\[ \max\{ 0, -f(t,\alpha',\alpha'')\} \lesssim_{d,k}  t\cdot (\log(1/t))^{1-p/3}.\]
(The maximum is used here to maintain the non-negativity convention for $\lesssim$.)
\end{enumerate}
\end{lemma}

Observe that at least one of the parts (b), (c) or (d) must hold whenever $t < \min\{ t_0^{\mathrm{(b})},t_0^{\mathrm{(c)}},t_0^{\mathrm{(d)}}\}$.


\begin{proof}
Each part of this lemma is symmetric under replacing $(\alpha',\alpha'')$ with ${(1-\alpha',1-\alpha'')}$, and so is the function $f$.  We therefore assume that $\alpha' \le 1/2$ throughout the proof.

\vspace{7pt}

\emph{Part (a).}\quad By the concavity of $\shent$, we have
\begin{align*}
    &(1-t)(1 - d^{-1})\shent(\alpha') + t\shent(\alpha'')  \\
    &\le \big((1-t)(1 - d^{-1}) + t\big)\shent\left(\frac{(1-t)(1 - d^{-1})}{(1-t)(1 - d^{-1}) + t}\alpha' + \frac{t}{(1-t)(1 - d^{-1}) + t}\alpha''\right).
    \end{align*}
    Since $1-d^{-1} < 1$, the convex combination inside the argument of $\shent$ here skews more towards $\alpha''$ than does the convex combination that gives $\alpha$.  Therefore, since $\alpha'' \le \alpha' \le 1/2$ and $\shent$ is increasing on $[0,1/2]$, the right-hand side above is bounded above by
    \[\big((1-t)(1 - d^{-1}) + t\big)\shent(\alpha) = (1 - d^{-1} + td^{-1})\shent(\alpha).\]
Therefore
\[f(t,\alpha',\alpha'') \ge \big((1-k^{-1}) - (1-d^{-1} + td^{-1})\big)\shent(\alpha) = (d^{-1} - k^{-1} - td^{-1})\shent(\alpha),\]
which is $\gtrsim_{d,k} \shent(\alpha)$ provided $t < t_0^{\mathrm{(a)}}$.

\vspace{7pt}

\emph{Part (b).}\quad For this part our assumptions are now $\alpha' \le 1/2$ and $\alpha' \ge t\gamma_p(t)$. Let $c_1 := (d^{-1} - k^{-1})/2 > 0$. Since $\alpha'\le 1/2$ and $\shent$ is continuous on $[0,1]$ and increasing on $[0,1/2]$, we have that
\begin{equation}\label{eq:step3.1}
\shent((1-t)\alpha' + t\alpha'') \ge \frac{1-d^{-1} + c_1}{1-k^{-1}}\shent(\alpha') \quad \hbox{whenever}\ t < t_0^{\mathrm{(b)}}.
\end{equation}

If $\alpha' \ge t\gamma_p(t)$ and $t < t_0^{\mathrm{(b)}}$, then
\begin{eqnarray}
\shent(\alpha') &\ge& \shent(t\gamma_p(t)) \ge t\cdot \gamma_p(t) \cdot \big(\log(1/\gamma_p(t)) + \log(1/t)\big) \label{eq:step3.2} \\
&\ge& t\cdot \gamma_p(t) \cdot \log(1/t) = t \cdot (\log(1/t))^{1-p/3}. \nonumber
\end{eqnarray}
Combining~\eqref{eq:step3.1} and~\eqref{eq:step3.2}, we obtain
\begin{align*}
f(t,\alpha',\alpha'') &\ge (1 - d^{-1} + c_1)\shent(\alpha') - (1-t)(1-d^{-1})\shent(\alpha') - t\shent(\alpha'')\\
&\ge c_1\shent(\alpha') - t\shent(\alpha'')\\
&\ge c_1\cdot t\cdot (\log(1/t))^{1-p/3} - \log 2\cdot t\\
&\gtrsim_{d,k} t\cdot (\log(1/t))^{1-p/3} \qquad \qquad \qquad \hbox{if}\ t < t_0^{\mathrm{(b)}}.
\end{align*}

\vspace{7pt}

\emph{Part (c).}\quad For this part our assumptions are now $\alpha' \le 1/2$ and $\alpha'' \ge \gamma_p(t)$.  We may also assume that $\alpha' < t\gamma_p(t)$, for otherwise part (b) already gives the desired bound.

If, in addition, we have $\alpha' \ge \alpha''$, then part (a) gives
\[f(t,\alpha',\alpha'')\gtrsim_{d,k} \shent(\alpha),\]
and this in turn satisfies
\begin{equation}\label{eq:Halpha''}
\shent(\alpha) \ge \shent(t\alpha'') \ge t\cdot \alpha''\cdot \log(1/t) \ge t\cdot \gamma_p(t) \cdot \log(1/t) = t\cdot (\log(1/t))^{1-p/3}.
\end{equation}

So for the rest of this part assume in addition that $\alpha' \le \alpha''$.  Then for sufficiently small $t$ we must have the ordering $\alpha' \le \alpha < 1/2$, and so
\begin{align*}
    (1-t)(1-d^{-1})\shent(\alpha') + t\shent(\alpha'') &\le (1-d^{-1})\shent(\alpha') + t\shent(\alpha'')\\
    &\le (1-d^{-1})\shent(\alpha) + t\shent(\alpha'').
\end{align*}
Therefore in this case it suffices to show that
\[(d^{-1} - k^{-1})\shent(\alpha) - t\shent(\alpha'') \gtrsim_{d,k} t\cdot (\log(1/t))^{1-p/3}.\]
Since the second left-hand term here is $O(t)$, this follows by another use of~\eqref{eq:Halpha''}.

\vspace{7pt}

\emph{Part (d).}\quad For this case we simply neglect the positive term in $f$ entirely.  If
\[t < t_0^{\mathrm{(d)}}, \quad \alpha' < t\cdot \gamma_p(t) \quad \hbox{and}\quad \alpha'' < \gamma_p(t),\]
then
\[\shent(\alpha') \lesssim_{d,k} t\cdot \gamma_p(t)\cdot \log\frac{1}{t} = t\cdot (\log(1/t))^{1-p/3}\]
and
\[t \shent(\alpha'') = O(t) \lesssim_{d,k} t\cdot (\log(1/t))^{1-p/3}.\]
In the cases where $\alpha' > 1 - t \cdot \gamma_p(t)$ or $\alpha'' > 1 - \gamma_p(t)$ the same estimates hold, by the symmetry $\shent(x) = \shent(1-x)$. Adding these estimates gives the conclusion.
\end{proof}

\begin{cor}\label{cor:precalc}
Fix $K\ge 1$ and $\delta > 0$, and let the other notation be as for Lemma~\ref{lem:precalc}.  If $\delta$ is sufficiently small in terms of $d$, $k$ and $K$, and if
\[t \le K\delta \quad \hbox{and} \quad \delta\cdot \gamma_1(\delta) \le \alpha \le 1 -  \delta\cdot \gamma_1(\delta),\]
then
\[f(t,\alpha',\alpha'') \gtrsim_{d,k,K} \delta \cdot (\log(1/\delta))^{2/3}\]
(irrespective of any further bounds on $\alpha'$ and $\alpha''$).
\end{cor}

\begin{proof}
By the same symmetry as for Lemma~\ref{lem:precalc}, we may assume that $\alpha' \le 1/2$.

Having done so, suppose first that $\alpha'' \le \alpha'$.  Then part (a) of Lemma~\ref{lem:precalc} gives
\[f(t,\alpha',\alpha'') \gtrsim_{d,k} \shent(\alpha),\]
and our assumed range for $\alpha$ gives
\begin{equation}\label{eq:Halpha-below}
\shent(\alpha) \ge\shent(\delta\cdot \gamma_1(\delta)) \ge \delta\cdot (\log(1/\delta))^{2/3},
\end{equation}
giving a lower bound of the desired form.

So now suppose that $\alpha' \le 1/2$ and $\alpha'' \ge \alpha'$.  Since $t \le K\delta$, it follows that
\[\alpha' \le \alpha\le (1-K\delta)\alpha' + K\delta.\]
Provided $\delta$ is sufficiently small in terms of $d$ and $k$, this range of possible values for $\alpha'$ and $\alpha$ implies that
\[(1-d^{-1})\shent(\alpha') \le \left(1-\frac{d^{-1}+k^{-1}}{2}\right)\shent(\alpha)\]
(noting that the constant in front of the entropy on the left is slightly smaller than the constant in front of the entropy on the right). Re-arranging, this implies that
\[(1-k^{-1})\shent(\alpha) - (1-d^{-1})\shent(\alpha') \ge \frac{d^{-1}-k^{-1}}{2}\shent(\alpha) \ge \frac{d^{-1}-k^{-1}}{2}\cdot \delta\cdot (\log(1/\delta))^{2/3},\]
using again the lower bound~\eqref{eq:Halpha-below}.  This now gives
\begin{align*}
f(t,\alpha',\alpha'') &\ge \frac{d^{-1}-k^{-1}}{2}\cdot \delta\cdot (\log(1/\delta))^{2/3} - t\big(\shent(\alpha'') - (1-d^{-1})\shent(\alpha')\big)\\
&\ge \frac{d^{-1}-k^{-1}}{2}\cdot \delta\cdot (\log(1/\delta))^{2/3} - \log 2\cdot K\cdot \delta.
\end{align*}
This implies the desired lower bound on $f$ for all sufficiently small $\delta$.
\end{proof}

\begin{proof}[Proof of Proposition~\ref{prop:few-extra-cancellations}]
Fix $K$ and $\varepsilon$ and also $F_i$, $W_i$, $w_i$ and $M$ as in the statement of Proposition~\ref{prop:few-extra-cancellations}, and suppose that~\eqref{eq:small-enough} holds.  We prove the convergence to zero of the required probabilities provided that $\delta$ is small enough in terms of $d$, $k$, $\varepsilon$ and $K$.

First, assume $\delta$ is small enough that
\begin{equation}\label{eq:delta-and-eps}
\gamma_1(\delta) \le \varepsilon,
\end{equation}
and also small enough in terms of $d$, $k$ and $K$ that Corollary~\ref{cor:precalc} applies.

By Lemma~\ref{lem:LOTP}, and since a small choice of $\delta$ ensures that $n-w_i \ge n/2$ for each $i$, it suffices to show that, if $\delta$ is sufficiently small, then the negative exponent
\begin{equation}\label{eq:exponent}
\sum_{i=1}^d f\left(\frac{w-w_i}{n-w_i},\frac{s_i}{n-w},\frac{r_i-s_i}{w-w_i}\right)    
\end{equation}
is bounded below by a positive quantity that is independent of $\mathbf{r} \in k\mathcal{R}$ and $\mathbf{s} \in \mathcal{S}(\mathbf{r})$.

So fix $\mathbf{r}$ and $\mathbf{s}$, and let
\[(t_i,\alpha'_i,\alpha''_i) := \left(\frac{w-w_i}{n-w_i},\frac{s_i}{n-w},\frac{r_i-s_i}{w-w_i}\right)\]
and
$$\alpha_i = (1-t_i)\alpha'_i + t_i \alpha''_i =\frac{r_i}{n-w_i}.$$
For each $i$ this implies that
\[t_i \le \frac{w}{n} \le K \delta.\]
Assume $\delta$ is also small enough that $K\delta < \min\{t_0^{\mathrm{(a)}}, t_0^{\mathrm{(b)}},t_0^{\mathrm{(c)}},t_0^{\mathrm{(d)}}\}$, so each $t_i$ is also less than this minimum.

Classify the indices $i \in [d]$ into two subsets:
\begin{align*}
I_1 &= \{i \in [d]:\ \varepsilon\delta \le r_i/(n-w_i) \le 1 - \varepsilon\delta\}\\
I_2 &= [d] \setminus I_1 .
\end{align*}
The definition of $\mathcal{R}$ implies that $I_1 \neq \emptyset$.  Consider the terms in~\eqref{eq:exponent} for indices in these subsets:
\begin{itemize}
 \item If $i \in I_1$, then, by our choice of $\delta$ in~\eqref{eq:delta-and-eps} and since $t_i \le K\delta$, Corollary~\ref{cor:precalc} gives
\[f\left(\frac{w-w_i}{n-w_i},\frac{s_i}{n-w},\frac{r_i-s_i}{w-w_i}\right) \ge C_1 \cdot \delta \cdot (\log(1/\delta))^{2/3}\]
for some positive constant $C_1$ depending only on $d$, 
$k$ and $K$.

\item Next, for $i\in I_2$, we use that at least one of parts (b), (c), and (d) of Lemma~\ref{lem:precalc} must hold, which gives that
\begin{align*}
    \max\left\{ 0, \, -f\left(\frac{w-w_i}{n-w_i},\frac{s_i}{n-w},\frac{r_i-s_i}{w-w_i}\right) \right\}
        &\lesssim_{d,k} t_i\cdot (\log(1/t_i))^{1/3}\\
        &\lesssim_K \delta\cdot (\log(1/\delta))^{1/3}
\end{align*}
provided $\delta$ is sufficiently small, and so the left-hand side is bounded above by $C_2\cdot \delta\cdot (\log(1/\delta))^{1/3}$ for some positive $C_2$ depending only on $d$, $k$ and $K$.
\end{itemize}

We bound~\eqref{eq:exponent} from below by adding these estimates.  At least one term has $i \in I_1$, and there are at most $d-1$ terms with $i \in I_2$.  This leaves the lower bound
\[C_1\cdot \delta \cdot (\log(1/\delta))^{2/3} - (d-1)\cdot C_2\cdot \delta\cdot (\log(1/\delta))^{1/3}.\]
This is positive for all sufficiently small $\delta$, uniformly over different choices of $\mathbf{r}$ or $\mathbf{s}$, as required.
\end{proof}

%
%

\section{Proof of local and empirical convergence}
\label{S:lde}

In this section we prove Theorem~\ref{thm:umm}(2): we show that if $(\Omega'_n)_n$ is the sequence given by Theorem~\ref{thm:umm}(1) and $\sigma_n \in \Omega'_n$ for each $n$, then the measures $\mu_n$ converge locally and empirically in probability to the Haar measure $m$.

Recall $X_n = X_{\sigma_n}$ is the set of parity check codewords over $\sigma_n \in \Hom_{\unif}(\Gamma,\Sym(V_n))$. For $v \in V_n$ let
    \[ X_{n,v} = \left\{ \Pi_v^{\sigma_n} \mb{x} \st \mb{x} \in X_n \right\} \]
be the set of pullback names at $v$ of all the codewords in $X_n$. As above, let
    \[ \Loc{\mu_{n}}{v} = \left( \Pi_v^{\sigma_n} \right)_* \mu_n \in \Prob(X) .\]

Let $X_{n,v,r}$ be the projection of $X_{n,v}$ onto $\ZZ_2^{B_r}$. Call a vertex $v \in V_n$ {\bf $r$-proper} if $X_{n,v,r} = X_{B_r}$. Otherwise, call it {\bf $r$-improper}.
 
The following lemma is a (stronger) version of \cite[Lemma 3.47]{richardson2008} for our random factor graph model.

\begin{lemma}
\label{lem:mostlyproper}
	For any $r, \varepsilon>0$
		 \[ \PP \left( \frac{1}{\abs{V_n}} \abs{\{ v \in V_n \st v \text{ is $r$-improper} \}} \geq \varepsilon \right) \to 0 \]
	as $n \to \infty$.
\end{lemma}
\begin{proof}
	In general $X_{n,v,r}$ is a vector subspace of $X_{B_r}$. If $v$ is $r$-improper, then it is a subspace of strictly smaller dimension, so
	\begin{equation}
    \label{eqn:Hsubspacebound}
	    \shent(\Loc{\mu_{n}}{v}_{B_r}) \leq \log \abs{X_{n,v,r}} = (\dim X_{n,v,r}) \log 2 \leq (\dim X_{B_r} -1)\log 2 = \shent(m_{B_r}) - \log 2 .
	\end{equation}
		
	With $\varepsilon,r$ as given, pick $\eta = \frac{\varepsilon \log 2}{4}$. By Theorem~\ref{thm:propertyM}, for small enough $\delta>0$, if for each $n$ we pick a subset $R_n$ of $V_n$ of size $\lceil \delta n \rceil$, then with high probability $\sigma_n$ satisfies
		\[ \shent( (\mu_n)_{\sigma_n^{B_r}\cdot R_n} ) \geq \left( \shent(m_{B_r}) - \eta \right) \abs{R_n} . \]
	
	
	Now for the sake of contradiction suppose that
        \[ \limsup_{n \to \infty} \PP \left( \frac{1}{\abs{V_n}} \abs{\{ v \in V_n \st v \text{ is $r$-improper} \}} \geq \varepsilon \right) > 0 . \]
    Then, by symmetry of the law of $\sigma_n$, and using that $\abs{R_n} \geq \delta \abs{V_n}$, the probability that the fraction of $r$-improper vertices within $R_n$ is at least $\frac{\varepsilon}{2}$ is uniformly bounded below for infinitely many $n$. But if an $\frac{\varepsilon}{2}$ fraction of vertices of $R_n$ are $r$-improper, then by~\eqref{eqn:Hsubspacebound} and subadditivity of Shannon entropy
		\[ \shent((\mu_n)_{\sigma_n^{B_r} \cdot R_n} ) \leq \left( \shent(m_{B_r}) - \tfrac{\varepsilon}{2} \log 2 \right) \abs{R_n}.\]
	Combining with the above, this implies that with nonvanishing probability
		\[\left( \shent(m_{B_r}) - \eta \right) \abs{R_n}\leq \left( \shent(m_{B_r}) - \tfrac{\varepsilon}{2} \log 2 \right) \abs{R_n}. \]
	But this is false by choice of $\eta$, so it must be that
        \[\limsup_{n \to \infty} \PP \left( \frac{1}{\abs{V_n}} \abs{\{ v \in V_n \st v \text{ is $r$-improper} \}} \geq \varepsilon \right) = 0\]
    as desired.
 \end{proof}

 \begin{lemma}\label{L:mixing}
  Assume $k\ge 3$. Then the action of $\G$ on $(X,m_X)$ is mixing and hence, ergodic.
 \end{lemma}

 \begin{proof}
 We prove this using Fourier analysis on the compact Abelian group $X$.
 
 \vspace{7pt}
 
 \emph{Step 1.}\quad For any $g \in \G$, let $|g|$ be its word length in the generating set $\{s_i^j:\ i = 1,\dots,d,\ j=1,\dots,k-1\}$.  The characters of $X$ form an orthonormal basis for $L^2_\CC(m_X)$, and all non-identity characters have mean zero. It therefore suffices to prove that, for any two non-identity characters $\chi_1$ and $\chi_2$, we have
\[\langle \chi_1,\chi_2\circ T^g\rangle = 0\]
whenever $|g|$ is sufficiently large.

\vspace{7pt}

\emph{Step 2.}\quad  Every character of $X$ is the restriction of a character of $\ZZ_2^\G$, and all of these have the form
\[\chi_U(\mathbf{x}) := (-1)^{\sum_{g \in U}x_g} \qquad (\mathbf{x} \in \ZZ_2^\G)\]
for some finite subset $U$ of $\G$.  In addition, the action $T$ is by group automorphisms, and satisfies $\chi_U\circ T^g = \chi_{gU}$.  Therefore $\chi_U\cdot (\chi_W\circ T^g) = \chi_{U\triangle gW}$ for any $U$, $W$ and $g$, where $\triangle$ denotes symmetric difference.

Let $E$ be the set of all $k$-cycles of the form $\{g,gs_i,\dots,gs_i^{k-1}\}$ in the Cayley graph of $\G$, where $g \in \G$ and $i \in \{1,\dots,d\}$.  As in the Introduction, we can regard $E$ as the set of $k$-hyper-edges of a hypergraph on $\Gamma$, and the corresponding characters $\chi_e$ for $e \in E$ give the parity checks that define the LDPC subgroup $X$.  Therefore, by Pontrjagin duality, a character $\chi_U$ restricts to the identity character on $X$ if and only if there is a finite subfamily $F$ of $E$ such that
\[1_U = \sum_{e \in F}1_E \mod 2.\]
Let us write $X^\perp$ for the collection of finite subsets $U$ that have this property. Regarded as a subspace of $\ZZ_2^{\oplus \G}$, this $X^\perp$ is the linear span of the set $\{1_e:\ e \in E\}$.

In these terms, we must show that, if $U$ and $W$ are finite subsets of $\G$ and neither of them lies in $X^\perp$, then $U\triangle gW$ also does not lie in $X^\perp$ whenever $|g|$ is sufficiently large.  Fix such $U$ and $W$ for the rest of the proof.

\vspace{7pt}

\emph{Step 3.}\quad Let $B_r$ and $B_s$ be closed balls around the identity in the right Cayley graph that contain $U$ and $W$, respectively.  More specifically, in this last step of the proof we assume that  $g \in \G$ satisfies both (i) $|g| > r+ s+ 2$ and also (ii) $U\triangle gW \in X^\perp$, and derive a contradiction from these. Assumption (i) implies $U$ and $gW$ are disjoint and therefore $U\triangle gW=U\cup gW$. Having made assumption (ii), let $F$ be a finite subfamily of $E$ such that
\begin{equation}\label{eq:UgW}
1_{U\triangle gW} = \sum_{e \in F}1_e \mod 2.
\end{equation}

To work with elements of $X^\perp$, it helps to introduce the dual graph $(E,\tilde{E}$) of the hypergraph $(\G,E)$.  The vertices of the dual graph are the hyperedges in $E$, and two hyperedges $e_1$ and $e_2$ are joined in $\tilde{E}$ if and only if $e_1\cap e_2 \ne \emptyset$.

In the dual graph, $F$ is the union of its connected components.  At least one of these must meet both $U$ and $gW$.  Indeed, otherwise we could let $G$ be the union of those connected components of $F$ that meet $U$, and would then find that $1_U$ agrees with $\sum_{e \in G}1_e$ mod $2$, contradicting our assumption that $U \not\in X^\perp$.

Next, because the hypergraph $(\G,E)$ is a hyper-tree and we have $|g| > r+s +2$, there is a single hyperedge $e_0$ lying `between' $B_r$ and $gB_s$ in the following sense: removing this hyperedge $e_0$ disconnects that hypergraph into $k$ components, one of which contains the whole of $B_r$ (and hence $U$), and a different one of which contains the whole of $gB_s$ (and hence $gW$).

As a result, if $F_0$ is a connected component of $F$ in the dual graph which meets both $U$ and $gW$, then $F_0$ must contain $e_0$.  Now consider again the $k$ connected components of the hypergraph $(\G,E\setminus e_0)$.  Since $k\ge 3$, at least one of them does not meet either $U$ or $gW$; let $V \subset \G$ be one such component. Then $V$ meets $e_0$ in a single vertex (that is, group element) $h$.

Finally, let $h'$ be a vertex of $V\cap \bigcup_{e\in F}e$ that lies at maximal distance from $h$.  This $h'$ may be equal to $h$, or it may lie `deeper inside' the component $V$.  However, by that distance maximality and the hyper-tree structure of $(\G,E)$, $h'$ can be contained in only one member of $F$, and because $h' \in V$ it cannot lie in either $U$ or $gW$. So at $h'$ the left-hand side of~\eqref{eq:UgW} must equal zero while the right-hand side must equal $1$. This is the desired contradiction with that equation.
 \end{proof}

\begin{proof}[Proof of Theorem~\ref{thm:umm}(2)]

     Let $\Loc{\mu_n}{v}_{B_r}$ denote the marginal of $\Loc{\mu_n}{v}$ on $X_{n,v,r}$.
     
	Since $\Loc{\mu_n}{v}_{B_r}$ is the uniform distribution on $X_{n,v,r}$ and $m_{B_r}$ is the uniform distribution on $X_{B_r}$, local convergence in probability to $\mu$ is implied by
		\[ \lim_{n\to\infty} \PP \left( \frac{1}{\abs{V_n}} \abs{ \left\{ v \in V_n \st X_{n,v,r} \ne X_{B_r} \right\}} > \varepsilon \right) = 0 , \]
    which is true by Lemma~\ref{lem:mostlyproper}.

    Since $(\mu_n)_n$ converges locally in probability to $m_X$, which is ergodic by Lemma \ref{L:mixing}, Lemma~\ref{lem:lepergodic} completes the proof.
\end{proof}

\section{Proof of Theorem~\ref{mainthm3}}
\label{sec:mainthm3pf}

Next we prove the upper bound in Theorem~\ref{mainthm3}, part (a). After this, part (b) of the Theorem~\ref{mainthm3} is an immediate consequence of Theorems~\ref{thm:umm} (items 1,2) and \ref{T:wmmcpe}. 

\subsection{Proof of the sofic entropy value}

\begin{proof}[Proof of Theorem~\ref{mainthm3}, part (a)]
Let $\sigma_n \sim \PP_n$ and let $\Omega_n' \subseteq \Omega_n^{\mathrm{sofic}}$ be as in Proposition \ref{P:upper-estimate}. Let $\Sigma=\{\sigma_n\}_{n=1}^\infty$ satisfy $\sigma_n \in \Omega'_{i_n}$ for some increasing sequence $(i_n)_n$ with $k \mid i_n$ for all $n$. It suffices to prove
  $$\h_\Sigma(X,m_X,T)= (1-d/k)\log 2.$$

By definition,
$$\kik(m_X) = (1-d) \shent( W_{m_X}(\cdot) ) + \frac{1}{k} \sum_{i \in [d]} \shent( W_{m_X}(\cdot; i) ).$$
The weight $W_{m_X}$ is uniform on $\ZZ_2$. So $ \shent( W_{m_X}(\cdot) )=\log(2)$. For each $i\in [d]$, the measure $W_{m_X}(\cdot; i)$ is supported on the subspace of $\ZZ_2^k$ which is the kernel of the homomorphism $(x_1,\ldots, x_k) \mapsto \sum_i x_i \in \ZZ_2$. This subspace has cardinality $2^{k-1}$. So
$$\shent( W_{m_X}(\cdot; i) ) \leq (k-1)\log(2).$$
Combined with the previous formula, we obtain
$$\kik(m_X) \leq (1 -d/k)\log(2).$$
The conclusion of Proposition \ref{P:upper-estimate} now implies the upper bound
  $$\h_\Sigma(X,m_X,T)\le (1-d/k)\log 2.$$


By Theorem \ref{thm:umm}(2), $\mu_n$ converges empirically to $m_X$. Therefore, if $\mathcal{O}$ is any open neighborhood of $m_X$ then
$$1=\lim_{n\to\infty} \mu_n(\Omega(\sigma_n, \mathcal{O})).$$
Since $\mu_n$ is the uniform measure on $X_{\sigma_n}$, this implies
  $$\h_\Sigma(X,m_X,T)\ge \lim_{n\to\infty}  \frac{1}{i_n} \log |X_{\sigma_n}|.$$

We may estimate $|X_{\sigma_n}|$ by simple dimension-counting. In fact, $X_{\sigma_n}$ is the kernel of a homomorphism from $\ZZ_2^{i_n}$ to $\ZZ_2^{di_n/k}$. So 
$$|X_{\sigma_n}| \ge 2^{(1-d/k)i_n}.$$
This gives the lower bound
  $$\h_\Sigma(X,m_X,T)\ge (1-d/k)\log 2.$$
\end{proof}

\section{Proof of Theorems~\ref{mainthm2} and~\ref{mainthm5}}
\label{S:mainthm2pf}

Theorem~\ref{mainthm2} is an immediate consequence of Theorem~\ref{thm:umm}(3) and  Corollary~\ref{cor:shatt-cons}(2).

\subsection{Expected number of low-density codewords}
In this section we derive an asymptotic formula for the expected number of low-density codewords. Versions of this appear in \cite[Formula (11.10)]{mezard2009} and \cite[Lemma 3.163]{richardson2008}. 
The authors of \cite{mezard2009} and \cite{richardson2008} use a configuration model instead of a permutation model. As discussed above, these models are not contiguous (see Appendix~\ref{sec:non-contig}). Moreover, we are able to show a stronger result than contiguity alone would imply.

For $\eta>0$ and $\sigma \in \Hom_{\unif}(\Gamma, \Sym(kn))$, let $X_{\sigma}^\eta$ denote the set of $\mb{x} \in (\ZZ_2)^{kn}$ for which the sum around all but at most a fraction $\eta$ of each of the $d$ hyper-edges types is even. More formally,
    \[ X_\sigma^\eta = \bigcap_{i \in [d]} \left\{\mb{x} \in (\ZZ_2)^{kn} \st \abs*{\left\{v \in [kn] \st \sum_{j=0}^{k-1} \mb{x}(\sigma_i^j v) = 0 \pmod{2} \right\}} > kn (1-\eta) \right\} .\]
We can think of these as ``approximate codewords.'' 

For $t \in [0,1]$ we define the upper exponential growth rate of the expected number of approximate codewords of density $t$ by
    \[ \codegrowth(t) = \inf_{\varepsilon, \eta>0}\limsup_{n \to \infty} \frac{1}{kn} \log \EE_{\sigma \sim \PP_{kn}} \abs*{\{\mb{x} \in X_\sigma^\eta \st \tfrac{1}{kn}\abs{\mb{x}} \in (t-\varepsilon,\, t+\varepsilon) \}} . \]

\begin{prop}
\label{prop:lowdensitygrowthrate}
	For any $d,k$,
		\[ \codegrowth(t) = \frac{1}{2} t \big( d \log (k-1)- d +2 +(d -2) \log (t) \big)+O(t^2) . \]
	In particular, for $k \geq 2$ this is negative for small $t>0$ if and only if $d > 2$.
\end{prop}

The ``$O(t^2)$'' term here is a power series convergent on some neighborhood of $t=0$ with lowest-order term $t^2$.

Figure \ref{fig:6uniform_approx} compares exact plots of $\codegrowth(t)$ (created using parametric plots in the parameter $s$ of Lemma~\ref{lem:codegrowth}) with plots of this approximation.
\begin{figure}
    \centering
    \includegraphics{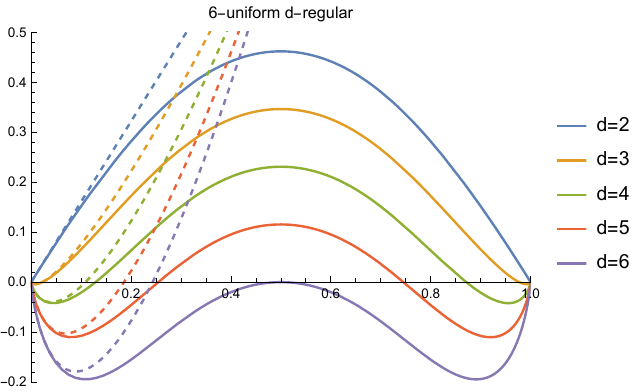}
    \caption{Comparison of $\codegrowth(t)$ (solid lines) with asymptotic in Proposition~\ref{prop:lowdensitygrowthrate} (dashed lines) for $k=6$ and several choices of $d$.}
    \label{fig:6uniform_approx}
\end{figure}

The proof of Proposition~\ref{prop:lowdensitygrowthrate} is based on the following lemma.

\begin{lemma}
\label{lem:codegrowth}
	For any $k,d$ and any $t \in [0,1]$
		\[ \codegrowth(t) = (1-d) \shent(t) + \frac{d}{k}\left( -kt \log s + \log Z \right) , \]
	where $s$ and $Z$ are related to $t$ via
		\[ t = s \frac{(1+s)^{k-1} - (1-s)^{k-1}}{(1+s)^k + (1-s)^k} \quad \text{and} \quad
			Z = \frac{1}{2} \left( (1+s)^k + (1-s)^k \right) . \]
\end{lemma}

 




\begin{proof}[Proof of Lemma~\ref{lem:codegrowth}]
Here, let $\A = \ZZ_2$. We also let $\mathcal{W}$ denote the set of weights which have edge weights that are cyclically invariant and supported on configurations with even parity. The terminology of weights  was defined in Section~\ref{sec:kikformula}.

We first show that, for every $t \in [0,1]$,
    \begin{equation}
    \label{eqn:codegrowth}
        \codegrowth(t) = \sup_{W \in \mathcal{W} \st W(1) = t}  \kik(W) .
    \end{equation}
    
First suppose $W \in \mathcal{W}$ has $W(1) = t$. For every $\varepsilon, \eta > 0$, if $\delta>0$ is small enough then we always have
    \[ \{ \mb{x} \in (\ZZ_2)^{kn} \st \norm{W_{\sigma, \mb{x}} - W} < \delta\} \subseteq \{ \mb{x} \in X_\sigma^\eta \st \tfrac{1}{kn} \abs{\mb{x}} \in (t - \varepsilon, t+ \varepsilon) \} .\]
By Proposition~\ref{prop:kikformula}, this implies that $\kik(W) \leq \codegrowth(t)$, which gives one half of Eqn.~\ref{eqn:codegrowth}.

For the converse inequality: 
Given $\varepsilon,\eta>0$ let $\mathcal{W}_{\varepsilon,\eta}$ be the set of cyclically-invariant weights with $\abs{W(1) - t} < \varepsilon$ and each $W(\cdot; i)$ giving mass at least $1 - \eta$ to even-parity configurations.
    
Given $\delta>0$, by Proposition~\ref{prop:kikformula} for each weight $W$ there is some $r_W > 0$ which satisfies
    \[\limsup_{n \to \infty} \frac{1}{kn} \log \EE_\sigma \abs{\{ \mb{x} \in (\ZZ_2)^{kn} \st \norm{W_{\sigma, \mb{x}} - W} < r_W \}} \leq \kik(W) + \delta . \]
By compactness, there is a finite set $\mathcal{S} \subset \overline{\mathcal{W}_{\varepsilon,\eta}}$ such that the balls centered at $W \in \mathcal{S}$ of radius $r_W$ cover $\mathcal{W}_{\varepsilon,\eta}$.
Then we have
    \[ \{ \mb{x} \in X_\sigma^\eta \st \tfrac{1}{kn}\abs{\mb{x}} \in (t-\varepsilon, t+\varepsilon) \} \subseteq \bigcup_{W \in \mathcal{S}} \{ \mb{x} \in (\ZZ_2)^{kn} \st \norm{W_{\sigma, \mb{x}} - W} < r_W \}. \]
Therefore,
    \begin{align*}
        \limsup_{n \to \infty} &\frac{1}{kn} \log \EE_\sigma \abs{\{ \mb{x} \in X_\sigma^\eta \st \tfrac{1}{kn}\abs{\mb{x}} \in (t-\varepsilon, t+\varepsilon) \}} \\
            &\leq \max_{W \in \overline{\mathcal{W}_{\varepsilon,\eta}}} \limsup_{n \to \infty} \frac{1}{kn} \log \EE_\sigma \abs{\{ \mb{x} \in (\ZZ_2)^{kn} \st \norm{W_{\sigma, \mb{x}} - W} < r_W \}} \\
            &\leq \max_{W \in \overline{\mathcal{W}_{\varepsilon,\eta}}} \kik(W) + \delta.
    \end{align*}
Taking $\varepsilon,\eta$ and then $\delta$ to 0 gives the other half of Eqn.~\ref{eqn:codegrowth}.

Now any weight achieving the supremum in \ref{eqn:codegrowth} must have all edge weights equal: to maximize $\kik$ under the constraint $W(1) = t$, whichever $W(\cdot; i)$ maximizes the edge term in the definition of $\kik(W)$ should be used for all $i \in [d]$.

We can specify such a weight by a single probability vector $\mb{p} \in \Prob(\A^{\ZZ_k}/\ZZ_k)$ recording the edge weight. We use a probability measure on $\A^{\ZZ_k}/\ZZ_k$ rather than $\A^{\ZZ_k}$ because the edge weights must be invariant under cyclic permutations. Let $W_{\mb{p}}$ be the weight with edge weights specified by $\mb{p}$, and write $\alpha(\mb{p}) = W_{\mb{p}}(1)$. The Kikuchi entropy of $W_{\mb{p}}$ is
	\[ \kik(W_{\mb{p}}) = (1-d) \shent(\alpha(\mb{p})) + \frac{d}{k}\left( \shent(\mb{p}) + \sum_{[\bta] \in \A^{\ZZ_k}/\ZZ_k} \mb{p}([\bta]) \log \abs{[\bta]} \right) \]
where $\abs{[\bta]}$ is the number of elements of the equivalence class $[\bta] \in \A^{\ZZ_k}/\ZZ_k$.

Now, we are interested in estimating
	\[ \codegrowth(t) = \max_{\mb{p} \st \alpha(\mb{p})=t} \kik(W_{\mb{p}})   , \]
where the maximum is also constrained to $\mb{p}$ supported on equivalence classes of even-parity configurations. Since $\alpha(\mb{p})$ is fixed to be $t$, we really just need to maximize
	\[ \shent(\mb{p}) + \sum_{[\bta] \in \A^{\ZZ_k}/\ZZ_k} \mb{p}([\bta]) \log \abs{[\bta]} \]
subject to the constraints
	\[ \sum_{[\bta]} \mb{p}([\bta]) \alpha([\bta]) = kt \quad \text{and} \quad \sum_{[\bta]} \mb{p}([\bta]) = 1 \]
 where $\alpha([\bta])$ is the number of $1$'s in any representative of $[\bta]$.
So we get two Lagrange multipliers $\lambda_1, \lambda_2$ so that a maximizer on the interior of the constraint region is given by
	\[ \mb{p}([\bta]) = \abs{[\bta]}\cdot e^{\lambda_1 \alpha([\bta]) + \lambda_2 - 1} \]
for $[\bta] \in \A^{\ZZ_k}/\ZZ_k$ with even parity and 0 otherwise. We rewrite this as
	\[ \mb{p}([\bta]) = \abs{[\bta]} \cdot \frac{s^{\alpha([\bta])}}{Z} \] 
where $Z$ is determined by the normalization constraint and $s$ is determined by the density-$t$ constraint. Since the objective function is strictly concave, the critical point given by these $Z,s$ is in fact the unique maximum. The corresponding $\kik$-maximizing weight $W$ therefore satisfies
	\[ W(\bta; i) = \frac{s^{\alpha(\bta)}}{Z} \]
if $\alpha(\bta)$ (the number of $1$'s in $\bta$) is even and $W(\bta; i) = 0$ otherwise.
	
Note that
\begin{equation}
	Z = \frac{1}{2} \left( (1+s)^k + (1-s)^k \right)
\end{equation}
since
	\[ Z = \sum_{\substack{\bta \in \A^k \\ \alpha(\bta) \text{ even}}} s^{\alpha(\bta)} = \sum_{\substack{m = 0 \\ m \text{ even}}}^k s^m \binom{k}{m} = \frac{1}{2} \left( \sum_{m=0}^k s^m \binom{k}{m} + \sum_{m=0}^k (-s)^m \binom{k}{m} \right). \]
The constraint $\sum_{[\bta]} \mb{p}([\bta])\, \alpha([\bta]) = kt$ gives the relation between $t$ and $s$:
	\[ kt = \sum_{[\bta] \text{ even}} \abs{[\bta]} \frac{s^{\alpha([\bta])}}{Z} \cdot \alpha([\bta]) = s \frac{\frac{dZ}{ds}}{Z} = s k \frac{(1+s)^{k-1} - (1-s)^{k-1}}{(1+s)^k + (1-s)^k}. \]
Finally, we can calculate 
	\[ \shent(W(\cdot; i)) = -\sum_{\bta \text{ even}} \frac{s^{\alpha(\bta)}}{Z}\alpha(\bta) \log s + \log Z = -kt \log s + \log Z \]
which gives the claimed formula.
\end{proof}

This lemma does not give an explicit formula for $\codegrowth(t)$: we have $t$ as a function of $s$ but we do not have an explicit formula for the inverse function. Still, we can use it to prove Proposition \ref{prop:lowdensitygrowthrate}.

\begin{proof}[Proof of Proposition~\ref{prop:lowdensitygrowthrate}]
	
	Expanding the formula for $t$ as a power series centered at $s=0$, we get $t = (k-1) s^2 + O(s^4)$ so, by Lagrange inversion \cite[\S 7.32]{MR4286926}, 
    there is a power series for $s^2$ on some interval around $0$ of the form
		\[ s^2 = \frac{t}{k-1} + O(t^2) . \]
	Hence
		\[ t \log s = \frac{1}{2} t \log \frac{t}{k-1} + O(t^2) \]
	and
		\[ \log Z = \frac{1}{2} k (k-1) s^2 + O(s^4) = \frac{1}{2} kt + O(t^2) . \]
	The estimate for Shannon entropy
		\[ \shent(t) = -t \log t + t + O(t^2) \]
	completes the proof.
\end{proof}

\subsection{Proof of Theorem~\ref{thm:umm}(3)}

Recall that $\mu$ has totally shattered microstate spaces along $(\sigma_n)_n$ if there exists a $\delta > 0$ for which the following holds: For every $\varepsilon > 0$ there exist a weak$^\ast$ neighbourhood $U$ of $\mu$ and a positive integer $n_0$ such that for any $n\ge n_0$ and any two microstates $\mathbf{x},\mathbf{y} \in \Omega(\sigma_n,U)$ we have either $\dee^{(V_n)}(\mathbf{x},\mathbf{y}) \ge \delta$ or $\dee^{(V_n)}(\mathbf{x},\mathbf{y}) < \varepsilon$.

\begin{proof}[Proof of Theorem~\ref{thm:umm}(3)]
    Given $k \geq 2$ and $d>2$ by Proposition~\ref{prop:lowdensitygrowthrate} there is some $\delta>0$ such that $\codegrowth(t) < 0$ for $t \in (0,\delta)$. We will use this $\delta$ to establish totally shattered microstates.

    For any fixed $\varepsilon>0$,
        \[ \inf_{\eta > 0} \limsup_{n \to \infty} \frac{1}{n} \log \EE | \{ \mb{x} \in X_{n}^\eta \st \tfrac{1}{n} |\mb{x}| \in [\varepsilon, \delta] \} | = \sup\{ \codegrowth(t) \st t \in (\varepsilon, \delta) \} < 0 . \]
    Pick some $\eta$ such that the expression in the infimum is negative, and let $U^\eta \subset \Prob(\ZZ_2^\Gamma)$ be the set of all probability measures whose marginal on every hyper-edge gives probability greater than $1 - \eta$ to labelings with even parity. Then $U^\eta$ is a weak$^\ast$ neighborhood of $m_X$ and
        \[ X_{n}^\eta = \Omega(\sigma_{n}, U^\eta) , \]
    so
        \[ \limsup_{n \to \infty} \frac{1}{n} \log \EE | \{ \mb{x} \in \Omega(\sigma_{n}, U^\eta) \st \tfrac{1}{n} |\mb{x}| \in [\varepsilon, \delta] \} | < 0 \]
    and there are subsets $\Omega_n^\varepsilon \subset \Omega_n^{\mathrm{sofic}}$ with $\PP_n(\Omega_n^\varepsilon) \to 1$ and such that for all large enough $n$ if $\sigma_n \in \Omega_n^\varepsilon$ then $\{ \mb{x} \in \Omega(\sigma_{n}, U^\eta) \st \tfrac{1}{n} |\mb{x}| \in [\varepsilon, \delta] \} = \varnothing$. Now if $\mb{x}, \mb{y} \in \Omega(\sigma_{n}, U^{\eta/2})$, then $\mb{x}+\mb{y} \in \Omega(\sigma_{n}, U^{\eta})$ so
        \[ \dee^{(V_n)}(\mb{x}, \mb{y}) = \tfrac{1}{n} |\mb{x} + \mb{y}| \notin [\varepsilon, \delta] . \]

    We can then get a single sequence $\Omega_n'$ that works for every $\varepsilon$ by picking one for each $\varepsilon = \frac{1}{2}, \frac{1}{3}, \ldots$ and then using a diagonal argument.
\end{proof}

\subsection{Proof of Theorem~\ref{mainthm5}}

Theorem~\ref{mainthm5} is an immediate consequence of Corollary~\ref{cor:shatt-cons}(3) and Theorem~\ref{thm:umm}(3).

\section{Directions for further study}

A probability measure-preserving action $\Gamma \cc (X,\mu)$ is {\bf anti-Pinkser} if it has positive entropy but does not have any nontrivial direct Bernoulli factors. We are being deliberately vague here by not specifying whether ``positive entropy'' refers to sofic, Rokhlin or some other notion of entropy. 

\subsection{Possible anti-Pinsker actions of other groups}
\begin{enumerate}
    \item 
Are there explicit anti-Pinsker actions of a free group? One candidate is the frozen model associated to independent sets~\cite{ding2016}.
    \item Do all non-amenable groups admit anti-Pinsker actions? Here it might be necessary to use Rokhlin rather than sofic entropy.
    
    \item Given a positive number $h$ and a non-amenable group $\Gamma$, does there exist an uncountable family of pairwise non-measurably conjugate ergodic pmp actions of $\Gamma$ which are anti-Pinsker, have completely positive sofic entropy  and have sofic entropy $h$ (with respect to some fixed sofic approximation)?  Starting from the example of the present paper, one place to look might be among its `typical' compact extensions.

\end{enumerate}




\subsection{Open problems for the parity-check subshift}

Let $(X,m_X,T)$ be the system in Theorem A.

\begin{enumerate}
\item Does there exist a sofic approximation $\Sigma$ to $\Gamma$ with respect to which $(X,m_X,T)$ does not have completely positive sofic entropy?

\item Is $(X,m_X,T)$ finitely determined? This would mean that, if $(\{0,1\}^\Gamma,\mu,T)$ is another system for which $\mu$ is close to $m_X$ both in sofic entropy and in the weak$^\ast$ topology, then there is a joining of these two systems under which the identity coordinates agree with high probability.  This property characterizes those processes isomorphic to Bernoulli shifts over amenable groups~\cite{OW87}, but little is known about it for non-amenable groups.

\item Formally, the family of equations that defines our LDPC shift $X$ can be used to define a system of algebraic origin inside $A^\Gamma$ for another compact Abelian group $A$, such as a finite cyclic group or the continuous circle $\RR/\ZZ$.  For which such $A$ (and which values of $d$ and $k$) is the resulting system still anti-Pinsker?  If there are such examples with $A = \RR/\ZZ$, do these have infinite sofic entropy?

\end{enumerate}
\appendix

\section{A failure of contiguity}\label{sec:non-contig}

Let $V$ be a vertex set of size $n$ divisible by $k$, let $E$ be a set of size $d n/k$, and let $\tilde\PP_n$ be the measure on $k$-uniform $d$-regular factor graphs on $(V,E)$ that is constructed in Section~\ref{subs:factor-graphs}. In addition, let $\PP_n^{\mathrm{unif}}$ be the uniform distribution on all such $k$-uniform $d$-regular factor graphs on $(V,E)$ produced without respect to a partition $E = \bigsqcup_{i\in [d]} E_i$.  


\begin{prop}
Let $k > d \ge 3$, and let
\[U_n := \left\{ H \subset V \times E :\ \exists E'\subseteq E\ \hbox{with every $v \in V$ adj.~to exactly two elements of } E' \right\}.\]
Then $\tilde\PP_n(U_n) = 1$ for all $n$, but $\PP_n^{\mathrm{unif}}(U_n)\to 0$ as $n\to\infty$.  As a result, the models $\tilde\PP_n$ and $\PP_n^{\mathrm{unif}}$ are not contiguous.
\end{prop}

Note that the definition of $U_n$ ensures that the contiguity also fails for the associated multi-hyper-graph models with unlabeled hyper-edges.

\begin{proof}
If $H$ arises from $\tilde\PP_n$, then let $i,j \in [d]$ be distinct and let $E' = E_i \cup E_j$.  Then every vertex is adjacent to exactly one check node in each of $E_i, E_j$, and in particular to exactly two check nodes in $E'$.

On the other hand, the very precise calculations  in~\cite{millercohen2003} show that, with high probability according to $\PP_n^{\mathrm{unif}}$, the transposed parity-check matrix associated to $H$ has kernel that is either trivial (if $d$ is odd) or one-dimensional (if $d$ is even, in which case the all 1's vector is in the kernel).  In either case, there can be no $E'$ as promised by the event $U_n$, since it would give an additional nontrivial element of the kernel.
\end{proof}


\bibliographystyle{plain}
\bibliography{refs}

\end{document}